\documentclass[10pt,oneside,a4paper,table]{article}

\usepackage{amscd,amssymb,amsmath,amsthm,mathrsfs,dsfont}
\usepackage[a4paper]{geometry}
\usepackage{tikz} 
\usetikzlibrary{positioning}
\usetikzlibrary{decorations.pathreplacing}
\usepackage[shortlabels]{enumitem}
\usepackage{comment}

\usepackage{caption}
\usepackage{array}
\usepackage{subcaption}
\usepackage[thinc]{esdiff}
 
\usepackage{hyperref}
\hypersetup{
    colorlinks=true,
    linkcolor=blue,
    filecolor=black,      
    urlcolor=blue,
    citecolor=blue,
}

\usepackage{mathtools}
\usepackage{mathrsfs}
\usepackage{lettrine}
\usepackage{mfirstuc}
\usepackage{multicol}
\usepackage{multirow}
\usepackage{makecell}
\usepackage{dirtytalk}
\usetikzlibrary{positioning}
\usepackage[font=small,labelfont=bf]{caption}

\newtheorem{theorem}{Theorem}
\newtheorem{proposition}{Proposition}
\newtheorem{lemma}{Lemma}
\newtheorem{remark}{Remark}
\newtheorem{definition}{Definition}
\def\N{\mathbb{N}}
\def\Z{\mathbb{Z}}
\def\R{\mathbb{R}}

\def\E{\mathbb{E}}
\numberwithin{equation}{section}
\usepackage{fancyhdr}
\usepackage[backend=bibtex,style=ieee,giveninits=true,sorting=anyt]{biblatex}
\newcommand\shorttitle{Bernoulli field on the line under the removal of isolates}

\begin{document}

\hypersetup{linkcolor=black}
\hypersetup{urlcolor=black}
\title{Three views on the thinned Bernoulli field on the line}
\author{Christof Külske\footnotemark[1] \and Niklas Schubert\footnotemark[2]}
\date{\today}

\maketitle
\begin{abstract}
This paper investigates the thinned Bernoulli field (TBF) on the one-dimensional integer lattice, where isolated occupied sites are removed from a standard Bernoulli configuration with density $p$. Our present work complements previous findings 
in higher dimensions and on trees by focusing on the detailed behavior on the line, particularly as $p$ approaches $1$.

First  we show that
while the TBF on the line is always quasilocally Gibbs, it displays a growing sensitivity to boundary conditions as $p$ increases, indicating an incipient loss of quasilocality. We provide precise asymptotics for this phenomenon, which is an echo of non-quasilocality happening in higher dimensions. Second, we turn to the one-sided point of view and prove that the TBF is a g-measure in the sense of dynamical systems and ergodic theory. The corresponding g-function
  is quasilocal but 
  becomes long-range again for large $p$. 
From that we finally develop our third view, in which  
we provide a transparent construction of the process 
in terms of a driving Markov chain on the
integers
of generalized house of cards type, offering a novel perspective on the TBF.

\end{abstract}
\textbf{Keywords:} Thinned Bernoulli field; Bernoulli site percolation, transformed measure; Gibbs property;  quasilocality; g-measure.\newline
\noindent\textbf{MSC2020 subject classifications:} 82B26 (primary);
60K35 (secondary).
\footnotetext[1]{Ruhr-University Bochum, Germany}
\footnotetext[0]{E-Mail: \href{mailto:Christof.Kuelske@ruhr-uni-bochum.de}{Christof.Kuelske@ruhr-uni-bochum.de}; ORCID iD: \href{https://orcid.org/0000-0001-9975-8329}{0000-0001-9975-8329}}
\footnotetext[0]{\url{https://math.ruhr-uni-bochum.de/en/faculty/professorships/stochastics/group-kuelske/staff/christof-kuelske/}}
\footnotetext[2]{E-Mail: \href{mailto:Niklas.Schubert@ruhr-uni-bochum.de}{Niklas.Schubert@ruhr-uni-bochum.de}; ORCID iD: \href{https://orcid.org/0009-0000-8912-4701}{0009-0000-8912-4701}}
\footnotetext[0]{\url{https://math.ruhr-uni-bochum.de/en/faculty/professorships/stochastics/group-kuelske/staff/niklas-schubert/}}
\tableofcontents
\hypersetup{linkcolor=blue}

\section{Introduction}

The study of random fields 
under local maps is relevant in many areas of statistical mechanics, probability, 
engineering, image analysis and statistics. 
In this setup a random field is a collection of random variables which are indexed by an infinite lattice, or a graph (network), which may in particular 
be a tree or the integer line. The local random variables, which in statistical mechanics are called \textit{spin variables}, then 
often take values in a finite alphabet. 

It was found in the context of rigorous investigations of the renormalization group
that local transformations may cause 
the so-called \textit{renormalization group pathologies}, see \cite{EFS93,Gr79}. This means that a transformation acting on a strongly coupled spatially Markovian random field may result in an 
image measure which has acquired some internal long-range dependence, and can not 
be represented by an absolutely convergent interaction potential. 
Famous examples are the low-temperature 
Ising model in spatial dimensions $d\geq 2 $ when subjected to a block spin transformation. 
It was proved that this transformation does not only destroy Markovianness, 
but it also destroys quasilocality. This means equivalently in the case of finite-alphabet processes that 
the continuity of finite-volume conditional probabilities w.r.t. the conditioning configurations is lost. Such measures are commonly called non-Gibbsian measures.
See \cite{Sc89,MaReShMo00,EFHR02,KN07,KuLeRe04} for various examples in which such singularities arise following from transformations applied to spin systems.

\subsubsection*{The thinned Bernoulli field (TBF) on graphs} Now, strong coupling is not 
necessary to produce this phenomenon:  
An example where a strictly local transformation even acting on an  \textit{independent} 
process can lead to such singular behavior was discussed in the two papers 
\cite{JaKu23,HeKuSc23}. 

The process they studied was the independent Bernoulli process 
on the integer lattice in spatial dimensions $d\geq 2$, and on a tree, 
under the removal of isolated occupied sites. 
It was found in both cases, that for small enough occupation densities $p$
the image measure was Gibbsian for a quasilocal specification, while for large enough $p$ the image measure 
was non-Gibbsian.

Intuitively it may seem surprising at first sight that it is the large-density 
regime which is singular, as  the removal 
map seems  less intrusive for large $p$ than it is for intermediate $p$. 
Indeed, look at the probability that an occupied  
site $i$ is removed from a realization of 
a Bernoulli site-percolation on a given graph. 
Writing $\hbox{deg}(i)$ for the degree of site $i$, 
we obtain the simple expression  
$p(1-p)^{\hbox{deg}(i)}$. 
This becomes arbitrarily small as the occupation probability $p$ tends to $1$, but still non-quasilocality prevails.

Investigating the process more closely this can be understood heuristically and mathematically, as the 
mechanism and also the proof of non-quasilocality 
at large $p$ rests on a phase transition in a \textit{hidden 
hard-core model} 
(which arises as the Bernoulli field constrained to have only isolated occupied sites). This phase transition can be triggered 
by shape-variations of large domains, 
and for this to be possible large occupation 
probabilities $p$ are beneficial and also necessary. 
For a related study of the natural companion
process to the TBF which only keeps the isolates 
from Bernoulli site-percolation, and removes clusters of size at least two, see \cite{EnJaKu22}.
This latter \textit{projection-to-isolates} process
can be seen as a discrete version of the 
Mat\'ern Type I process which is frequently studied in the continuum
\cite{Ma86,LaPe18} where configurations are point clouds in Euclidean space.  
For studies of transformed point processes, exhibiting transitions between quasilocal and non-quasilocal behavior, see for instance \cite{JaKuZa23}. 

In the present paper we are heading to understand in detail what
happens in the purely one-dimensional situation, with a particular 
view on the asymptotics $p\uparrow 1$, 
where long-range dependence builds up. The underlying transformation, the removal of isolated sites from the i.i.d. Bernoulli field on the line, is illustrated in Figure \ref{fig: Transformation T}.

\begin{figure}[ht!]
\centering
\scalebox{0.5}{
\begin{tikzpicture}[every label/.append style={scale=1.3}]
    \node[shape=circle,draw=black, fill=black, minimum size=0.5cm] (-B) at (-1,0) {};
    \node[shape=circle,draw=black, fill=black, minimum size=0.5cm] (-C) at (-2,0) {};
    \node[shape=circle,draw=black, fill=black, minimum size=0.5cm] (-D) at (-3,0) {};
    \node[shape=circle,draw=black, minimum size=0.5cm] (-E) at (-4,0) {};
    \node[shape=circle,draw=black,  minimum size=0.5cm] (-F) at (-5,0) {};
    \node[shape=circle,draw=cyan, fill=cyan, minimum size=0.5cm] (-G) at (-6,0) {}; \node[shape=circle,draw=black, minimum size=0.5cm] (-H) at (-7,0) {};
    \node[shape=circle,draw=black, fill=black, minimum size=0.5cm] (-I) at (-8,0) {};
    \node[shape=circle,draw=black, fill=black, minimum size=0.5cm] (-J) at (-9,0) {};
    \node[shape=circle,draw=black,  minimum size=0.5cm] (-K) at (-10,0) {};
    \node[shape=circle,draw=black, fill=black, minimum size=0.5cm] (-L) at (-11,0) {};
    \node[shape=circle,draw=black, fill=black, minimum size=0.5cm] (-M) at (-12,0) {};
    \node[shape=circle,draw=black, fill=black, minimum size=0.5cm] (-O) at (-13,0) {};
    \node[shape=circle,draw=black, minimum size=0.5cm, label=above:{$0$}] (A) at (0,0) {};
    \node[shape=circle,draw=cyan, fill=cyan, minimum size=0.5cm] (B) at (1,0) {};
    \node[shape=circle,draw=black,minimum size=0.5cm] (C) at (2,0) {};
    \node[shape=circle,draw=black, fill=black, minimum size=0.5cm] (D) at (3,0) {};
    \node[shape=circle,draw=black, fill=black, minimum size=0.5cm] (E) at (4,0) {};
    \node[shape=circle,draw=black,  minimum size=0.5cm] (F) at (5,0) {};
    \node[shape=circle,draw=black, minimum size=0.5cm] (G) at (6,0) {};
    \node[shape=circle,draw=black, fill=black, minimum size=0.5cm] (H) at (7,0) {};
    \node[shape=circle,draw=black, fill=black, minimum size=0.5cm] (I) at (8,0) {};
    \node[shape=circle,draw=black, fill=black, minimum size=0.5cm] (J) at (9,0) {};
    \node[shape=circle,draw=black,  minimum size=0.5cm] (K) at (10,0) {};
    \node[shape=circle,draw=black,  minimum size=0.5cm] (L) at (11,0) {};
    \node[shape=circle,draw=cyan, fill=cyan, minimum size=0.5cm] (M) at (12,0) {};
    \node[shape=circle,draw=black, minimum size=0.5cm] (O) at (13,0) {};

    \draw[-] (-O)--(-13.7,0){};
    \draw [-] (-O) -- (-M);
     \draw [-] (-L) -- (-M);
     \draw [-] (-L) -- (-K);
    \draw [-] (-J) -- (-K);
    \draw [-] (-I) -- (-J);
    \draw [-] (-H) -- (-I);
    \draw [-] (-G) -- (-H);
    \draw [-] (-F) -- (-G);
    \draw [-] (-F) -- (-E);
    \draw [-] (-E) -- (-D);
    \draw [-] (-D) -- (-C);
    \draw [-] (-C) -- (-B);
    \draw [-] (-B) -- (A);
    \draw [-] (A) -- (B);
    \draw[-](O)--(13.7,0){};
    \draw [-] (M) -- (O);
    \draw [-] (M) -- (L);
    \draw [-] (K) -- (L);
    \draw [-] (K) -- (J);
    \draw [-] (J) -- (I);
    \draw [-] (H) -- (I);
    \draw [-] (G) -- (H);
    \draw [-] (G) -- (F);
    \draw [-] (F) -- (E);
    \draw [-] (E) -- (D);
    \draw [-] (D) -- (C);
    \draw [-] (C) -- (B);
    
    \draw[->, black, very thick] (0,-0.8) to   node [text width=2.5cm,midway,above=0.5cm, align=center] {} (0,-2.8);
    \node[] () at (0.5,-1.8) {\textbf{T}};

    \node[shape=circle,draw=black, fill=black, minimum size=0.5cm] (-1) at (-1,-4) {};
    \node[shape=circle,draw=black, fill=black, minimum size=0.5cm] (-2) at (-2,-4) {};
    \node[shape=circle,draw=black, fill=black, minimum size=0.5cm] (-3) at (-3,-4) {};
    \node[shape=circle,draw=black, minimum size=0.5cm] (-4) at (-4,-4) {};
    \node[shape=circle,draw=black,  minimum size=0.5cm] (-5) at (-5,-4) {};
    \node[shape=circle,draw=black, minimum size=0.5cm] (-6) at (-6,-4) {};
    \node[shape=circle,draw=black, minimum size=0.5cm] (-7) at (-7,-4) {};
    \node[shape=circle,draw=black, fill=black, minimum size=0.5cm] (-8) at (-8,-4) {};
    \node[shape=circle,draw=black, fill=black, minimum size=0.5cm] (-9) at (-9,-4) {};
    \node[shape=circle,draw=black, minimum size=0.5cm] (-10) at (-10,-4) {};
    \node[shape=circle,draw=black, fill=black, minimum size=0.5cm] (-11) at (-11,-4) {};
    \node[shape=circle,draw=black, fill=black, minimum size=0.5cm] (-12) at (-12,-4) {};
    \node[shape=circle,draw=black, fill=black, minimum size=0.5cm] (-13) at (-13,-4) {};
    \node[shape=circle,draw=black,  minimum size=0.5cm, label=above:{$0$}] (0) at (0,-4) {};
    \node[shape=circle,draw=black, minimum size=0.5cm] (1) at (1,-4) {};
    \node[shape=circle,draw=black, minimum size=0.5cm] (2) at (2,-4) {};
    \node[shape=circle,draw=black, fill=black, minimum size=0.5cm] (3) at (3,-4) {};
    \node[shape=circle,draw=black, fill=black, minimum size=0.5cm] (4) at (4,-4) {};
    \node[shape=circle,draw=black, minimum size=0.5cm] (5) at (5,-4) {};
    \node[shape=circle,draw=black, minimum size=0.5cm] (6) at (6,-4) {};
    \node[shape=circle,draw=black, fill=black, minimum size=0.5cm] (7) at (7,-4) {};
    \node[shape=circle,draw=black, fill=black, minimum size=0.5cm] (8) at (8,-4) {};
    \node[shape=circle,draw=black, fill=black, minimum size=0.5cm] (9) at (9,-4) {};
    \node[shape=circle,draw=black, minimum size=0.5cm] (10) at (10,-4) {};
    \node[shape=circle,draw=black, minimum size=0.5cm] (11) at (11,-4) {};
    \node[shape=circle,draw=black, minimum size=0.5cm] (12) at (12,-4) {};
    \node[shape=circle,draw=black, minimum size=0.5cm] (13) at (13,-4) {};

    \draw[-](-13)-- (-13.7,-4){};
    \draw[-](-12)--(-13){};
    \draw[-](-11)--(-12){};
    \draw[-](-10)--(-11){};
    \draw[-] (-9) -- (-10);
    \draw [-] (-8) -- (-9);
    \draw[-] (-8) -- (-7);
    \draw [-] (-7) -- (-6);
    \draw [-] (-6) -- (-5);
    \draw [-] (-5) -- (-4);
    \draw [-] (-4) -- (-3);
    \draw [-] (-3) -- (-2);
    \draw [-] (-2) -- (-1);
    \draw [-] (-1) -- (0);
    \draw [-] (0) -- (1);
    \draw[-](13)--(13.7,-4){};
    \draw[-](12)--(13){};
    \draw[-](11)--(12){};
    \draw[-](10)--(11){};
    \draw[-] (9) -- (10);
    \draw [-] (8) -- (9);
    \draw[-] (8) -- (7);
    \draw [-] (7) -- (6);
    \draw [-] (6) -- (5);
    \draw [-] (5) -- (4);
    \draw [-] (4) -- (3);
    \draw [-] (3) -- (2);
    \draw [-] (2) -- (1);

\end{tikzpicture}
}
\small \caption{An example of a spin configuration on the line and the application of the map $T$. Every coloured dot is an occupied site, while every uncoloured dot is unoccupied. The blue coloured dots mark the isolated occupied sites.}
\label{fig: Transformation T}
\end{figure}
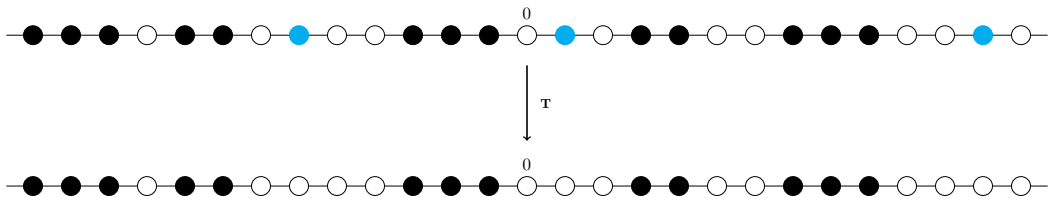

\subsubsection*{Processes on the line: 
Two-sided vs. one-sided view}
The study of processes with one-dimensional index set like $\Z$ or $\R$
has enjoyed renewed interest, in particular with regards to possible long-range 
phenomena \cite{BBDMSW21}, \cite{EnFeMaVe24}. In the present note we are interested in local 
descriptions of infinite-volume measure as consistent 
measures for systems of local kernels which describe conditional behavior on finite volume. 
Now, investigating infinite-volume measures for interacting one-dimensional systems can be done from two perspectives \cite{BeBeCoLe20}: 
First, with a two-sided (Gibbsian) view, 
where one defines consistent infinite-volume measures 
by prescribing probabilities in finite volumes 
conditioning on the outside of the volume, where this means here conditioning on  both \textit{past and future}. 
This is the common procedure in the statistical mechanics 
of lattice systems. 
Coming to phase transitions in the sense 
of multiplicity of the consistent infinite-volume measure, there is a great richness 
in particular in situations of long-range (polynomial) interactions (Dyson models), see
\cite{AfBiEnHa25}, \cite{BiEnvaEn18}, \cite{CaFePaMe05}, \cite{Dy69}, \cite{FrSp82}, \cite{GiGrRu66}, \cite{OkOh25}, \cite{Pa88}, \cite{Pa288}.

The second perspective, where one conditions only on the past, leads to a one-sided view. Here the key objects 
are left-interval specifications (LIS), g-function, and g-measures (which are then the consistent measures). The concept of a g-function is a possibly 
non-local generalization of a Markov chain transition kernel, for precise definitions see Subsection \ref{subsec: GM on line} below.
In this view one asks for shift-invariant measures on the space of two-sided infinite trajectories \cite{Si72} which are consistent under 
application of the given $g$-function. 
Note again, that 
$g$-measures, i.e. the consistent measures for 
a given specifying $g$-function, do not need 
to be unique \cite{Hu06,BrKa93}. Uniqueness only follows with 
further assumptions 
bounding the dependence of the $g$-function 
on the past, see below. 

The one-sided view and the two-sided view are related, but not be fully equivalent: 
\cite{FeGaGr11} provided an example 
of a $g$-measure which is not a Gibbs measure in the two-sided sense. The same phenomenon occurs 
for one-dimensional projections of 
the extremal low-temperature measures 
of the two-dimensional Ising model. These measures 
are famously known as  
Schonman-projections. That they are $g$-measures 
but not Gibbs measures for two-sided specifications 
was only recently proved in \cite{EnSh24}. 
Conversely, there are examples of Gibbs measures for polynomially decaying interactions in the low-temperature region 
which are provably not representable 
as g-measures. This was shown in the context 
of Dyson models in \cite{BiEnEnLe18}.

For a general discussion of fine distinctions for 
$g$-functions depending on the degree of dependence 
on the past, 
and their consequences see \cite{BeFeVe19}, \cite{FeMa04}, \cite{FeMa05}. 
There is moreover an analytical approach to understand properties 
of consistent measures for one-sided specifications, 
via eigenfunctions of Ruelle's 
transition operator, see \cite{EnFeMaVe24,Ma25} and the literature 
therein.

Long-range effects for one-dimensional (and higher-dimensional) systems 
appear also in a related but different way.   
We mention in this context  the study of long-range bond-percolation on the integer line. In these models  
one connects edges between vertices 
independently, with distance-dependent probabilities. 
Then interesting transitions between mean-field behavior and non-mean field behavior 
can be found (\cite{Hu24,AiChChNe88}).

\subsection*{Main results for the TBF on the line}

Let us come back to the thinned Bernoulli site percolation process, 
and provide an outline of the results of our present work. 
Summarizing, while we will show that 
the process is a Gibbs measure in the two-sided sense, and also a g-measure in the one-sided sense, 
we will find \textit{echoes of non-quasilocality} in the limit of large occupation probabilities $p$. Put differently, we see 
an \textit{incipient loss of quasilocality}. 
This means that while quasilocality still holds on the line, 
boundary condition variations are felt over longer 
and longer distances, as $p$ goes to one.

We provide a detailed quantitative discussion of this phenomenon from all three perspectives, as follows. For an overview of the results in this work, see Table \ref{tab:results}.

As some standard Gibbsian theory relies 
on non-nullness, special care is needed in some of our proof steps, as our image process does not 
enjoy the non-nullness property, since 
 isolated sites do not appear by construction.
\newline

\begingroup

\begin{table}[ht]
\scriptsize
    \centering
    \begin{tabular}{c||c|c|c}
    \rowcolor{gray!50}

        \multicolumn{1}{c||}{View}& \vtop{\hbox{\thead{Local description\\ of TBF $\mu_p'$}}}  & Regularity properties & \vtop{\hbox{\thead{Echoes of\\ singularities}}} \\\hline\hline
        \multirow{2}{*}{\thead{Gibbs \\ (2-sided)}}& \multirow{2}{*}{\thead{Explicit form of \\  specification $\gamma'_p$,\\ Thm. \ref{thm: Existence image specification}} }&\thead{B.c. influence,\\ upper bound Thm. \ref{thm: Quasilocality Gamma'}} & \thead{B.c. influence,\\ lower bound Thm. \ref{thm: Quasilocality Gamma'}}  \\ \cline{3-4}
      
        &&\thead{Uniqueness\\$|\mathcal{G}(\gamma'_p)|=1$, Prop. \ref{thm: Uniqueness of the Gibbs measure}}& \thead{Asymptotics\\ for $p\uparrow 1$, Thm. \ref{thm: Quasilocality Gamma'}}\\ \hline
        \rule{0pt}{8mm}
  \thead{g-measure\\(1-sided)} & \vtop{\hbox{\thead{Explicit form of\\g-function  $g_p$ in terms of\\stopping word, Thm. \ref{thm: g-function for mu'_p}}}}&\thead{Uniqueness of $g$-measure\\for $g$-function, Prop. \ref{thm: Unique g-measure}}   & \thead{Parity dependence\\ on b.c.\\ for $p\uparrow 1$, Lem. \ref{lem: Limit g-function n odd}}\\
  \hline
\rowcolor{white}
  \rule{0pt}{5mm}

         \multirow{2}{*}{\thead{Generalized\\House of cards\\Markov chain on $\N_{\geq0}$}} &\thead{\\Explicit form of\\ transition matrix $\Pi_p$, \\ \eqref{eq: Entries transfer matrix Markov chain}}  &\multirow{2}{*}{\thead{Representation\\ $\mu_p'=\tau(\mathds{P}_p^{\text{GHoC}})$,\\ Thm. \ref{thm: Relation TBF and GHoC}}} &  \\\cline{2-2}
     
         &\thead{Unique stationary\\
         measure $\mathds{P}^{\text{GHoC}}_p$}&&\\
    \end{tabular}\medskip
    \begin{tikzpicture}[overlay, remember picture]
                   \scalebox{2}{\node[minimum size=0.5, label=:{$\Downarrow$}] () at (-4.33,-0.05) {};}
                   \scalebox{2}{\node[minimum size=0.5, label=:{$\Downarrow$}] () at (-4.33,-0.73) {};}
                   \scalebox{2.1}{\node[rotate=45,minimum size=0.5] () at (-3.2,-0.39) {$\Downarrow$};}
\end{tikzpicture}
    \caption{Overview of the results on the TBF on the line, depending on the occupation density $p\in (0,1)$.}
    \label{tab:results}
\end{table}   
\endgroup

\noindent\textbf{First view via two-sided specifications: Full construction of specification kernels and large $p$-asymptotics.}  We give a closed formula for the two-sided specification of our process 
which holds on any, possibly disconnected 
finite volumes of arbitrary shapes. We provide a representation which makes 
explicit the correction terms relative to the independent Bernoulli measure. 
We will make further use 
of this representation for 
the construction of the left-interval specification below. 

Next we analyze these expressions and  obtain matching upper and lower bounds on the exponential 
asymptotics in the distance to sets where the boundary conditions are varied.
We pay special attention to the full occupation limit $p\uparrow 1$.   
Interestingly, here the exponential asymptotics become worse, 
and reveals "echoes of non-quasilocality", see Theorem \ref{thm: Quasilocality Gamma'}.

We derive uniqueness of the Gibbs measure for the image specification $\gamma'_p$ via an argument involving the finite-energy condition 
going back to Dobrushin, see \cite{Ge11},  but which is different from 
the Dobrushin uniqueness criterion. 
This holds for all $p\in (0,1)$, and not only in the 
Dobrushin uniqueness region which is strictly smaller. 
\newline

\noindent\textbf{Second view via one-sided specifications.} We discuss in detail the representability of our process via a left-interval specification (LIS). 
A LIS can always be reduced to a single-site object, namely a corresponding g-function. 

To derive the g-function explicitely for our process we proceed 
via limits of two-sided specifications we already obtained in part one. 
Sending the right point of the observation window to infinity we obtain formula \eqref{eq: g-function for the image measure}. 

Also for the one-sided description 
we provide details in particular on the large-density limit, together with combinatorial explanations, see Lemma \ref{lem: Limit g-function n odd} and Remark \ref{rk: Parity dependence on b.c. one-sided view}. 
Rounding up this part, we prove 
uniqueness of the corresponding 
invariant measure 
for our g-function in \eqref{eq: g-function fraction}. Note, that conclusion does not 
rely on the assumption of shift-invariance of 
the measures considered. 
This follows by ensuring summability \eqref{eq: Condition uniqueness g-measure} of the g-function variations
applying the uniqueness criterion of \cite{FeMa05}.\newline  

\noindent\textbf{Third view: House of cards type Markov chain.}
The g-function we obtained in the second part has an intuitive meaning as it can be reinterpreted to provide Markov transition probabilities 
for a (generalized) house of 
cards Markov chain, with state space $\N_0\cup \{\infty\}$. 
This state space has the meaning as the space of possible distances 
to the nearest \text{stopping word} in a conditioning configuration 
in the past, as we will explain, see Subsection \ref{Subsec: Markov chain view: GHoC process}.
See the transition graph in Figure \ref{fig: Markov chain on N}.  
Note that our house-of-cards type Markov  chain 
has modified transition probabilities 
for states $0$ and $1$
w.r.t. to the standard form of a house-of-cards chain presented e.g. in  \cite{ChRe09}.
This view provides a particularly fast argument 
for uniqueness of the invariant measure, however additionally
assuming translation invariance, see Remark \ref{rk: GHoC implies unique g-measure}. The uniqueness statements obtained in the first and the second view are therefore logically stronger. Nevertheless, 
the GHoC Markov chain seems to be the most efficient tool to analyze more properties of the process.

.

\subsubsection*{Organization of the paper}

Section \ref{Sec: Mathematical Preliminaries} contains preliminaries 
concerning one-sided and two-sided consistency, 
the definition of the TBF as an image measure, and the 
introduction of the unfixed area in one dimension. 
The latter concept is relevant to understand in particular the 
form of the two-sided specifications.  
Section \ref{Sec: Main results} presents and explains our results in detail, organized along the three views. 
Section \ref{Sec: Proofs two-sided view} and Section \ref{Sec: Proofs one-sided view} provide the proofs, with 
additional material given in \hyperref[Sec: Appendix A]{Appendix A} and \hyperref[Sec: Appendix B]{B}.

\section{Mathematical preliminaries}\label{Sec: Mathematical Preliminaries}

We will start in this section by giving the necessary background to analyze Gibbs measures on the line. Afterwards, we introduce the model of interest and provide a candidate for a consistent specification. The constructions in Subsection \ref{subsec: GM on line} are based on the ones in the books \cite{Ge11,FV17} and the papers \cite{FeMa04}, \cite{FeMa05}, \cite{FeMa06}. 

\subsection{Gibbs measures on the line}\label{subsec: GM on line}

\textbf{The graph $\Z$.} In this paper, we will regard $\Z$ as a graph $(V,E)$ with vertex set $V=\Z$ and edge set $E=\{\{i,j\}\subset \Z: |i-j|=1\}$. 
Let $\Lambda \subset \Z$. The \textit{outer boundary} of $\Lambda$ is given as the set $\partial_+\Lambda:=\{j\in \Z \setminus \Lambda: j\sim i~\text{for some}~i\in \Lambda\}$ and the \textit{interior boundary} as $\partial_- \Lambda:=\{j\in \Lambda: j \sim i~\text{for some}~i\in \partial_+\Lambda\}$. We will denote with $\Bar{\Lambda}:=\Lambda \cup \partial_+\Lambda$ the \textit{closure} of $\Lambda$ and with $\mathring{\Lambda}:=\Lambda \setminus \partial_-\Lambda$ the \textit{interior} of $\Lambda$. For the sake of brevity, we will regard intervals $I\subset \R$ as subsets of $\Z$, i.e., as $I\cap \Z$, and will write $\Lambda \Subset \Z$ for finite subsets $\Lambda$ of $\Z$. \newline

\textbf{Spin configurations.} We consider i.i.d. Bernoulli variables $\sigma_i$ with \textit{occupation density} $p\in (0,1)$ in each site $i\in \Z$. Thus, we assign to each site $i\in \Z$ a state $\omega_i\in \{0,1\}=:\Omega_0$ and call its collection $\omega:=(\omega_i)_{i\in\Z}$ a \textit{spin configuration}. The set of all possible configurations $\Omega:=\{\omega=(\omega_i)_{i\in\Z}: \omega_i\in \{0,1\} ~\forall i\in \Z\}$ is called \textit{configuration space} and possesses the underlying product $\sigma$-algebra $\mathscr{F}:=\big(\mathcal{P}(\Omega_0)\big)^{\otimes V}$ which is generated by the \textit{spin projections} $\sigma_i:\Omega \rightarrow \Omega_0$, $\omega \mapsto \omega_i$ where $i\in \Z$. We say a site $i\in \Z$ is \textit{occupied} if $\sigma_i=1$ and it is \textit{unoccupied} if $\sigma_i=0$. Let us denote with $\mu_p:=\text{Ber}(p)^{\otimes V}$ the Bernoulli-$p$ product measure on $(\Omega,\mathscr{F})$, and $(\sigma_i)_{i\in \Z}$ is then called the \textit{Bernoulli-$p$ field} on $\Z$. 

If $\Lambda \subset \Z$, we can introduce the restriction of the configuration space on $\Lambda$ as $\Omega_\Lambda:=\{0,1\}^\Lambda$ and the projection onto the spins in $\Lambda$ is then defined as $\sigma_\Lambda: \Omega \rightarrow \Omega_\Lambda$ where $\omega \mapsto \omega_\Lambda:=(\omega_i)_{i\in \Lambda}$. To ease readability, we will omit the projections in the notation, e.g., abbreviating expressions like $\mu_p(\sigma_\Lambda=\omega_\Lambda)$ by $\mu_p(\omega_\Lambda)$, unless the projections are necessary for understanding. It is also useful to consider events in $\mathscr{F}$ which depend only on spins in a certain subset $\Delta \subset \Z$. Therefore, we define $\mathscr{F}_\Delta:=\sigma(\sigma_i,~i\in \Delta)$ as the $\sigma$-algebra on $\Omega$ which is generated by all events occurring in $\Delta$. Let $\Lambda\subset \Delta \subset \Z$, $\omega\in \Omega_\Lambda$ and $\eta \in \Omega_{\Delta \setminus \Lambda}$, then the \textit{concatenation} $\omega\eta\in \Omega_\Delta$ is defined as $\sigma_\Lambda(\omega\eta)=\omega$ and $\sigma_{\Delta \setminus \Lambda}(\omega\eta)=\eta$.\newline

\textbf{Quasilocal specifications.} We would like to investigate the Bernoulli field on the line under certain constraints. In order to describe the behaviour of these constraints on the model, we will need the notion of \textit{specifications}. These are families $\gamma=(\gamma_\Lambda)_{\Lambda \Subset \Z}$ of proper probability kernels each from $(\Omega,\mathscr{F}_{\Lambda^c})$ to $(\Omega,\mathscr{F})$ satisfying a consistency relation. A kernel $\gamma_\Lambda$ is called \textit{proper} if $\gamma_\Lambda(A|\omega)=\mathds{1}_A(\omega)$ for all $A\in \mathscr{F}_{\Lambda^c}$. Let $\Lambda \subset \Delta\Subset \Z$, then the two kernels $\gamma_\Lambda$ and $\gamma_\Delta$ should satisfy the following \textit{consistency relation} $(\gamma_\Delta \gamma_\Lambda)(A|\omega)=\gamma_\Delta(A|\omega)$ for all $A\in \mathscr{F}$ and $\omega\in \Omega$. 

Our goal is to show that the influence of perturbations in the boundary condition decreases with the distance to the given finite observation window $\Lambda$. Therefore, we introduce a \textit{local} function $f:\Omega \rightarrow \R$ which is defined to be $\mathscr{F}_\Lambda$-measurable for some subset $\Lambda \Subset \Z$. Then, we call a specification $\gamma=(\gamma_\Lambda)_{\Lambda\Subset \Z}$ \textit{quasilocal} if for each $\Lambda \Subset \Z$ and every local function $f:\Omega \rightarrow \R$ the following holds
\begin{equation*}
    \lim_{n \rightarrow \infty} \sup_{\substack{\zeta,\eta\in \Omega\\\zeta_{\Lambda_n}=\eta_{\Lambda_n}}}\big|\int f(\omega)\gamma_\Lambda(d\omega|\zeta)-\int f(\omega)\gamma_\Lambda(d\omega|\eta)\big|=0
\end{equation*}
where $(\Lambda_n)_{n\in \N}$ is a \textit{cofinal sequence}, i.e. $\Lambda_n \subset \Lambda_m \Subset \Z$ for all $n\leq m$ and $\bigcup^\infty_{n=1} \Lambda_n=\Z$.\newline

\textbf{Gibbs measures and the quasilocal Gibbs property.} We are interested in infinite-volume measures being consistent with these specification kernels in the sense that these kernels provide a regular conditional distribution for this measure. In more detail, let $\gamma=(\gamma_\Lambda)_{\Lambda \Subset \Z}$ be a specification, and $\mu\in \mathscr{M}_1(\Omega,\mathscr{F})$, a probability measure on the whole line. Then we call $\mu$ a \textit{Gibbs measure} for $\gamma$ if it satisfies the \textit{DLR-equation}
\begin{equation*}
    \mu(A|\mathscr{F}_{\Lambda^c})=\gamma_\Lambda(A|\cdot)~~~~\mu\text{-almost surely}
\end{equation*}
for all $\Lambda \Subset \Z$ and $A\in \mathscr{F}$. We denote with $\mathscr{G}(\gamma)$ the set of all Gibbs measures for $\gamma$. Finally, $\mu$ is called \textit{quasilocally Gibbs} if there exists a quasilocal specification $\gamma$ such that $\mu\in \mathscr{G}(\gamma)$.\newline

\textbf{One-sided specifications.} Let us denote with $\mathcal{S}_b$ the set of finite intervals and introduce for each $\Lambda\in \mathcal{S}_b$ the abbreviations $m_\Lambda:=\max\Lambda$ as well as $l_\Lambda:=\min \Lambda$. Furthermore, we write $\Lambda_-:=(-\infty,l_\Lambda]$ and $\Lambda_+:=[m_\Lambda,+\infty)$ as well as $\mathscr{F}_{\leq x}:=\mathscr{F}_{(-\infty,x]}$ and $\mathscr{F}_{<x}:=\mathscr{F}_{(-\infty,x)}$. Let $\rho=(\rho_\Lambda)_{\Lambda\in \mathcal{S}_b}$ be a family of probability kernels $\rho_\Lambda:\mathscr{F}_{\leq m_\Lambda}\times \Omega\rightarrow [0,1]$. We call $\rho$ a \textit{left interval-specification (LIS)} on $(\Omega,\mathscr{F})$ iff it satisfies the following properties: For each $A\in \mathscr{F}_{\leq m_\Lambda}$, the map $\rho_\Lambda(A|\cdot)$ is $\mathscr{F}_{<l_\Lambda}$-measurable and for each $B\in \mathscr{F}_{<l_\Lambda}$ and $\omega\in \Omega$, we have $\rho_\Lambda(B|\omega)=\mathds{1}_B(\omega)$. Furthermore, for each $\Lambda,\Delta\in \mathcal{S}_b$ with $\Lambda\subset\Delta$ it has to satisfy the following consistency equation $(\rho_\Delta\rho_\Lambda)(f|\omega)=\rho_\Delta(f|\omega)$ for each function $f$ which is $\mathscr{F}_{\leq m_\Lambda}$-measurable and each configuration $\omega\in \Omega$.\newline

\textbf{g-measures and g-functions.} A probability measure $\mu$ on $(\Omega,\mathscr{F})$ is said to be \textit{consistent} with a left interval-specification $\rho$ if $\mu\rho_\Lambda(f)=\mu(f)$ for each $\Lambda\in \mathcal{S}_b$ and each function $f$ which is $\mathscr{F}_{\leq m_\Lambda}$-measurable.

Note that the single-site kernels $(\rho_{\{i\}})_{i\in \Z}$ of a LIS already determine the entire LIS $\rho=(\rho_\Lambda)_{\Lambda\in \mathcal{S}_b}$ in a unique kind of way, see Theorem 3.1 in \cite{FeMa05} or Theorem 3.2 in \cite{FeMa04}. If the model is translation invariant, it suffices to compute the kernel at the origin which we denote by $g(\omega):=\rho_{\{0\}}(\omega_{\{0\}}|\omega_{\{0\}^c})$ for all $\omega\in \Omega$. We call this function the \textit{g-function} and any translation-invariant measure $\mu$ that is consistent with the associated LIS $\rho$ is called a \textit{g-measure}.

\subsection{The projection to non-isolation}\label{subsec: Projection to non-iso}

Define the projection to the non-isolated configurations as the map $T:\Omega \rightarrow \Omega$ which is given as
\begin{equation}\label{eq: Map T}
    (T\omega)_i:=\omega'_i:=\omega_i \bigg(1-\prod_{j\in \partial i}(1-\omega_j)\bigg)
\end{equation}
in a site $i \in \Z$ and for a configuration $\omega \in \Omega$. An application of the map $T$ to a particular configuration $\omega$ is depicted in Figure \ref{fig: Transformation T}. For the brevity of notation, we will write $T_\Lambda(\omega)$ instead of $\big(T(\omega))_\Lambda$ which is the restriction of the map $T$ on a subset $\Lambda \subset \Z$. Note that $T_\Lambda:\Omega \rightarrow \Omega_{\Lambda}$ is $\mathscr{F}_{\Bar{\Lambda}}$-measurable for any $\Lambda \subset \Z$. In order to investigate the transformed model, we will use a two-layer representation. Let us denote with $\Omega':=T(\Omega)\subset \Omega$ the \textit{second-layer} which is the set of configurations obeying the non-isolation constraint and define $\mathscr{F}'_\Lambda:=\mathscr{F}_\Lambda \cap \Omega'$ for any $\Lambda \subset \Z$. We will consider for each $\omega'\in \Omega'$ the inverse image $T^{-1}(\omega')$ and call this the (configuration space of) the \textit{constrained first-layer model of $\omega'$}. In the particular case of the all-unoccupied configuration $0'\in \Omega'$, we obtain the first-layer model constrained on isolation on the full line $\Z$. The symbols describing second-layer quantities will carry a prime, to distinguish them from the objects in the first-layer. Our objective is to analyze the \textit{thinned Bernoulli field (TBF)} described by $\mu_p':=\mu_p \circ T^{-1}$ on $(\Omega',\mathscr{F}'_\Z)$ for all $p\in (0,1)$.

\begin{figure}[ht!]
\centering
\scalebox{0.5}{
\begin{tikzpicture}[every label/.append style={scale=1.3},fill fraction/.style={path picture={
\fill[#1] 
(path picture bounding box.south) rectangle
(path picture bounding box.north west);
}},
fill fraction/.default=gray!50
]
    \node[shape=circle,draw=orange, fill=orange, minimum size=0.5cm] (-B) at (-1,0) {};
    \node[shape=circle,draw=orange, fill=orange, minimum size=0.5cm] (-C) at (-2,0) {};
    \node[shape=circle,draw=orange,  minimum size=0.5cm] (-D) at (-3,0) {};
    \node[shape=circle,draw=orange, minimum size=0.5cm] (-E) at (-4,0) {};
    \node[shape=circle,draw=orange, fill=orange,  minimum size=0.5cm] (-F) at (-5,0) {};
    \node[shape=circle,draw=orange, fill=orange, minimum size=0.5cm] (-G) at (-6,0) {}; \node[shape=circle,draw=orange, fill=orange, minimum size=0.5cm] (-H) at (-7,0) {};
    \node[shape=circle,draw=orange,  minimum size=0.5cm] (-I) at (-8,0) {};
    \node[shape=circle,draw=black,fill fraction=black,  minimum size=0.5cm] (-J) at (-9,0) {};
    \node[shape=circle,draw=black,fill fraction=black,  minimum size=0.5cm] (-K) at (-10,0) {};
    \node[shape=circle,draw=orange,  minimum size=0.5cm] (-L) at (-11,0) {};
    \node[shape=circle,draw=orange, fill=orange, minimum size=0.5cm] (-M) at (-12,0) {};
    \node[shape=circle,draw=orange, fill=orange, minimum size=0.5cm] (-O) at (-13,0) {};
    \node[shape=circle,draw=orange, minimum size=0.5cm] (A) at (0,0) {};
    \node[shape=circle,draw=black,fill fraction=black,  minimum size=0.5cm] (B) at (1,0) {};
    \node[shape=circle,draw=black,fill fraction=black,minimum size=0.5cm] (C) at (2,0) {};
    \node[shape=circle,draw=black, fill fraction=black,   minimum size=0.5cm] (D) at (3,0) {};
    \node[shape=circle,draw=orange,  minimum size=0.5cm] (E) at (4,0) {};
    \node[shape=circle,draw=orange, fill=orange,  minimum size=0.5cm] (F) at (5,0) {};
    \node[shape=circle,draw=orange, fill=orange,minimum size=0.5cm] (G) at (6,0) {};
    \node[shape=circle,draw=orange, fill=orange, minimum size=0.5cm] (H) at (7,0) {};
    \node[shape=circle,draw=orange,  minimum size=0.5cm] (I) at (8,0) {};
    \node[shape=circle,draw=black, fill fraction=black,  minimum size=0.5cm] (J) at (9,0) {};
    \node[shape=circle, draw=orange,  minimum size=0.5cm] (K) at (10,0) {};
    \node[shape=circle,draw=orange, fill=orange, minimum size=0.5cm] (L) at (11,0) {};
    \node[shape=circle,draw=orange, fill=orange, minimum size=0.5cm] (M) at (12,0) {};
    \node[shape=circle,draw=orange, fill=orange,minimum size=0.5cm] (O) at (13,0) {};
     \scalebox{1.4}{\node[] () at (-3,0.8) {{\color{orange}$\overline{\Theta}$}};}

    \draw[-] (-O)--(-13.7,0){};
    \draw [-] (-O) -- (-M);
     \draw [-] (-L) -- (-M);
     \draw [-] (-L) -- (-K);
    \draw [-] (-J) -- (-K);
    \draw [-] (-I) -- (-J);
    \draw [-] (-H) -- (-I);
    \draw [-] (-G) -- (-H);
    \draw [-] (-F) -- (-G);
    \draw [-] (-F) -- (-E);
    \draw [-] (-E) -- (-D);
    \draw [-] (-D) -- (-C);
    \draw [-] (-C) -- (-B);
    \draw [-] (-B) -- (A);
    \draw [-] (A) -- (B);
    \draw[-](O)--(13.7,0){};
    \draw [-] (M) -- (O);
    \draw [-] (M) -- (L);
    \draw [-] (K) -- (L);
    \draw [-] (K) -- (J);
    \draw [-] (J) -- (I);
    \draw [-] (H) -- (I);
    \draw [-] (G) -- (H);
    \draw [-] (G) -- (F);
    \draw [-] (F) -- (E);
    \draw [-] (E) -- (D);
    \draw [-] (D) -- (C);
    \draw [-] (C) -- (B);

\end{tikzpicture}
}
\small \caption{A second-layer configuration $\omega' \in \Omega'$ on the line $\Z$. 
Clusters of occupied sites $\Theta$ prescribed by $\omega'$ are in full orange, all other sites are empty in $\omega'$. 
The fixed area $\overline{\Theta}$ of this configuration is highlighted in orange and the remaining half-black and half-white colored dots correspond to the unfixed areas in $\mathscr{U}$. 
They may be occupied or unoccupied, while obeying the isolation constraint, in the underlying first-layer configuration.
}
\label{fig: Fixed area}
\end{figure}
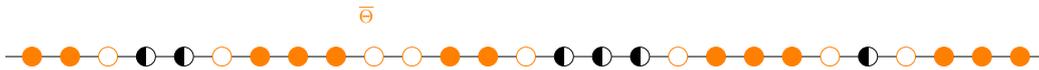

\begin{remark}\label{rk: Fixed area}
    Every configuration $\omega' \in \Omega'$ in the second-layer is uniquely determined by its connected subsets of occupied sites in $\Z$. Recall that $\Omega'$ contains the non-isolated configurations and thus each subset consists of occupied clusters of size at least two. All of these clusters are surrounded by unoccupied sites. Let us denote the union of the clusters of ones together with the surrounding layer of zeros by $\overline{\Theta} \subset \Z$, see Figure \ref{fig: Fixed area}. Thus, every configuration $\omega \in T^{-1}(\omega')$ has to coincide with $\omega'$ on $\overline{\Theta}$. Furthermore, each configuration $\omega \in T^{-1}(\omega')$ contains only isolated sites outside of $\overline{\Theta}$ regarding a boundary of zeros, compare Figure \ref{fig: Fixed area}.
\end{remark}

Let us rigorously define the objects introduced in Remark \ref{rk: Fixed area}.

\begin{definition}\label{def: Fixed area}
    Let $\omega' \in \Omega'$ be a non-isolated configuration. The \textbf{fixed area} of $\omega'$ is defined as $\overline{\Theta}=\Theta\cup \partial_+\Theta$ where  
   \begin{equation*}
       \Theta=\Theta(\omega'):= \{x\in \Z: \omega'_x=1\} \subset \Z
   \end{equation*}
   denotes the set of occupied sites, 
   and we say that the spins in the first-layer model of $\omega'$ are \textbf{fixed} in $\overline{\Theta}$. The complement $(\overline{\Theta})^c$ is called the \textbf{unfixed area} of $\omega'$. For the brevity of notation, we will write $\overline{\Theta}_{\Lambda}:=\overline{\Theta}\cap \Lambda$ for a subset $\Lambda \subset \Z$. Moreover, let us denote by $\mathscr{U}=\mathscr{U}(\omega')$ the set of connected components of $(\overline{\Theta})^c$. 
\end{definition}

\section{Main results}\label{Sec: Main results}

In this section, we present the main results of the paper concerning the TBF on the line. To analyze this model, we adopt three complementary perspectives:
\begin{enumerate}[label=\roman*)]
    \item the \textbf{two-sided view} via Gibbs measures and Specifications (Subsection \ref{Subsec: Two-sided view: Gibbs measure}),
    \item the \textbf{one-sided view} via g-measures and LIS (Subsection \ref{Subsec: One-sided view: g-measure}),
    \item the \textbf{Markov chain view} via the GHoC Markov process and its stationary distribution (Subsection \ref{Subsec: Markov chain view: GHoC process}).
\end{enumerate}
These three perspectives build upon one another, and the  results for each view, as well as their interrelations, are summarized in Table \ref{tab:results}.

\subsection{Two-sided view: Gibbs measure}\label{Subsec: Two-sided view: Gibbs measure}

 In the two-sided view, the TBF is locally described by a specification $\gamma'_p$, whose kernels $\gamma'_{p,\Lambda}$ give the conditional probabilities of $\mu_p'$ inside a finite observation set $\Lambda$, given the configuration outside of $\Lambda^c$. We begin by providing an explicit specification $\gamma'_p$ for the TBF $\mu_p'$.

\begin{theorem}\label{thm: Existence image specification}     
    For each $p\in (0,1)$, there exists a  specification $\gamma_p'=(\gamma_{p,\Lambda}')_{\Lambda\Subset \Z}$ 
    as a family of consistent kernels on $\Omega'$
    for the TBF $\mu_p'$. Let $\Lambda\Subset \Z$ be a finite observation set and $\omega'\in \Omega'$ be a boundary condition, then each kernel $\gamma'_{p,\Lambda}$ has the form
\begin{equation}\label{eq: Kernels image specification}
    \gamma'_{p,\Lambda}(\omega'_\Lambda|\omega'_{\Lambda^c})=\frac{1}{Z_\Lambda(\omega'_{\Lambda^c})}p^{|\Theta\cap\bar{\Lambda}|}(1-p)^{|\partial_+ \Theta \cap \bar{\Lambda}|}\prod_{U\in \mathscr{U}_\Lambda(\omega')}Z(U)
\end{equation}
where $Z_\Lambda(\omega'_{\Lambda^c})$ is the partition function chosen to obtain a probability kernel on $\Omega'$. Recall that $\overline{\Theta}$ is the fixed area of $\omega'$, compare Definition \ref{def: Fixed area}. Furthermore,
\begin{equation}\label{def: Partition functions unfixed area}
    Z(U):=(1-p)^{-1}\cdot\begin{cases}
        Q^{|U|+1}(0,0)~~&\text{if}~~|U|<\infty,\\
        (\lambda_{\text{PF}})^{|U\cap V_-|},~~&\text{if}~~|U|=+\infty~~\text{and}~~\sup U<\infty,\\
        (\lambda_{\text{PF}})^{|U\cap V_+|},~~&\text{if}~~|U|=+\infty~~\text{and}~~\inf U<\infty,\\
       D\cdot(\lambda_{\text{PF}})^{|V_-\cap V_+|},~~&\text{else}
        
    \end{cases}
\end{equation}
denotes the partition function on the connected component $U\in \mathscr{U}_\Lambda$ of the unfixed area for $\omega'$, see \eqref{eq: restricted unfixed areas}. Here, $Q=Q(p)$ is the symmetric transfer operator of the model constrained on isolation, given below in \eqref{eq: Transfer matrix Q}, $\lambda_{\text{PF}}=\lambda_{\text{PF}}(p)$ its Perron-Frobenius eigenvalue, see \eqref{eq: Eigenvalues of Q} and $D=D(p):=\frac{(1-p)(\lambda_{\text{PF}}+2p)}{\lambda_{\text{PF}}^3}$. Furthermore, we denote by $V_-:=[\min \overline{\Lambda},+\infty)$ and $V_+:=(-\infty,\max \overline{\Lambda}]$.
\end{theorem}

Let us give a comment on the expression of the specification kernels $\gamma'_{p,\Lambda}$ in \eqref{eq: Kernels image specification}.

\begin{remark}
Under Bernoulli site-percolation, the probability of a configuration $\omega'\in \Omega'$ restricted to its fixed area $\overline{\Theta}$ is equal to the Bernoulli weights of $\omega'$ on $\overline{\Theta}$. In particular, each site $i\in \Theta$ belongs to an occupied cluster of size at least two and carries a weight $p$, while each site $i\in \partial_+\Theta$, in the outer boundary of such clusters, carries a weight $1-p$. The probabilities associated with the unfixed areas $U\in \mathscr{U}$ are given by $Z(U)$. These weights are partition functions of the model constrained on isolation on the connected subset $U$ with a fully unoccupied boundary condition outside of $U$, see Remark \ref{rk: Fixed area} and Definition \ref{def: Fixed area} for details on the fixed and unfixed areas. Note that the exponents of $\lambda_{\text{PF}}$ in the definition \eqref{def: Partition functions unfixed area} of $Z(U)$ are finite. They represent the length of the finite subset of $U$ which remains after removing the infinite parts outside of $(\min \Lambda,\max\Lambda)$. Moreover, given a boundary condition $\omega'_{\Lambda^c}$, the numerator of $\eqref{eq: Kernels image specification}$ takes into account the weights of all spins that may influence the spins in $\Lambda$ independently of the given configuration in $\omega'_{\Lambda}$, see the definition of the influencing set in \eqref{eq: Influence set} and Remark \ref{rk: Influence set}. This explains the choice of $\overline{\Theta}_\Lambda$ and $\mathscr{U}_\Lambda$.

Note 
that the complicated looking 
last three lines of the definition of $Z(U)$ must also be considered, as a specification 
needs to be defined for all possible boundary conditions in $\Omega'$. Their form
ensures that  
the specification is 
quasilocal as a function of the boundary condition, to be shown below. 
For a typical 
boundary condition (which has finite components for the unfixed area) only the first line in \eqref{def: Partition functions unfixed area} matters. 
\end{remark}

Let us continue to aim for the analysis of the exponential asymptotics of $\gamma'_p$ in the distance to the sets where the boundary conditions are varied. In more detail, we investigate 
\begin{equation}\label{eq: Sensitivity of b.c. variations}
        s(\Lambda,\Delta):=\sup_{\substack{\omega',\eta'\in \Omega'\\\omega'_\Delta=\eta'_\Delta}} \big|
    \gamma'_{p,\Lambda}(\omega'_\Lambda|\omega'_{\Lambda^c})-\gamma'_{p,\Lambda}(\omega'_\Lambda|\eta'_{\Lambda^c})\big|
    \end{equation}
    for $\Lambda\subset \Delta\Subset \Z$ which describes the sensitivity of probabilities on $\Lambda$ with respect to boundary condition variations outside of $\Delta$. The following theorem shows that $\gamma'_p$ possesses matching upper and lower bounds in the exponential asymptotics. 
    
    \begin{theorem}\label{thm: Quasilocality Gamma'}
     Let $p \in (0,1)$, $\gamma'_p$ be the specification for the TBF $\mu_p'$ given in Theorem \ref{thm: Existence image specification}, and $l,r,L,R \in\mathbb{Z}$ satisfying $L< l-1\leq r+1< R$.
 Then, the following finite-volume estimates hold 
on finite observation windows $[l,r]$ w.r.t 
boundary condition variations outside of $[L,R]$
     \begin{equation}\label{eq: Bounds on sensitivity}
      C_-(l,r) \cdot |a|^{n}\leq  s([l,r],[L,R])\leq C_+(l,r) \cdot |a|^{n},
    \end{equation}
    where $a=a(p)$ is the eigenvalue ratio of the transfer matrix $Q$ corresponding to the model constrained on isolation, see \eqref{eq: Fraction eigenvalues} and $n:=\min\{l-L,R-r\}$.
    The prefactors $C_-(l,r)$ and $C_+(l,r)$ depend on the size of the observation window $[l,r]$ itself and satisfy
    \begin{equation*}
        C_-(l,r)= A_-\cdot(b_-)^{r-l+1}~~\text{and}~~C_+(l,r)= A_+\cdot(b_+)^{r-l+1}
    \end{equation*}
    where $A_-,b_-\in(0,1)$ are given in \eqref{eq: Constants b_- and A_-} and $b_+\in(0,1)$ as well as $A_+$ are given in \eqref{eq: Constants b_+ and A_+}.
\end{theorem}

\begin{remark}\label{rk: Echoe of quasilocality two-sided}
The case of $p \uparrow 1$, in which the occupation becomes denser and denser, appears on first sight to be trivial, as there are fewer and fewer isolated sites. However, we see in \eqref{eq: Bounds on sensitivity} that the precise rate of dependence on b.c. variation deteriorates, as $|a|$ goes to one in the limit 
of full occupation. 
So the case of almost full occupation is the most non-trivial one, and we see an \say{echo of non-quasilocal behavior}.

 Note that the main asymptotics is provided by $|a|^n$. This can be understood as the eigenvalue ratio $a$ controls the transmission of information over distance $n$ in the hidden model, which is the model constrained on isolation. This hidden model transports boundary condition variations for conditional probabilities of the TBF. 

The constants $b_-$ and $b_+$ are related to the weights {\normalfont on} the observation window $[l,r]$. The constant $b_-$ has the character of a non-nullness constant 
on the space of non-isolates, whereas $(b_+)^{r-l+1}$ provides an upper bound on the maximal possible weight in the observation window. This will become clear in the proofs. 
\end{remark}

Note that our specification is supported on the smaller space $\Omega'$ (the subspace of non-isolated configurations). One may worry that this causes problems 
in the application of the finite-energy criterion, which is an 
efficient tool to prove one-dimensional uniqueness. However, it is available 
(in the standard reference \cite{Ge11}) only under the assumption that the specifications are defined on the \textit{full product space}, 
and have the corresponding finite-energy property in that space. 
Therefore, for the convenience of the reader we include a 
discussion in \hyperref[Sec: Appendix B]{Appendix B} which explains 
why we can still apply this tool in our case.

\begin{proposition}\label{thm: Uniqueness of the Gibbs measure}
    Consider the specification $\gamma'_p$ given in Theorem \ref{thm: Existence image specification}. The Bernoulli product measure $\mu_p$ under the transformation $T$, defined in \eqref{eq: Map T}, is the unique Gibbs measure $\mu'_p$ for the specification $\gamma'_p$ for every $p\in (0,1)$. More precisely, the model satisfies the finite energy condition: there exists a constant $C>0$ such that for each cylinder event $A \in \mathscr{F}'$ one can find a subset $\Lambda \Subset \Z$ with the property that
    \begin{equation}\label{Ineq: finite energy cond}
    \gamma'_{p,\Lambda}
    (A|\zeta') \geq C~ \gamma'_{p,\Lambda}(A|\eta')
\end{equation}
for all configurations $\zeta', \eta' \in \Omega'$. 
\end{proposition}

The proofs of these results are given in Section \ref{Sec: Proofs two-sided view}.

\subsection{One-sided view: g-measure}\label{Subsec: One-sided view: g-measure}

We continue with a one-sided view of the model using left-interval specifications (LIS's). They are prescribing conditional probabilities of the next symbol given the information in the infinite past of this symbol. Since our model is spatially translation-invariant, the associated $g$-function, namely, the single-site kernel at the origin, suffices to describes an entire consistent LIS. In the following theorem, we present the explicit form of a $g$-function for the TBF.

\begin{theorem}\label{thm: g-function for mu'_p}
    For all $p\in (0,1)$, the TBF $\mu_p'$ is consistent with a translation-invariant LIS $\rho_p$. The corresponding g-function depends only on the spin at the origin and on the distance $n$ to the nearest pair of ones in the past of a configuration $\omega'\in \Omega'$. More precisely, define 
    \begin{equation}\label{eq: distance to the next pair of ones}
        n=n(\omega'):=\min\big\{i\in \N:~(\omega'_{\{-i-1\}},\omega'_{\{-i\}})=(1,1)\big\}.
    \end{equation}
    Set $n=0$ if $(\omega'_{\{-2\}},\omega'_{\{-1\}})=(0,1)$, and $n=\infty$ if no such pair appears in the past of $\omega'$. The $g$-function is then defined by $g_p(n):=\rho_{p,\{0\}}(0_{\{0\}}|\omega')$, which gives the probability of finding another empty site at the origin, given that the nearest pair of occupied sites in the past occurs at distance $n$. The function $g_p$ can be expressed in terms of the eigenvalues $\lambda_{\text{PF}}$ and $\lambda_{\text{r}}$, computed in \eqref{eq: Eigenvalues of Q}, corresponding to the symmetric transfer matrix $Q$ for the model constrained on isolation, where $Q$ is explicitly given in \eqref{eq: Transfer matrix Q}.
    Explicitly, we have $g_p(0)=0$, $g_p(1)=1-p$, while for $n\in \N_{\geq 2}\cup \{\infty\}$,
    \begin{equation}\label{eq: g-function for the image measure}
                g_p(n)=\frac{(-1)^{n+1}|a|^n\frac{\lambda_{PF}}{1-\lambda_{r}}+ \frac{\lambda_{PF}}{1-\lambda_{PF}}}{(-1)^n|a|^{n-1}\frac{1}{1-\lambda_{r}}+\frac{1}{1-\lambda_{PF}}}
    \end{equation}
    where $a=a(p)$ denotes the eigenvalue ratio introduced in \eqref{eq: Fraction eigenvalues} and displayed in Figure \ref{fig: Speed of convergence and plot g-function} (a).
    Figure \ref{fig: Speed of convergence and plot g-function} (b) displays $g_p$ as a function of $p$ for various values of $n$.
\end{theorem}
\begin{figure}[ht]
    \centering
    \subfloat[The plot of the eigenvalue ratio $a(p)$ which is monotonically decreasing and negative for all $p\in (0,1)$. Moreover, note that $\lim_{p\downarrow 0} a(p)=0$ and $\lim_{p\uparrow 1}a(p)=-1$.]
    {\includegraphics[width=0.425\textwidth]{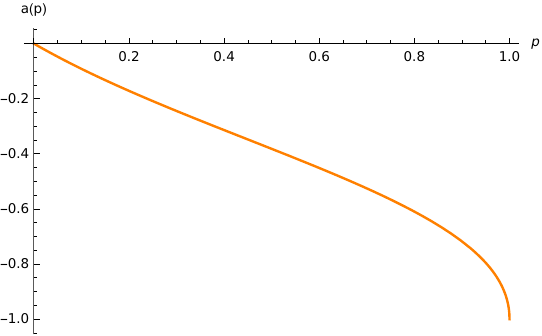}}%
    \hfill
    \subfloat[The plot of the function $g_p(n)$ for different values of the distance $n$ to the next pair of occupied sites in the past of the origin.]
    {\includegraphics[width=0.525\textwidth]{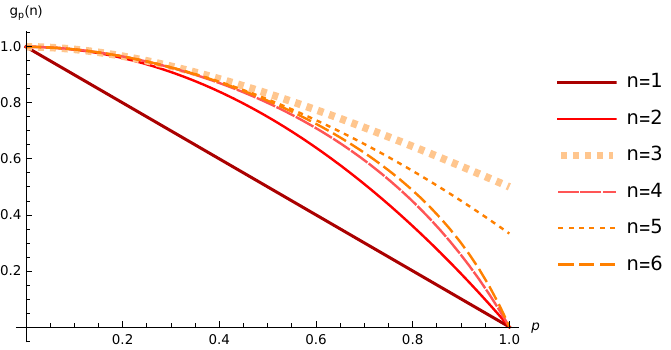}}
    \caption{Depicted are plots the eigenvalue ratio $a=a(p)$ and the $g$-function $g_p$ as a function of the occupation density $p$ describing the behavior of the TBF.}
    \label{fig: Speed of convergence and plot g-function}
\end{figure}

Note the interesting dependence on whether the distance to an occupied boundary condition is even or odd as $p\uparrow 1$.

\begin{lemma}\label{lem: Limit g-function n odd}
Let $g_p$ be the function defined in \eqref{eq: g-function for the image measure}. For $n\in \N_{\geq 3}$ odd, we have $\lim_{p\uparrow 1} g_p(n)=\frac{2}{n+1}$, whereas for $n\in \N$ even, we have $\lim_{p\uparrow 1} g_p(n)=0$.
\end{lemma}

 The dependence of fully occupied boundary conditions on whether the distance to the root is even or odd was also observed in the TBF on trees with bounded degrees in the regime of sufficiently large occupation density $p$. In that case, jumps of finite size remained even in the large $n$ limit. This fact was used in \cite{HeKuSc23} and \cite{JaKu23} to exhibit the non-quasilocality of the conditional probabilities of the TBF. Therefore, Lemma \ref{lem: Limit g-function n odd} should be seen as a weaker form of the same phenomenon in the present one-dimensional case. Finally, we investigate the number of consistent measures for the LIS $\rho_p$ given by the $g$-function in \eqref{eq: g-function for the image measure}.

\begin{proposition}\label{thm: Unique g-measure}
    Let $p\in (0,1)$. The LIS $\rho_p$ given in Theorem \ref{thm: g-function for mu'_p} possesses a unique consistent measure, namely the TBF $\mu'_p$.
\end{proposition}

The proofs of Theorem \ref{thm: g-function for mu'_p} and Proposition \ref{thm: Unique g-measure} are given in Section \ref{Sec: Proofs one-sided view}, while the proof of Lemma \ref{lem: Limit g-function n odd} is provided in \hyperref[Sec: Appendix B]{Appendix B}, which also contains a discussion of the limit $p\rightarrow 1$ of the $g$-function.

\subsection{Markov chain view: Generalized House of Cards process (GHoC)} \label{Subsec: Markov chain view: GHoC process}

Finally, we show that the model can be viewed as a Markov process on $\Z$ with state space $\N_0\cup \{\infty\}$, where each state represents the distance to the next double one in the past of the given site. The transition probabilities are given by the $g$-function introduced in \eqref{eq: g-function for the image measure}. This process admits an intuitive interpretation as a \textit{house of cards type Markov chain}. In contrast to the more common house of cards process of \cite{ChRe09}, the transition probabilities are modified for jumps from the state $0$ and $1$.

 Given a configuration $\omega'\in \Omega'$, define for each site $i\in \Z$ the distance to the next pair of ones in the past of this site as
\begin{equation}\label{def: D2D1 definition}
    n_i=n_i(\omega'):=\min\{k\in \N: (\omega'_{i-k-1},\omega'_{i-k})=(1,1)\},
\end{equation}
$n_i:=0$ if $(\omega'_{i-2},\omega'_{i-1})=(0,1)$ and $n_i:=\infty$ if there is no occupied site before $i$. Furthermore, define a \textit{GHoC path} as a sequence $(a_n)_{n\in \Z}$ taking values in $\N_0\cup \{\infty\}$ and satisfying the constraints
\begin{equation*}
    \begin{cases}
        a_{n+1}\in \{0,a_{n}+1\}~&\text{if}~a_n\notin \{0,1\},\\
        a_{n+1}\in \{1,2\}~&\text{if}~a_n=1,\\        a_{n+1}=1~&\text{if}~a_n=0.
    \end{cases}
\end{equation*}
There is a one-to-one correspondence between a configuration $\omega'\in \Omega'$ and a GHoC path $(a_n)_{n\in \Z}$ taking values in $\N_0\cup \{\infty\}$, see for example Figure \ref{fig: D2D1 sequence}. This correspondence can be formalized by introducing the local map 
\begin{equation}\label{eq: Definition map Tau}
  \tau:\N_0\cup \{\infty\}\rightarrow\{0,1\},~~  \tau(n):=\mathds{1}_{\{0,1\}}(n),
\end{equation}
see Figure \ref{fig: Markov chain on N}. Note that $\tau(n_i)$ gives the occupation at the site $i-1$ and consequently $\omega'_i=\tau(n_{i+1})$ for all $i\in \Z$.

\begin{figure}[ht!]
\centering
\scalebox{0.5}{
\begin{tikzpicture}[every label/.append style={scale=1.3}]
    \node[shape=circle,draw=cyan, fill=cyan, minimum size=0.5cm] (-B) at (-1,0) {};
    \node[shape=circle,draw=cyan, fill=cyan, minimum size=0.5cm] (-C) at (-2,0) {};
    \node[shape=circle,draw=cyan, fill=cyan, minimum size=0.5cm] (-D) at (-3,0) {};
    \node[shape=circle,draw=cyan, minimum size=0.5cm] (-E) at (-4,0) {};
    \node[shape=circle,draw=cyan,  minimum size=0.5cm] (-F) at (-5,0) {};
    \node[shape=circle,draw=cyan, minimum size=0.5cm] (-G) at (-6,0) {}; \node[shape=circle,draw=cyan, minimum size=0.5cm] (-H) at (-7,0) {};
    \node[shape=circle,draw=cyan, minimum size=0.5cm] (-I) at (-8,0) {};
    \node[shape=circle,draw=cyan, minimum size=0.5cm] (-J) at (-9,0) {};
    \node[shape=circle,draw=cyan,  minimum size=0.5cm] (-K) at (-10,0) {};
    \node[shape=circle,draw=cyan, minimum size=0.5cm] (-L) at (-11,0) {};
    \node[shape=circle,draw=cyan, fill=cyan, minimum size=0.5cm] (-M) at (-12,0) {};
    \node[shape=circle,draw=cyan, fill=cyan, minimum size=0.5cm] (-O) at (-13,0) {};
    \node[shape=circle,draw=cyan, minimum size=0.5cm] (A) at (0,0) {};
    \node[shape=circle,draw=cyan, minimum size=0.5cm] (B) at (1,0) {};
    \node[shape=circle,draw=cyan,minimum size=0.5cm] (C) at (2,0) {};
    \node[shape=circle,draw=cyan, fill=cyan, minimum size=0.5cm] (D) at (3,0) {};
    \node[shape=circle,draw=cyan, fill=cyan, minimum size=0.5cm] (E) at (4,0) {};
    \node[shape=circle,draw=cyan, fill=cyan,  minimum size=0.5cm] (F) at (5,0) {};
    \node[shape=circle,draw=cyan, fill=cyan, minimum size=0.5cm] (G) at (6,0) {};
    \node[shape=circle,draw=cyan, minimum size=0.5cm] (H) at (7,0) {};
    \node[shape=circle,draw=cyan, minimum size=0.5cm] (I) at (8,0) {};
    \node[shape=circle,draw=cyan, minimum size=0.5cm] (J) at (9,0) {};
    \node[shape=circle,draw=cyan,  minimum size=0.5cm] (K) at (10,0) {};
    \node[shape=circle,draw=cyan, fill=cyan,  minimum size=0.5cm] (L) at (11,0) {};
    \node[shape=circle,draw=cyan, fill=cyan, minimum size=0.5cm] (M) at (12,0) {};
    \node[shape=circle,draw=cyan, fill=cyan, minimum size=0.5cm] (O) at (13,0) {};

     \node[minimum size=0.5, label=:{\Large 1}] () at (0,-1.3) {};
     \node[minimum size=0.5, label=:{\Large 2}] () at (1,-1.3) {};
     \node[minimum size=0.5, label=:{\Large 3}] () at (2,-1.3) {};
     \node[minimum size=0.5, label=:{\Large 4}] () at (3,-1.3) {};
     \node[minimum size=0.5, label=:{\Large 0}] () at (4,-1.3) {};

\node[minimum size=0.5, label=:{\Large 1}] () at (5,-1.3) {};
\node[minimum size=0.5, label=:{\Large 1}] () at (6,-1.3) {};
\node[minimum size=0.5, label=:{\Large 1}] () at (7,-1.3) {};
\node[minimum size=0.5, label=:{\Large 2}] () at (8,-1.3) {};
\node[minimum size=0.5, label=:{\Large 3}] () at (9,-1.3) {};
\node[minimum size=0.5, label=:{\Large 4}] () at (10,-1.3) {};
\node[minimum size=0.5, label=:{\Large 5}] () at (11,-1.3) {};
\node[minimum size=0.5, label=:{\Large 0}] () at (12,-1.3) {};
\node[minimum size=0.5, label=:{\Large 1}] () at (13,-1.3) {};

\node[minimum size=0.5, label=:{\Large 1}] () at (-1,-1.3) {};
\node[minimum size=0.5, label=:{\Large 0}] () at (-2,-1.3) {};
     \node[minimum size=0.5, label=:{\Large 9}] () at (-3,-1.3) {};
     \node[minimum size=0.5, label=:{\Large 8}] () at (-4,-1.3) {};

\node[minimum size=0.5, label=:{\Large 7}] () at (-5,-1.3) {};
\node[minimum size=0.5, label=:{\Large 6}] () at (-6,-1.3) {};
\node[minimum size=0.5, label=:{\Large 5}] () at (-7,-1.3) {};
\node[minimum size=0.5, label=:{\Large 4}] () at (-8,-1.3) {};
\node[minimum size=0.5, label=:{\Large 3}] () at (-9,-1.3) {};
\node[minimum size=0.5, label=:{\Large 2}] () at (-10,-1.3) {};
\node[minimum size=0.5, label=:{\Large 1}] () at (-11,-1.3) {};
\node[minimum size=0.5, label=:{\Large \dots}] () at (-12,-1.3) {};

    \draw[-] (-O)--(-13.7,0){};
    \draw [-] (-O) -- (-M);
     \draw [-] (-L) -- (-M);
     \draw [-] (-L) -- (-K);
    \draw [-] (-J) -- (-K);
    \draw [-] (-I) -- (-J);
    \draw [-] (-H) -- (-I);
    \draw [-] (-G) -- (-H);
    \draw [-] (-F) -- (-G);
    \draw [-] (-F) -- (-E);
    \draw [-] (-E) -- (-D);
    \draw [-] (-D) -- (-C);
    \draw [-] (-C) -- (-B);
    \draw [-] (-B) -- (A);
    \draw [-] (A) -- (B);
    \draw[-](O)--(13.7,0){};
    \draw [-] (M) -- (O);
    \draw [-] (M) -- (L);
    \draw [-] (K) -- (L);
    \draw [-] (K) -- (J);
    \draw [-] (J) -- (I);
    \draw [-] (H) -- (I);
    \draw [-] (G) -- (H);
    \draw [-] (G) -- (F);
    \draw [-] (F) -- (E);
    \draw [-] (E) -- (D);
    \draw [-] (D) -- (C);
    \draw [-] (C) -- (B);

\end{tikzpicture}
}
 \small \caption{A non-isolated configuration $\omega'$ on the line, highlighted in blue, together with its corresponding GHoC path $(n_i)_{i\in \Z}$ in $\N_0\cup \{\infty\}$, where $n_i$ is indicated below each site $i$. For each site $i\in \Z$, the state shown below indicates the distance $n_i$ to the next pair of occupied sites in its past. Note that, if the configuration in the past ends with 
the pair "empty followed by occupied", the distance to the next 
double one is put to be $0$, as the example shows in the middle 
of the triplet of occupied sites. 
  }
\label{fig: D2D1 sequence}
\end{figure}
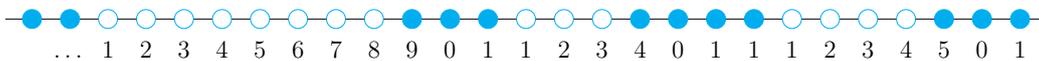

We define the \textit{generalized house of cards (GHoC) process} $(X_n)_{n\in \Z}$ on the state space $\N_0\cup \{\infty\}$ as a Markov chain with transition probabilities $\Pi_p=(\Pi_p(i,j))_{i,j\in \N_0\cup \{\infty\}}$, given by
\begin{equation}\label{eq: Entries transfer matrix Markov chain}
    \Pi_p(i,j):=\begin{cases}
        g_p(i)~&\text{if}~j=i+1,\\
        1-g_p(i)~&\text{if}~j=0~\text{and}~i\neq 1,\\
        p~&\text{if}~i=j=1,\\
        0~&\text{else}
    \end{cases}
\end{equation}
for all $i,j\in \N_0\cup \{\infty\}$. For an illustration of the transition graph, together with the map $\tau$ to the TBF, see Figure \ref{fig: Markov chain on N}.

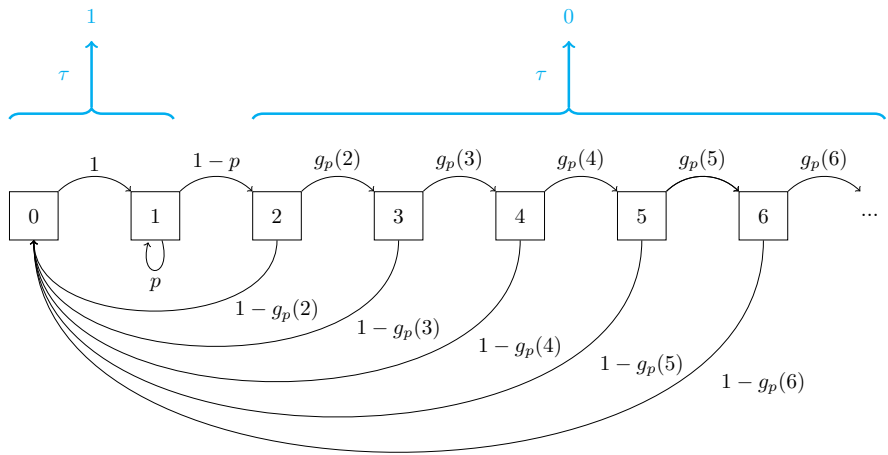
\begin{figure}[ht!]
\centering
\scalebox{0.8}{
\begin{tikzpicture}[every label/.append style={scale=1}]
 \node[shape=rectangle,draw=black,  minimum size=0.8cm] (0) at (0,0) {0};
 \node[shape=rectangle,draw=black,  minimum size=0.8cm] (1) at (2,0) {1};
 \node[shape=rectangle,draw=black,  minimum size=0.8cm] (2) at (4,0) {2};
\node[shape=rectangle,draw=black,  minimum size=0.8cm] (3) at (6,0) {3};
\node[shape=rectangle,draw=black,  minimum size=0.8cm] (4) at (8,0) {4};
\node[shape=rectangle,draw=black,  minimum size=0.8cm] (5) at (10,0) {5};
\node[shape=rectangle,draw=black,  minimum size=0.8cm] (6) at (12,0) {6};
\node[shape=rectangle,  minimum size=0.8cm] (7) at (14,0) {};
 \node[minimum size=0.5, label=:{\text{...}}] () at (13.75,-0.25) {};

\draw[->,bend angle=45, bend left] (0) to (1);
\draw[->,bend angle=45, bend left] (1) to (2);
\draw[->,bend angle=45, bend left] (2) to (3);
\draw[->,bend angle=45, bend left] (3) to (4);
\draw[->,bend angle=45, bend left] (4) to (5);
\draw[->,bend angle=45, bend left] (5) to (6);
\draw[->,bend angle=45, bend left] (5) to (6);
\draw[->,bend angle=45, bend left] (6) to (7);

\draw[->,bend angle=90, bend left] (2) to (0);
\draw[->,bend angle=90, bend left] (3) to (0);
\draw[->,bend angle=90, bend left] (4) to (0);
\draw[->,bend angle=90, bend left] (5) to (0);
\draw[->,bend angle=90, bend left] (6) to (0);

\path[->] (1) edge  [loop below] node {$p$} ();

\node[minimum size=0.5, label=:{$1$}] () at (1,0.5) {};
\node[minimum size=0.5, label=:{$1-p$}] () at (3,0.5) {};
\node[minimum size=0.5, label=:{$g_p(2)$}] () at (5,0.5) {};
\node[minimum size=0.5, label=:{$g_p(3)$}] () at (7,0.5) {};
\node[minimum size=0.5, label=:{$g_p(4)$}] () at (9,0.5) {};
\node[minimum size=0.5, label=:{$g_p(5)$}] () at (11,0.5) {};
\node[minimum size=0.5, label=:{$g_p(6)$}] () at (13,0.5) {};

\node[minimum size=0.5, label=:{$1-g_p(2)$}] () at (4,-2) {};
\node[minimum size=0.5, label=:{$1-g_p(3)$}] () at (6,-2.3) {};
\node[minimum size=0.5, label=:{$1-g_p(4)$}] () at (8,-2.6) {};
\node[minimum size=0.5, label=:{$1-g_p(5)$}] () at (10,-2.9) {};
\node[minimum size=0.5, label=:{$1-g_p(6)$}] () at (12,-3.2) {};

    \draw [decorate,very thick,decoration={brace,amplitude=5pt,mirror,raise=4ex},cyan]
  (2.3,1) -- (-0.4,1) node[midway,yshift=-3em]{};
  \draw [decorate,very thick,decoration={brace,amplitude=5pt,mirror,raise=4ex},cyan]
  (14,1) -- (3.6,1) node[midway,yshift=-3em]{};
\draw[->,very thick, cyan] (0.95,1.8) to (0.95,2.9)  node [text width=2.5cm,midway,above=0.5cm, align=center] {};

\draw[->,very thick, cyan] (8.8,1.8) to (8.8,2.9)  node [text width=2.5cm,midway,above=0.5cm, align=center] {};

\node[minimum size=0.5, label=:{{\color{cyan}$1$}}] () at (0.95,2.95) {};
\node[minimum size=0.5, label=:{{\color{cyan}$\tau$}}] () at (0.5,2) {};

\node[minimum size=0.5, label=:{{\color{cyan}$0$}}] () at (8.8,2.95) {};
\node[minimum size=0.5, label=:{{\color{cyan}$\tau$}}] () at (8.35,2) {};

\end{tikzpicture}
}
\small \caption{Depicted is the GHoC Markov process with state space $\N_0\cup \{\infty\}$ for the first 7 states. The transition probabilities are given by the g-function $g_p$ of the TBF. Furthermore, the map $\tau$ is depicted in blue which provides a one-to-one correspondence between a non-isolated configuration $\omega'$ and its GHoC path $(n_i)_{i\in \Z}$ taking values in $\N_0\cup\{\infty\}$. Note that an example of this correspondence is given in Figure \ref{fig: D2D1 sequence}.}
\label{fig: Markov chain on N}
\end{figure}

\begin{theorem}\label{thm: Relation TBF and GHoC}
    Let $p\in (0,1)$. The GHoC Markov process on $\Z$ with state space $\N_0\cup \{\infty\}$ and transition probabilities given in \eqref{eq: Entries transfer matrix Markov chain} possesses a unique stationary distribution on path space, denoted by $\mathds{P}_p^{GHoC}$. Furthermore, the TBF is a local deterministic map of a GHoC Markov process in the sense that 
    \begin{equation*}
        \tau\Big(\mathds{P}_p^{GHoC}\Big)=\mu_p'
    \end{equation*}
    where $\tau$ was defined in \eqref{eq: Definition map Tau}. More precisely, let $(n_i)_{i\in \Z}$ denote a path of the GHoC Markov chain, i.e. $(n_i)_{i\in \Z}\sim \mathds{P}_p^{GHoC}$. Then the random sequence $\big(\tau(n_i)\big)_{i\in \Z}$ is distributed according to $\mu_p'$.
\end{theorem}

From this relation, we obtain an alternative proof of the uniqueness of the $g$-measure, albeit assuming translation invariance.

\begin{remark}\label{rk: GHoC implies unique g-measure}
    Note that the bijection on path level 
between TBF and GHoC provides an alternative 
route to deduce uniqueness of the g-measure (that is the measure invariant 
under the $g$-function, which is assumed to be translation invariant), 
which only uses Foster's criterion for the GHoC. 
\end{remark}
The proofs of Theorem \ref{thm: Relation TBF and GHoC} and Remark \ref{rk: GHoC implies unique g-measure} are presented in Section \ref{Sec: Proofs one-sided view}.

\section{Proofs for the two-sided view} \label{Sec: Proofs two-sided view}

\subsection{Theorem \ref{thm: Existence image specification}: Computation of a specification $\gamma_p'$ for the TBF}\label{Subsec: The image specification}

The image process of the i.i.d. Bernoulli field on the line under the transformation $T$, see \eqref{eq: Map T} and Figure \ref{fig: Transformation T}, is strongly related to the Bernoulli field constrained on isolation. More precisely, it appears as a hidden process within the unfixed areas of a second-layer configuration $\omega'\in \Omega'$, see Remark \ref{rk: Fixed area} and Figure \ref{fig: Fixed area}. Therefore, we will introduce its corresponding symmetric transfer matrix
\begin{equation}\label{eq: Transfer matrix Q}
    Q=\begin{pmatrix}
        Q(0,0)&Q(0,1)\\
        Q(1,0)&Q(1,1)
    \end{pmatrix}
    :=\begin{pmatrix}
        1-p&\sqrt{p(1-p)}\\
        \sqrt{p(1-p)}&0
    \end{pmatrix}
\end{equation}
for all $p\in (0,1)$ which possesses the eigenvalues
\begin{equation}\label{eq: Eigenvalues of Q}
    \begin{split}
        \lambda_{\text{r}}&=\lambda_{\text{r}}(p):=\frac{1}{2}\bigg(1-p-\sqrt{(1-p)(3p+1)}\bigg)\in \Big(-\frac{1}{3},0\Big),\\
    \lambda_{\text{PF}}&=\lambda_{\text{PF}}(p):=\frac{1}{2}\bigg(1-p+\sqrt{(1-p)(3p+1)}\bigg)\in (0,1)
    \end{split}
\end{equation}
where $\lambda_{\text{PF}}$ is the Perron-Frobenius eigenvalue. An important quantity in further results and computations will be the \textit{eigenvalue ratio}
\begin{align}\label{eq: Fraction eigenvalues}
    a=a(p):=\frac{\lambda_r(p)}{\lambda_{\text{PF}}(p)}\in (-1,0)
\end{align}
for all $p\in(0,1)$. For the plot of $a(p)$ as a function of $p$, see Figure \ref{fig: Speed of convergence and plot g-function} (a).

In order to construct the specification $\gamma'_p$, we need to investigate in which way the spins of the given boundary condition may influence the spins in the finite observation set $\Lambda$.

\begin{definition}\label{def: restricted unfixed areas}
    Let $\Lambda\Subset \Z$ and $\omega'\in \Omega'$. We introduce the set 
    \begin{equation}\label{eq: restricted unfixed areas}
        \mathscr{U}_\Lambda=\mathscr{U}_\Lambda(\omega'):=\{U\in \mathscr{U}: ~\text{dist}(\overline{\Lambda},U)\leq 
        1\}
    \end{equation}
    for all $\Lambda\Subset \Z$ and write $U_-$ respectively $U_+$ for those connected components $U_-,U_+\in \mathscr{U}_\Lambda$ satisfying $\inf U_-<\min \Lambda$ respectively $\sup U_+>\max \Lambda$, in case such components exist. We denote by
    \begin{equation}\label{eq: Influence set}
        I=I_\Lambda(\omega'):=\overline{\Theta}_{\bar{\Lambda}}\cup\bigcup_{U\in \mathscr{U}_\Lambda}U
    \end{equation}
    the \textbf{influencing set} of $\Lambda$ for a given boundary condition $\omega'_{\Lambda^c}$.
\end{definition}

For an illustration of the influcencing set $I$ and the outer unfixed areas $U_-$ and $U_+$, see Figure \ref{fig: Existence of the limit}. Let us give some comments about the influcencing set defined above.

\begin{remark}\hfill\label{rk: Influence set}
\begin{enumerate}[label=\alph*)]
    \item Given a finite observation window $\Lambda$ and a boundary condition $\omega'_{\Lambda^c}$, the influencing set is defined to include all sites whose spins may affect those in $\Lambda$, regardless of the configuration inside $\Lambda$. Moreover, it is defined so that the spins in the first-layer of any configuration $\eta' \in \Omega'$ are fixed in $\partial_+ I_\Lambda(\omega')$ whenever $\eta'_{\Lambda^c}=\omega'_{\Lambda^c}$. This property is crucial for the proof of Theorem \ref{thm: Existence image specification}. By definition, $I_\Lambda(\omega')$ depends on $\omega'$, but it always contains $\overline{\Lambda}$ for all $\omega' \in \Omega'$.  
    \item The condition $\text{dist}(\overline{\Lambda}, U)\leq 1$ in the definition of $\mathscr{U}_\Lambda$ is introduced to ensure statement a). For instance, consider the spin at the site $7$ of Figure \ref{fig: Existence of the limit}. In the depicted configuration, the spin is fixed. However, if the spin at vertex $6$ were empty, the unfixed area $U_+$, consisting of the vertices $\{7,8,9\}$, would affect the spins in $\Lambda$.
    \item The sets $U_-$ and $U_+$ are the outer unfixed areas on the left respectively right side of $\Lambda$. These sets could be empty or possibly $U_-=U_+$.
\end{enumerate}
\end{remark}

\begin{figure}[ht!]
\centering
\scalebox{0.5}{
\begin{tikzpicture}[every label/.append style={scale=1.3},fill fraction/.style={path picture={
\fill[#1] 
(path picture bounding box.south) rectangle
(path picture bounding box.north west);
}},
fill fraction/.default=gray!50]
    \node[shape=circle,draw=orange, fill=orange, minimum size=0.5cm] (-B) at (-1,0) {};
    \node[shape=circle,draw=orange, fill=orange, minimum size=0.5cm] (-C) at (-2,0) {};
    \node[shape=circle,draw=orange, fill=orange, minimum size=0.5cm] (-D) at (-3,0) {};
    \node[shape=circle,draw=orange, minimum size=0.5cm] (-E) at (-4,0) {};
    \node[shape=circle,draw=orange,fill fraction=orange,  minimum size=0.5cm] (-F) at (-5,0) {};
    \node[shape=circle,draw=orange,fill fraction=orange, minimum size=0.5cm] (-G) at (-6,0) {}; \node[shape=circle,draw=orange,fill fraction=orange, minimum size=0.5cm] (-H) at (-7,0) {};
    \node[shape=circle,draw=orange,fill fraction=orange, minimum size=0.5cm] (-I) at (-8,0) {};
    \node[shape=circle,draw=orange,fill fraction=orange, minimum size=0.5cm] (-J) at (-9,0) {};
    \node[shape=circle,draw=orange,fill fraction=orange,  minimum size=0.5cm] (-K) at (-10,0) {};
    \node[shape=circle,draw=black, minimum size=0.5cm] (-L) at (-11,0) {};
    \node[shape=circle,draw=black, fill=black, minimum size=0.5cm] (-M) at (-12,0) {};
    \node[shape=circle,draw=black, fill=black, minimum size=0.5cm] (-O) at (-13,0) {};
    \node[shape=circle,draw=orange, minimum size=0.5cm, label=above:{$0$}] (A) at (0,0) {};
    \node[shape=circle,draw=orange,fill fraction=orange, minimum size=0.5cm] (B) at (1,0) {};
    \node[shape=circle,draw=orange,minimum size=0.5cm] (C) at (2,0) {};
    \node[shape=circle,draw=orange, fill=orange, minimum size=0.5cm] (D) at (3,0) {};
    \node[shape=circle,draw=orange, fill=orange, minimum size=0.5cm] (E) at (4,0) {};
    \node[shape=circle,draw=orange, fill=orange,  minimum size=0.5cm] (F) at (5,0) {};
    \node[shape=circle,draw=orange, fill=orange, minimum size=0.5cm] (G) at (6,0) {};
    \node[shape=circle,draw=orange, minimum size=0.5cm] (H) at (7,0) {};
    \node[shape=circle,draw=orange,fill fraction=orange, minimum size=0.5cm] (I) at (8,0) {};
    \node[shape=circle,draw=orange,fill fraction=orange, minimum size=0.5cm] (J) at (9,0) {};
    \node[shape=circle,draw=black,  minimum size=0.5cm] (K) at (10,0) {};
    \node[shape=circle,draw=black, fill=black,  minimum size=0.5cm] (L) at (11,0) {};
    \node[shape=circle,draw=black, fill=black, minimum size=0.5cm] (M) at (12,0) {};
    \node[shape=circle,draw=black, fill=black, minimum size=0.5cm] (O) at (13,0) {};

    \draw[-] (-O)--(-13.7,0){};
    \draw [-] (-O) -- (-M);
     \draw [-] (-L) -- (-M);
     \draw [-] (-L) -- (-K);
    \draw [-] (-J) -- (-K);
    \draw [-] (-I) -- (-J);
    \draw [-] (-H) -- (-I);
    \draw [-] (-G) -- (-H);
    \draw [-] (-F) -- (-G);
    \draw [-] (-F) -- (-E);
    \draw [-] (-E) -- (-D);
    \draw [-] (-D) -- (-C);
    \draw [-] (-C) -- (-B);
    \draw [-] (-B) -- (A);
    \draw [-] (A) -- (B);
    \draw[-](O)--(13.7,0){};
    \draw [-] (M) -- (O);
    \draw [-] (M) -- (L);
    \draw [-] (K) -- (L);
    \draw [-] (K) -- (J);
    \draw [-] (J) -- (I);
    \draw [-] (H) -- (I);
    \draw [-] (G) -- (H);
    \draw [-] (G) -- (F);
    \draw [-] (F) -- (E);
    \draw [-] (E) -- (D);
    \draw [-] (D) -- (C);
    \draw [-] (C) -- (B);

    \draw[-, black, very thick] (-7.5,-0.8) to (6.5,-0.8)  node [text width=2.5cm,midway,above=0.5cm, align=center] {};
    \draw[-, black, very thick] (-7.5,-0.8) to (-7.5,0)  node [text width=2.5cm,midway,above=0.5cm, align=center] {};
    \draw[-, black, very thick] (6.5,-0.8) to (6.5,0)  node [text width=2.5cm,midway,above=0.5cm, align=center] {};
    \scalebox{1.4}{\node[] () at (0,-1.0) {$\Lambda$};}

    \draw[-, cyan, very thick] (-10.5,0.8) to (-4.5,0.8)  node [text width=2.5cm,midway,above=0.5cm, align=center] {};
    \draw[-, cyan, very thick] (-10.5,0.8) to (-10.5,0)  node [text width=2.5cm,midway,above=0.5cm, align=center] {};
    \draw[-, cyan, very thick] (-4.5,0.8) to (-4.5,0)  node [text width=2.5cm,midway,above=0.5cm, align=center] {};
\scalebox{1.4}{\node[] () at (-5.5,0.9) {{\color{cyan}$U_-$}};}

     \draw[-, cyan, very thick] (7.5,0.8) to (9.5,0.8)  node [text width=2.5cm,midway,above=0.5cm, align=center] {};
    \draw[-, cyan, very thick] (7.5,0.8) to (7.5,0)  node [text width=2.5cm,midway,above=0.5cm, align=center] {};
    \draw[-, cyan, very thick] (9.5,0.8) to (9.5,0)  node [text width=2.5cm,midway,above=0.5cm, align=center] {};
    \scalebox{1.4}{\node[] () at (6.1,0.9) {{\color{cyan}$U_+$}};}
    
    \scalebox{1.4}{\node[] () at (3,0.8) {{\color{orange}$I_\Lambda$}};}

\end{tikzpicture}
}
\small \caption{A non-isolated configuration $\omega'$ on the line, with the subset $\Lambda:=[-7,6]$ indicated by a square bracket. Orange-colored sites denote the influencing set $I_\Lambda$. Blue-colored square brackets indicate the outer unfixed areas $U_-$ and $U_+$.}
\label{fig: Existence of the limit}
\end{figure}
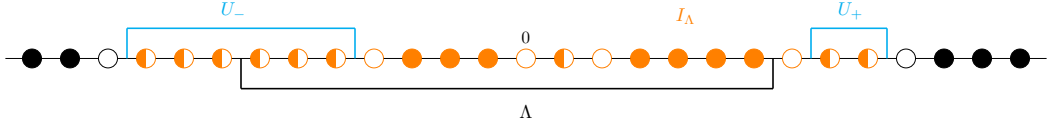

\begin{proof}[Proof of Theorem \ref{thm: Existence image specification}]
Let $\Lambda\Subset \Z$ and $\omega'\in \Omega'$. For the construction of a specification $\gamma'_p$ for $\mu_p'$, we will use an idea introduced in \cite{JaKu23}. Start to consider second-layer conditional probabilities in finite volumes $\Lambda\subset\Delta \Subset V$ with a first-layer boundary condition $\omega \in T^{-1}(\omega')$ defined as follows
\begin{equation}\label{eq: second-layer cond. prob}
    \gamma'_{p,\omega,\Delta}(\omega'_\Lambda|\omega'_{\Delta\setminus \Lambda}):= \frac{\sum_{\Tilde{\omega}_\Delta} \mu_p(\Tilde{\omega}_\Delta)\mathds{1}_{T_\Delta(\Tilde{\omega}_\Delta \omega_{\Delta^c})=\omega'_\Delta}}{\sum_{\hat{\omega}_{\Delta }}\mu_p(\hat{\omega}_{\Delta})\mathds{1}_{T_{\Delta \setminus \Lambda}(\hat{\omega}_{\Delta}\omega_{\Delta^c})=\omega'_{\Delta \setminus \Lambda}}}.
\end{equation}
Now, Lemma 3.3 in \cite{JaKu23} ensures that the family $\gamma'_p=(\gamma'_{p,\Lambda})_{\Lambda\Subset \Z}$ of kernels defined by
 \begin{equation}\label{eq: Limit second-layer cond prob}
    	\lim_{\Delta \uparrow V} \gamma'_{p,\omega,\Delta}(\omega'_\Lambda|\omega'_{\Delta \setminus \Lambda})=:\gamma'_{p,\Lambda}(\omega'_\Lambda|\omega'_{\Lambda^c})
    \end{equation}
    is a specification, provided thazt this limit exists and is independent of the choice $\omega\in T^{-1}(\omega')$.

In order to prove the existence of this limit, let us start to choose $\Delta\Subset \Z$ large enough such that $\text{dist}(\Lambda,\Delta^c)>1$ and without loss of generality, we assume that $\Delta$ is connected. The spins inside $\Lambda$ depend on those in $\Delta^c$ only through the outer unfixed areas $U_-$ or $U_+$, see Definition \ref{def: restricted unfixed areas}, which satisfy $\text{dist}(U_-,\Delta^c)\leq 1$ and $\text{dist}(U_+,\Delta^c)\leq 1$, respectively. Thus, we have to discuss three different cases for these unfixed areas.

     \textbf{Case 1.} ($|U_-|<\infty$ and $|U_+|<\infty$)
In this case, the influencing set $I$, given in \eqref{eq: Influence set}, is finite and thus choose $\Delta\subset \Z$ big enough such that $\text{dist}(\bar{I}, \Delta^c)\geq 1$. Consider the expression in \eqref{eq: second-layer cond. prob} and note that the spins of first-layer configurations $\Tilde{\omega},\hat{\omega}\in T^{-1}(\omega')$ are fixed in $\partial_+I$ due to Remark \ref{rk: Influence set}. Consequently, we can rewrite the conditional probability as follows
    \begin{align*}
        \gamma'_{p,\omega,\Delta}(\omega'_\Lambda|\omega'_{\Delta\setminus \Lambda})&= \frac{\sum_{\Tilde{\omega}_{I}} \mu_p(\Tilde{\omega}_{I})\mathds{1}_{T_{I}(\Tilde{\omega}_{I} \omega'_{I^c})=\omega'_{I}}}{\sum_{\hat{\omega}_{I }}\mu_p(\hat{\omega}_{I})\mathds{1}_{T_{I \setminus \Lambda}(\hat{\omega}_{I}\omega'_{I^c})=\omega'_{I \setminus \Lambda}}}\\ \nonumber
        &\times \frac{\sum_{\Tilde{\omega}_{\Delta \cap (\bar{I})^c}} \mu_p(\Tilde{\omega}_{\Delta \cap (\bar{I})^c})\mathds{1}_{T_{\Delta \cap (\bar{I})^c}(\hat{\omega}_{\Delta \cap (\bar{I})^c}\omega'_{\bar{I}} \omega_{\Delta^c})=\omega'_{\Delta \cap (\bar{I})^c}}}{\sum_{\hat{\omega}_{\Delta \cap (\bar{I})^c }}\mu_p(\hat{\omega}_{\Delta \cap (\bar{I})^c})\mathds{1}_{T_{(\Delta \cap (\bar{I})^c) \setminus \Lambda}(\hat{\omega}_{\Delta \cap (\bar{I})^c }\omega'_{\bar{I}}\omega_{\Delta^c})=\omega'_{(\Delta \cap (\bar{I})^c) \setminus \Lambda}}}
    \end{align*}
    because the map $T_\Lambda$ is $\mathscr{F}_{\Bar{\Lambda}}$-measurable for all $\Lambda\Subset \Z$. Furthermore, note that $(\Delta \cap (\bar{I})^c) \setminus \Lambda=(\Delta \cap (\bar{I})^c)$, since $\Lambda\subset I$, and hence the second factor equals one. The remaining term
    \begin{equation}\label{eq: Proof limit image spec 0}
        \frac{\sum_{\Tilde{\omega}_{I}} \mu_p(\Tilde{\omega}_{I})\mathds{1}_{T_{I}(\Tilde{\omega}_{I} \omega'_{I^c})=\omega'_{I}}}{\sum_{\hat{\omega}_{I }}\mu_p(\hat{\omega}_{I})\mathds{1}_{T_{I \setminus \Lambda}(\hat{\omega}_{I}\omega'_{I^c})=\omega'_{I \setminus \Lambda}}}
    \end{equation}
    is equal to $\gamma'_\Lambda(\omega'_\Lambda|\omega'_{\Lambda^c})$ as it is independent of $\Delta$ and $\omega$. Let us continue to prove the equality in \eqref{eq: Kernels image specification}. We define the denominator of \eqref{eq: Proof limit image spec 0}, depending only on $\omega'_{\Lambda^c}$, as the partition function $Z_{\Lambda}(\omega'_{\Lambda^c})$. By definition of the influencing set in \eqref{eq: Influence set}, the numerator can be rewritten as follows 
\begin{equation}\label{eq: Proof limit image spec 1}
    \sum_{\Tilde{\omega}_{\overline{\Theta}_{\bar{\Lambda}}}} \mu_p(\Tilde{\omega}_{\overline{\Theta}_{\bar{\Lambda}}})\mathds{1}_{T_{\overline{\Theta}_{\bar{\Lambda}}}(\Tilde{\omega}_{\overline{\Theta}_{\bar{\Lambda}}} \omega'_{(\overline{\Theta}_{\bar{\Lambda}})^c})=\omega'_{\overline{\Theta}_{\bar{\Lambda}}}}\cdot\prod_{U\in \mathscr{U}_\Lambda}\sum_{\Tilde{\omega}_{U}}\mu_p(\Tilde{\omega}_{U})\mathds{1}_{T_{U}(\Tilde{\omega}_{U} \omega'_{U^c})=0'_{U}}.
\end{equation}
Recall that the first-layer configurations $\Tilde{\omega}\in \Omega$ of $\omega'$ are fixed on $\overline{\Theta}_{\bar{\Lambda}}$ and consequently, the first sum is equal to $\mu_p(\omega'_{\overline{\Theta}_{\Lambda}})$, see Remark \ref{rk: Fixed area}. Furthermore, the probability in the unfixed areas is given by the transfer operator $Q$, defined in \eqref{eq: Transfer matrix Q}. Note that
\begin{equation}\label{eq: Relation Q and prob.}
    \sum_{\omega_U\in \Omega_U}\mu_p(\omega_U)\mathds{1}_{T_U(\omega_U\omega_{U^c})=0'_U}= \frac{Q^{|U|+1}(\omega_{\min \overline{U}},\omega_{\max \overline{U}})}{\sqrt{\mu_p(\omega_{\min \overline{U}})}\sqrt{\mu_p(\omega_{\max \overline{U}})}}
\end{equation}
holds for all $U\Subset \Z$ connected and $\omega\in \Omega$. The prefactor is due to the fact that we do not consider the probabilities in $\partial_+ U$ but their square root is contained in each matrix product. Using the relation \eqref{eq: Relation Q and prob.} with the fact that the unfixed areas are surrounded by fixed unoccupied sites, we have
\begin{equation*}
     \mu_p(\omega'_{\overline{\Theta}_{\bar{\Lambda}}})\cdot\prod_{U\in \mathscr{U}_\Lambda}\frac{Q^{|U|+1}(0,0)}{(1-p)}
\end{equation*}
for the expression in \eqref{eq: Proof limit image spec 1}. Rewriting the Bernoulli weights on $\overline{\Theta}_{\bar{\Lambda}}$ and expressing the matrix products in terms of the partition functions $Z(U)$ leads to the desired expression in \eqref{eq: Kernels image specification}.

    \textbf{Case 2.} ($\inf U_-=-\infty$ and $\sup U_+<\infty$ or $\inf U_->-\infty$ and $\sup U_+=\infty$)
    We verify the statement for the case $\inf U_-=-\infty$ and $\sup U_+<\infty$, first assuming that $U_- \neq U_+$. Choose $\Delta\Subset \Z$ large enough such that $\max \overline{I}\in  \Delta$. From the fact that the spins of first-layer configurations $\Tilde{\omega},\hat{\omega}\in T^{-1}(\omega')$ are fixed at $\max \bar{I}$, we obtain
    \begin{align*}
        \gamma'_{p,\omega,\Delta}(\omega'_\Lambda|\omega'_{\Delta\setminus \Lambda})&=
        \frac{\sum_{\Tilde{\omega}_{\Delta \cap I }} \mu_p(\Tilde{\omega}_{\Delta \cap I })\mathds{1}_{T_{\Delta \cap I }(\Tilde{\omega}_{\Delta \cap I } \omega'_{\Delta \cap I^c }\omega_{\Delta^c})=\omega'_{\Delta \cap I }}}{\sum_{\hat{\omega}_{\Delta \cap I  }}\mu_p(\hat{\omega}_{\Delta \cap I })\mathds{1}_{T_{(\Delta \cap I ) \setminus \Lambda}(\hat{\omega}_{\Delta \cap I }\omega'_{\Delta \cap I^c }\omega_{\Delta^c})=\omega'_{(\Delta \cap I ) \setminus \Lambda}}}\\ \nonumber
        &\times \frac{\sum_{\Tilde{\omega}_{\Delta \cap (\bar{I})^c}} \mu_p(\Tilde{\omega}_{\Delta \cap (\bar{I})^c})\mathds{1}_{T_{\Delta \cap (\bar{I})^c}(\hat{\omega}_{\Delta \cap (\bar{I})^c}\omega'_{\Delta\cap\bar{I}} \omega_{\Delta^c})=\omega'_{\Delta \cap (\bar{I})^c}}}{\sum_{\hat{\omega}_{\Delta \cap (\bar{I})^c }}\mu_p(\hat{\omega}_{\Delta \cap (\bar{I})^c})\mathds{1}_{T_{(\Delta \cap (\bar{I})^c) \setminus \Lambda}(\hat{\omega}_{\Delta \cap (\bar{I})^c }\omega'_{\Delta\cap \bar{I}}\omega_{\Delta^c})=\omega'_{(\Delta \cap (\bar{I})^c) \setminus \Lambda}}}
    \end{align*}
    for the probability in \eqref{eq: second-layer cond. prob}. Using $(\Delta \cap (\bar{I})^c) \setminus \Lambda=\Delta \cap (\bar{I})^c$, the second factor is equal to one. Separating the probabilities corresponding to $\Delta \cap U_-$ and $\Delta \cap \Lambda_-$ in the numerator and denominator, respectively, leads to
    \begin{align}\label{eq: Proof limit image spec 2}
        \nonumber\gamma'_{p,\omega,\Delta}(\omega'_\Lambda|\omega'_{\Delta\setminus \Lambda})&=
        \frac{\sum_{\Tilde{\omega}_{\Delta \cap U_- }} \mu_p(\Tilde{\omega}_{\Delta \cap U_- })\mathds{1}_{T_{\Delta \cap U_- }(\Tilde{\omega}_{\Delta \cap U_- } \omega'_{\Delta \cap (U_-)^c }\omega_{\Delta^c})=0'_{\Delta \cap U_- }}}{\sum_{\hat{\omega}_{\Delta \cap \Lambda_-  }}\mu_p(\hat{\omega}_{\Delta \cap \Lambda_- })\mathds{1}_{T_{(\Delta \cap \Lambda_- ) \setminus \Lambda}(\hat{\omega}_{\Delta \cap \Lambda_- }\omega'_{\Delta \cap (\Lambda_-)^c }\omega_{\Delta^c})=0'_{(\Delta \cap \Lambda_- ) \setminus \Lambda}}}\\ 
        &\times\frac{\sum_{\Tilde{\omega}_{ I\setminus U_- }} \mu_p(\Tilde{\omega}_{ I\setminus U_- })\mathds{1}_{T_{ I \setminus U_- }(\Tilde{\omega}_{ I \setminus U_- } \omega'_{(I \setminus U_- )^c})=\omega'_{I \setminus U_-}}}{\sum_{\hat{\omega}_{I \setminus \Lambda_-  }}\mu_p(\hat{\omega}_{I \setminus \Lambda_- })\mathds{1}_{T_{(I \setminus \Lambda_- ) \setminus \Lambda}(\hat{\omega}_{I \setminus \Lambda_-}\omega'_{(I \setminus \Lambda_-)^c})=\omega'_{(I \setminus \Lambda_-)  \setminus \Lambda}}}.
    \end{align}
     With the relation \eqref{eq: Relation Q and prob.}, we can express the first fraction as follows
    \begin{equation*}
        \frac{(1-p)^{-\frac{1}{2}}Q^{|\Delta\cap U_-|+1}(\omega_{\min \overline{\Delta}},0)}{\sum_{x\in \{0,1\}}\sqrt{\alpha(x)}Q^{|\Delta\cap \Lambda_-|}(\omega_{\min \overline{\Delta}},x)}
    \end{equation*}
    where $\alpha(x):=p^x(1-p)^{1-x}$ for $x\in \{0,1\}$. This sequence converges to
    \begin{equation*}
        \frac{\lambda_{\text{PF}}^{|V_-\cap U_-|}}{(1-p)\sqrt{p}\sum_{x\in \{0,1\}}\sqrt{\alpha(x)}v_{\text{PF}}(x)}
    \end{equation*}
    for $\Delta \uparrow \Z$ independent of $\omega\in T^{-1}(\omega')$, see Lemma \ref{lem: Limit Matrix Q}. Here, we exploited the fact that $|\Delta\cap U_-|-|V_-\cap U_-|=|\Delta\cap \Lambda_-|-2$. Therefore, define 
    \begin{equation*}
        Z_\Lambda(\omega'_{\Lambda^c}):=\sqrt{p}\sum_{x\in \{0,1\}}\sqrt{\alpha(x)}v_{\text{PF}}(x)\cdot\sum_{\hat{\omega}_{I \setminus \Lambda_-  }}\mu_p(\hat{\omega}_{I \setminus \Lambda_- })\mathds{1}_{T_{(I \setminus \Lambda_- ) \setminus \Lambda}(\hat{\omega}_{I \setminus \Lambda_-}\omega'_{(I \setminus \Lambda_-)^c})=\omega'_{(I \setminus \Lambda_-)  \setminus \Lambda}}
    \end{equation*}
    and the expression in \eqref{eq: Proof limit image spec 2} reads
    \begin{equation}\label{eq: Proof limit image spec 3}
        Z_\Lambda(\omega'_{\Lambda^c})^{-1}(1-p)^{-1}\lambda_{\text{PF}}^{|V_-\cap U_-|}\cdot\sum_{\Tilde{\omega}_{ I\setminus U_- }} \mu_p(\Tilde{\omega}_{ I\setminus U_- })\mathds{1}_{T_{ I \setminus U_- }(\Tilde{\omega}_{ I \setminus U_- } \omega'_{(I \setminus U_- )^c})=\omega'_{I \setminus U_-}}.
    \end{equation}
    Recalling the definition of the influencing set in \eqref{eq: Influence set} and repeating the arguments that led to \eqref{eq: Kernels image specification} in Case 1, we arrive at
    \begin{equation*}
    \sum_{\Tilde{\omega}_{ I\setminus U_- }} \mu_p(\Tilde{\omega}_{ I\setminus U_- })\mathds{1}_{T_{ I \setminus U_- }(\Tilde{\omega}_{ I \setminus U_- } \omega'_{(I \setminus U_- )^c})=\omega'_{I \setminus U_-}}=\mu_p(\omega'_{\overline{\Theta}_{\bar{\Lambda}}})\cdot\prod_{U\in \mathscr{U}_\Lambda\setminus \{U_-\}}\frac{Q^{|U|+1}(0,0)}{(1-p)}.
    \end{equation*}
    This directly implies the stated expression for $\gamma'_{p,\Lambda}$ in \eqref{eq: Kernels image specification}. One can proceed similarly in the case $U_-=U_+$. Here, $I=U_-$, and the numerator of the second fraction in \eqref{eq: Proof limit image spec 2} equals one, which is the only difference in this case.

\textbf{Case 3.} ($\inf U_-=-\infty$ and $\sup U_+=\infty$)
We restrict our discussion to the case $|\Lambda|\neq 1$. The case $|\Lambda|=1$ is trivial as the conditional probabilities in \eqref{eq: second-layer cond. prob} equal one for sufficiently large $\Delta$ and for all non-isolated configurations $\omega'\in \Omega'$. We start to consider the case $U_-\neq U_+$. Here, the conditional probabilities can be written as
    \begin{align}\label{eq: Proof limit image spec 4}
        \gamma'_{p,\omega,\Delta}(\omega'_\Lambda&|\omega'_{\Delta\setminus \Lambda})=
        \frac{\sum_{\Tilde{\omega}_{\Delta \cap U_- }} \mu_p(\Tilde{\omega}_{\Delta \cap U_- })\mathds{1}_{T_{\Delta \cap U_- }(\Tilde{\omega}_{\Delta \cap U_- } \omega'_{\Delta \cap (U_-)^c }\omega_{\Delta^c})=0'_{\Delta \cap U_- }}}{\sum_{\hat{\omega}_{\Delta \cap \Lambda_-  }}\mu_p(\hat{\omega}_{\Delta \cap \Lambda_- })\mathds{1}_{T_{(\Delta \cap \Lambda_- ) \setminus \Lambda}(\hat{\omega}_{\Delta \cap \Lambda_- }\omega'_{\Delta \cap (\Lambda_-)^c }\omega_{\Delta^c})=0'_{(\Delta \cap \Lambda_- ) \setminus \Lambda}}}\\ \nonumber
        \nonumber&\times \frac{\sum_{\Tilde{\omega}_{\Delta \cap U_+ }} \mu_p(\Tilde{\omega}_{\Delta \cap U_+ })\mathds{1}_{T_{\Delta \cap U_+ }(\Tilde{\omega}_{\Delta \cap U_+ } \omega'_{\Delta \cap (U_+)^c }\omega_{\Delta^c})=0'_{\Delta \cap U_+ }}}{\sum_{\hat{\omega}_{\Delta \cap \Lambda_+  }}\mu_p(\hat{\omega}_{\Delta \cap \Lambda_+ })\mathds{1}_{T_{(\Delta \cap \Lambda_+ ) \setminus \Lambda}(\hat{\omega}_{\Delta \cap \Lambda_+ }\omega'_{\Delta \cap (\Lambda_+)^c }\omega_{\Delta^c})=0'_{(\Delta \cap \Lambda_+ ) \setminus \Lambda}}}\\ \nonumber
        \nonumber&\times\frac{\sum_{\Tilde{\omega}_{ I\setminus (U_-\cup U_+) }} \mu_p(\Tilde{\omega}_{ I\setminus (U_- \cup U_+) })\mathds{1}_{T_{ I \setminus (U_- \cup U_+) }(\Tilde{\omega}_{ I \setminus (U_- \cup U_+) } \omega'_{(I \setminus (U_- \cup U_+) )^c})=\omega'_{I \setminus (U_- \cup U_+)}}}{\sum_{\hat{\omega}_{I \setminus (\Lambda_- \cup \Lambda_+)  }}\mu_p(\hat{\omega}_{I \setminus (\Lambda_- \cup \Lambda_+) })\mathds{1}_{T_{(I \setminus (\Lambda_- \cup \Lambda_+) ) \setminus \Lambda}(\hat{\omega}_{I \setminus (\Lambda_- \cup \Lambda_+)}\omega'_{(I \setminus (\Lambda_- \cup \Lambda_+))^c})=\omega'_{(I \setminus (\Lambda_- \cup \Lambda_+))  \setminus \Lambda}}}
    \end{align}
    by separating the probabilities corresponding to $\Delta \cap U_-$, $\Delta \cap \Lambda_-$, $\Delta \cap U_+$ and $\Delta \cap \Lambda_+$. The first two fractions can be expressed in terms of the transfer matrix $Q$ as follows
        \begin{equation*}
        \frac{(1-p)^{-1/2}Q^{|\Delta\cap U_-|+1}(\omega_{\min \overline{\Delta}},0)}{\sum_{x\in \{0,1\}}\sqrt{\alpha(x)}Q^{|\Delta\cap \Lambda_-|}(\omega_{\min \overline{\Delta}},x)}\cdot\frac{(1-p)^{-1/2}Q^{|\Delta\cap U_+|+1}(0,\omega_{\max \overline{\Delta}})}{\sum_{y\in \{0,1\}}\sqrt{\alpha(y)}Q^{|\Delta\cap \Lambda_+|}(y,\omega_{\max \overline{\Delta}})},
    \end{equation*}
    see relation \eqref{eq: Relation Q and prob.}. If we take the limit $\Delta\uparrow \Z$ and apply Lemma  \ref{lem: Limit Matrix Q}, we obtain 
    \begin{equation*}
        \frac{(\lambda_{\text{PF}})^{|V_-\cap U_-|}}{(1-p)\sqrt{p}\sum_{x\in \{0,1\}}\sqrt{\alpha(x)}v_{\text{PF}}(x)}\cdot \frac{(\lambda_{\text{PF}})^{|V_+\cap U_+|}}{(1-p)\sqrt{p}\sum_{y\in \{0,1\}}\sqrt{\alpha(y)}v_{\text{PF}}(y)}
    \end{equation*}
    independent of the first-layer boundary condition $\omega\in T^{-1}(\omega')$. Choosing 
\begin{align*}
    Z_\Lambda(\omega'_{\Lambda^c})&:=p\sum_{x\in \{0,1\}}\sqrt{\alpha(x)}v_{\text{PF}}(x)\sum_{y\in \{0,1\}}\sqrt{\alpha(y)}v_{\text{PF}}(y)\\
    &\times\sum_{\hat{\omega}_{I \setminus (V_- \cup V_+)  }}\mu_p(\hat{\omega}_{I \setminus (V_- \cup V_+) })\mathds{1}_{T_{(I \setminus (V_- \cup V_+) ) \setminus \Lambda}(\hat{\omega}_{I \setminus (V_- \cup V_+)}\omega'_{(I \setminus (V_- \cup V_+))^c})=\omega'_{(I \setminus (V_- \cup V_+))  \setminus \Lambda}},
\end{align*}
we can rewrite \eqref{eq: Proof limit image spec 4} as follows
\begin{align}\label{eq: Proof limit image spec 5}
    \nonumber&Z_\Lambda(\omega'_{\Lambda^c})^{-1}(1-p)^{-2}(\lambda_{\text{PF}})^{|V_-\cap U_-|+|V_+\cap U_+|}\\
    &\times\sum_{\Tilde{\omega}_{ I\setminus (U_-\cup U_+) }} \mu_p(\Tilde{\omega}_{ I\setminus (U_- \cup U_+) })\mathds{1}_{T_{ I \setminus (U_- \cup U_+) }(\Tilde{\omega}_{ I \setminus (U_- \cup U_+) } \omega'_{(I \setminus (U_- \cup U_+) )^c})=\omega'_{I \setminus (U_- \cup U_+)}}.
\end{align}
The last sum becomes $\mu_p(\omega'_{\overline{\Theta}_{\bar{\Lambda}}})\prod_{U\in \mathscr{U}_\Lambda\setminus \{U_-,U_+\}}Q^{|U|+1}(0,0)/(1-p)$ by rewriting it with the definition of the influencing set. Hence, the statement of \eqref{eq: Kernels image specification} follows.

If $U_-=U_+=\Z$, we have
    \begin{align*}
        \gamma'_{p,\omega,\Delta}(\omega'_\Lambda&|\omega'_{\Delta\setminus \Lambda})=
        \frac{\sum_{\Tilde{\omega}_{\Delta}} \mu_p(\Tilde{\omega}_{\Delta})\mathds{1}_{T_{\Delta}(\Tilde{\omega}_{\Delta} \omega_{\Delta^c})=0'_{\Delta  }}}{\sum_{\hat{\omega}_{\Delta \cap \Lambda_-  }}\mu_p(\hat{\omega}_{\Delta \cap \Lambda_- })\mathds{1}_{T_{(\Delta \cap \Lambda_- ) \setminus \Lambda}(\hat{\omega}_{\Delta \cap \Lambda_- }\omega'_{\Delta \cap (\Lambda_-)^c }\omega_{\Delta^c})=0'_{(\Delta \cap \Lambda_- ) \setminus \Lambda}}}\\ \nonumber
        &\times \frac{1}{\sum_{\hat{\omega}_{\Delta \cap \Lambda_+  }}\mu_p(\hat{\omega}_{\Delta \cap \Lambda_+ })\mathds{1}_{T_{(\Delta \cap \Lambda_+ ) \setminus \Lambda}(\hat{\omega}_{\Delta \cap \Lambda_+ }\omega'_{\Delta \cap (\Lambda_+)^c }\omega_{\Delta^c})=0'_{(\Delta \cap \Lambda_+ ) \setminus \Lambda}}}\\ \nonumber
        &\times\frac{1}{\sum_{\hat{\omega}_{I \setminus (\Lambda_- \cup \Lambda_+)  }}\mu_p(\hat{\omega}_{I \setminus (\Lambda_- \cup \Lambda_+) })\mathds{1}_{T_{(I \setminus (\Lambda_- \cup \Lambda_+) ) \setminus \Lambda}(\hat{\omega}_{I \setminus (\Lambda_- \cup \Lambda_+)}\omega'_{(I \setminus (\Lambda_- \cup \Lambda_+))^c})=\omega'_{(I \setminus (\Lambda_- \cup \Lambda_+))  \setminus \Lambda}}}.
    \end{align*}
    The first two fractions can be expressed in terms of the transfer matrix $Q$ as follows
        \begin{equation*}
        \frac{Q^{|\Delta|+1}(\omega_{\min \overline{\Delta}},\omega_{\max \overline{\Delta}})}{\sum_{x\in \{0,1\}}\sqrt{\alpha(x)}Q^{|\Delta\cap \Lambda_-|}(\omega_{\min \overline{\Delta}},x)\cdot\sum_{y\in \{0,1\}}\sqrt{\alpha(y)}Q^{|\Delta\cap \Lambda_+|}(y,\omega_{\max \overline{\Delta}})}.
    \end{equation*}
    Taking the cofinal limit $\Delta\uparrow \Z$ leads to 
    \begin{equation*}
        \frac{C(p)(\lambda_{\text{PF}})^{|V_-\cap V_+|-3}}{\sum_{x\in \{0,1\}}\sqrt{\alpha(x)}v_{\text{PF}}(x)\sum_{y\in \{0,1\}}\sqrt{\alpha(y)}v_{\text{PF}}(y)}
    \end{equation*}
    independent of the sequence $\Delta$ and the first-layer boundary condition $\omega\in T^{-1}(\omega')$, see Lemma \ref{lem: Limit Matrix Q}. Here, we applied the third statement of Lemma \ref{lem: Limit Matrix Q} together with the relation $|\Delta\cap \Lambda_-|+|\Delta\cap \Lambda_+|-|\Delta|=-|V_-\cap V_+|+4$. Defining the denominator as $Z_\Lambda(\omega'_{\Lambda^c})$, we arrive at the expression for $\gamma'_{p,\Lambda}$ in \eqref{eq: Kernels image specification}.  
\end{proof}

\subsection{Theorem \ref{thm: Quasilocality Gamma'}: Quasilocality and asymptotics of $\gamma'_p$} \label{Subsec: Quasilocality Gamma'}

In order to derive the exponential bounds in \eqref{eq: Bounds on sensitivity}, it is necessary to compute the explicit form of $Q^m$. From \eqref{eq: N-th power Q generalized}, we obtain
\begin{align}\label{eq: N-th power Q}
     Q^m&=\frac{\lambda_{PF}^m}{\sqrt{(1-p)(3p+1)}}\begin{pmatrix}
         \lambda_{PF}-a^m\lambda_r&(1-a^m)\sqrt{p(1-p)}\\
    (1-a^m)\sqrt{p(1-p)}&-\lambda_r+a^m\lambda_{PF}
    \end{pmatrix}
\end{align}
for all $m\in \N$. In the following, we will prove the upper and lower bounds of \eqref{eq: Bounds on sensitivity} separately.

\begin{proof}[Proof of lower bound in Theorem \ref{thm: Quasilocality Gamma'}]

 Let $l,r,L,R\in  \Z$  satisfying $L<l\leq r<R$ and recall that $n=\min\{l-L,R-r\}$. Without loss of generality, we assume that $n=l-L$. We lower bound the sensitivity in \eqref{eq: Sensitivity of b.c. variations} by choosing $\omega',\eta'\in \Omega'$ where $\omega'_{(l,+\infty)}=\eta'_{(l,+\infty)}=1'_{(l,+\infty)}$ and especially
    \begin{equation*}
        \omega'_{(-\infty,l]}=1_{(-\infty,l-n)}0_{[l-n,l]}
    \end{equation*}
    as well as
    \begin{equation*}
        \eta'_{(-\infty,l]}=1_{(-\infty,l-n-1)}0_{[l-n-1,l]}.
    \end{equation*}
     Figure \ref{fig: Configs lower bound Quasilocality} presents an illustration of these configurations for $n=4$.

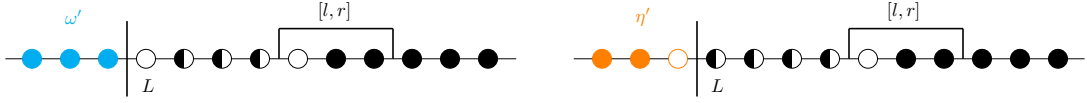
\begin{figure}[ht!]
\centering
\scalebox{0.5}{
\begin{tikzpicture}[every label/.append style={scale=1.3},fill fraction/.style={path picture={
\fill[#1] 
(path picture bounding box.south) rectangle
(path picture bounding box.north west);
}},
fill fraction/.default=gray!50]
    \node[shape=circle,draw=black, fill=black, minimum size=0.5cm] (-B) at (-1,0) {};
    \node[shape=circle,draw=black, fill=black,minimum size=0.5cm] (-C) at (-2,0) {};
    \node[shape=circle,draw=black, fill=black, minimum size=0.5cm] (-D) at (-3,0) {};
    \node[shape=circle,draw=black, fill=black, minimum size=0.5cm] (-E) at (-4,0) {};
    \node[shape=circle,draw=black,fill=black,  minimum size=0.5cm] (-F) at (-5,0) {};
    \node[shape=circle,draw=black, minimum size=0.5cm] (-G) at (-6,0) {}; \node[shape=circle,draw=black,fill fraction=black, minimum size=0.5cm] (-H) at (-7,0) {};
    \node[shape=circle,draw=black,fill fraction=black, minimum size=0.5cm] (-I) at (-8,0) {};
    \node[shape=circle,draw=black,fill fraction=black, minimum size=0.5cm] (-J) at (-9,0) {};
    \node[shape=circle,draw=black,  minimum size=0.5cm] (-K) at (-10,0) {};
    \node[shape=circle,draw=cyan, fill=cyan, minimum size=0.5cm] (-L) at (-11,0) {};
    \node[shape=circle,draw=cyan, fill=cyan, minimum size=0.5cm] (-M) at (-12,0) {};
    \node[shape=circle,draw=cyan, fill=cyan, minimum size=0.5cm] (-O) at (-13,0) {};

    \node[shape=circle,draw=orange, fill=orange, minimum size=0.5cm] (C) at (2,0) {};
    \node[shape=circle,draw=orange, fill=orange, minimum size=0.5cm] (D) at (3,0) {};
    \node[shape=circle,draw=orange, minimum size=0.5cm] (E) at (4,0) {};
    \node[shape=circle,draw=black, fill fraction=black,  minimum size=0.5cm] (F) at (5,0) {};
    \node[shape=circle,draw=black,fill fraction=black, minimum size=0.5cm] (G) at (6,0) {};
    \node[shape=circle,draw=black,fill fraction=black, minimum size=0.5cm] (H) at (7,0) {};
    \node[shape=circle,draw=black, fill fraction=black,  minimum size=0.5cm] (I) at (8,0) {};
    \node[shape=circle,draw=black,  minimum size=0.5cm] (J) at (9,0) {};
    \node[shape=circle,draw=black, fill=black,  minimum size=0.5cm] (K) at (10,0) {};
    \node[shape=circle,draw=black, fill=black,  minimum size=0.5cm] (L) at (11,0) {};
    \node[shape=circle,draw=black, fill=black,  minimum size=0.5cm] (M) at (12,0) {};
    \node[shape=circle,draw=black, fill=black,  minimum size=0.5cm] (N) at (13,0) {};
      \node[shape=circle,draw=black, fill=black,  minimum size=0.5cm] (O) at (14,0) {};

    \draw[-] (-O)--(-13.7,0){};
    \draw [-] (-O) -- (-M);
     \draw [-] (-L) -- (-M);
     \draw [-] (-L) -- (-K);
    \draw [-] (-J) -- (-K);
    \draw [-] (-I) -- (-J);
    \draw [-] (-H) -- (-I);
    \draw [-] (-G) -- (-H);
    \draw [-] (-F) -- (-G);
    \draw [-] (-F) -- (-E);
    \draw [-] (-E) -- (-D);
    \draw [-] (-D) -- (-C);
    \draw [-] (-C) -- (-B);
    \draw [-] (-B) -- (A);
   
    \draw[-](O)--(14.7,0){};
    \draw [-] (O) -- (N);
    \draw [-] (M) -- (N);
    \draw [-] (M) -- (L);
    \draw [-] (K) -- (L);
    \draw [-] (K) -- (J);
    \draw [-] (J) -- (I);
    \draw [-] (H) -- (I);
    \draw [-] (G) -- (H);
    \draw [-] (G) -- (F);
    \draw [-] (F) -- (E);
    \draw [-] (E) -- (D);
    \draw [-] (D) -- (C);
    \draw [-] (C) -- (B);

    \draw[-, black, very thick] (-6.5,0.8) to (-3.5,0.8)  node [text width=2.5cm,midway,above=0.5cm, align=center] {};
    \draw[-, black, very thick] (-6.5,0.8) to (-6.5,0)  node [text width=2.5cm,midway,above=0.5cm, align=center] {};
    \draw[-, black, very thick] (-3.5,0.8) to (-3.5,0)  node [text width=2.5cm,midway,above=0.5cm, align=center] {};
\scalebox{1.4}{\node[] () at (-3.6,0.9) {{\color{black}$[l,r]$}};}

\draw[-, black, very thick] (-10.5,1) to (-10.5,-1)  node [text width=2.5cm,midway,above=0.5cm, align=center] {};
\draw[-, black, very thick] (4.5,1) to (4.5,-1)  node [text width=2.5cm,midway,above=0.5cm, align=center] {};

     \draw[-, black, very thick] (8.5,0.8) to (11.5,0.8)  node [text width=2.5cm,midway,above=0.5cm, align=center] {};
    \draw[-, black, very thick] (8.5,0.8) to (8.5,0)  node [text width=2.5cm,midway,above=0.5cm, align=center] {};
    \draw[-, black, very thick] (11.5,0.8) to (11.5,0)  node [text width=2.5cm,midway,above=0.5cm, align=center] {};
    \scalebox{1.4}{\node[] () at (7.1,0.9) {{\color{black}$[l,r]$}};}
    
    \scalebox{1.4}{\node[] () at (2.2,0.8) {{\color{orange}$\eta'$}};}
    \scalebox{1.4}{\node[] () at (3.6,-0.5) {{\color{black}$L$}};}
    \scalebox{1.4}{\node[] () at (-8.5,0.8) {{\color{cyan}$\omega'$}};}
    \scalebox{1.4}{\node[] () at (-7.1,-0.5) {{\color{black}$L$}};}

\end{tikzpicture}
}
\small \caption{Depicted are two non-isolated configurations $\omega',\eta'\in \Omega'$ on the line together with a finite observation window $[l,r]$ containing three sites, indicated by a square bracket. The two configurations differ at a distance of five from the set $[l,r]$.}
\label{fig: Configs lower bound Quasilocality}
\end{figure}
    
    Note that the configurations $\omega'$ and $\eta'$ posses the outer unfixed areas $U_-=(l-n,l)$ and $\Tilde{U}_-=(l-n-1,l)$, respectively. All remaining parts of these configurations correspond to the fixed area $\overline{\Theta}$. Note that the partition function for the kernel corresponding to $\omega'$ is given by 
   \begin{equation*}
        Z_{[l,r]}(\omega'_{[l,r]^c})=\sum_{\hat{\omega}_{I }}\mu_p(\hat{\omega}_{I})\mathds{1}_{T_{I \setminus [l,r]}(\hat{\omega}_{I}\omega'_{I^c})=\omega'_{I \setminus [l,r]}}
   \end{equation*}
   where $I$ is the influencing set of $[l,r]$ for $\omega'_{[l,r]^c}$, see Definition \ref{def: restricted unfixed areas} and Theorem \ref{thm: Existence image specification}.
   Recall that $[l,r]\subset I$ and the map $T_{\Lambda}$ is $\mathscr{F}_{\Bar{\Lambda}}$-measurable for all $\Lambda \subset \Z$. Consequently, the indicator appearing in the partition function does not depend on the spins in $(l,r)$, and may therefore be expressed as
 \begin{equation}\label{eq: Partition function kernel rewritten}
        \sum_{\hat{\omega}_{I\setminus (l,r)}}\mu_p(\hat{\omega}_{I\setminus (l,r)})\mathds{1}_{T_{I \setminus [l,r]}(\hat{\omega}_{I\setminus (l,r)}\omega'_{I^c\cup (l,r)})=\omega'_{I \setminus [l,r]}}.
   \end{equation}
   Note that $I\setminus[l,r]= U_-\cup \{r+1\}$ and decompose the indicator as follows 
   \begin{align*}
       \mathds{1}_{T_{I \setminus [l,r]}(\hat{\omega}_{I\setminus (l,r)}\omega'_{I^c\cup (l,r)})=\omega'_{I \setminus [l,r]}}&=\mathds{1}_{T_{U_-}(\hat{\omega}_{[\min U_-,l]}\omega'_{[\min U_-,l]^c})=0'_{U_-}}\cdot\mathds{1}_{T_{\{r+1\}}(\hat{\omega}_{[r,r+1]}\omega'_{[r,r+1]^c})=1_{\{r+1\}}}.
   \end{align*}
   Note that the last indicator forces $\hat{\omega}_{\{r+1\}}$ to be one independently of the value $\hat{\omega}_{\{r\}}$. Inserting this equation in \eqref{eq: Partition function kernel rewritten} and using relation \eqref{eq: Relation Q and prob.} together with the fact that $\omega'_{\min U_--1}=0$, we obtain 
\begin{equation*}
       Z_{[l,r]}(\omega'_{[l,r]^c})=p(1-p)^{-\frac{1}{2}}\sum_{x \in \{0,1\}} \sqrt{\alpha(x)}Q^{|U_-|+1}(0,x)
   \end{equation*}
   where we recall that $\alpha(x)=p^x(1-p)^{1-x}$ for all $x\in \{0,1\}$. Analogously, one obtains an expression for $Z_{[l,r]}(\eta'_{[l,r]^c})$. Using Theorem \ref{thm: Existence image specification} together with the expression for the partition functions, we obtain
    \begin{equation}\label{eq: Lower bound quasilocality eq 1}
        \bigg|\frac{Q^{n}(0,0)(1-p)^{-\frac{1}{2}}}{\sum_{x\in \{0,1\}}\sqrt{\alpha(x)}Q^{n}(0,x)}-\frac{Q^{n+1}(0,0)(1-p)^{-\frac{1}{2}}}{\sum_{x\in \{0,1\}}\sqrt{\alpha(y)}Q^{n+1}(0,x)}\bigg|\cdot (1-p)p^{r-l}
    \end{equation}
    as a lower bound for the sensitivity in \eqref{eq: Sensitivity of b.c. variations}. Here, the factor $(1-p)p^{r-l}$ corresponds to the weights on $[l,r]$, which are fixed sites. The remaining terms in \eqref{eq: Lower bound quasilocality eq 1} represent the weights on the outer unfixed areas $U_-$ and $\Tilde{U}_-$, satisfying $|U_-|=n-1$ and $|\Tilde{U}_-|=n$, together with the corresponding parts of the partition functions. Substituting the entries of $Q^m$ from \eqref{eq: N-th power Q}, we obtain
    \begin{equation*}
        \bigg|\frac{\lambda_{\text{PF}}-a^n\lambda_r}{\lambda_{\text{PF}}-a^n\lambda_r+(1-a^n)p}-\frac{\lambda_{\text{PF}}-a^{n+1}\lambda_r}{\lambda_{\text{PF}}-a^{n+1}\lambda_r+(1-a^{n+1})p}\bigg|p^{r-l}.
    \end{equation*}
    This expression is equal to
     \begin{equation}\label{eq: Lower bound quasilocality eq 2}
        (1-a)\frac{\lambda_{\text{PF}}+|\lambda_r|}{\big(\lambda_{\text{PF}}+p-a^n(\lambda_r+p)\big)\big(\lambda_{\text{PF}}+p-a^{n+1}(\lambda_r+p)\big)}p^{r-l+1}|a|^n
    \end{equation}
    where we used that $a\in (-1,0)$ and $|\lambda_r|\leq \lambda_{\text{PF}}$ to omit the absolute value. From the fact that $\lambda_r+p\geq 0$, the denominator is upper bounded by 
    $(\lambda_{\text{PF}}+\lambda_r+2p)^2=(1+p)^2$.
    Note that $(1-a)=\lambda_{\text{PF}}^{-1}(\lambda_{\text{PF}}+|\lambda_r|)$ and $\lambda_{\text{PF}}+|\lambda_r|=\sqrt{(1-p)(3p+1)}$. Consequently, we obtain $\big((1-p)(3p+1)p^{r-l+1}\big)/\big(\lambda_{\text{PF}}(1+p)^2\big)$ as a lower bound for \eqref{eq: Lower bound quasilocality eq 2}. Finally, using the estimate $\lambda_{\text{PF}}\leq \frac{3}{2}\sqrt{1-p}$ and $p\in (0,1)$, 
    we obtain the bound in \eqref{eq: Bounds on sensitivity}. More precisely, the $n$-independent constant $C_-(l,r)=A_-\cdot(b_-)^{r-l+1}$ is given by 
    \begin{equation}\label{eq: Constants b_- and A_-}
        b_-=b_-(p):=p~~\text{and}~~A_-=A_-(p):=\frac{1}{6}\sqrt{1-p}.
    \end{equation}
    \end{proof}

  \begin{proof}[Proof of upper bound in Theorem \ref{thm: Quasilocality Gamma'}] Let $[l,r] \subset [L,R]\Subset \Z$. In order to upper bound the quantity \eqref{eq: Sensitivity of b.c. variations}, we need to compute an upper bound for the difference
   \begin{equation}\label{eq: Quasilocality eq 0}
       \big|
    \gamma'_{p,[l,r]}(\omega'_{[l,r]}|\omega'_{[l,r]^c})-\gamma'_{p,[l,r]}(\omega'_{[l,r]}|\eta'_{[l,r]^c})\big|
   \end{equation}
   covering all possible cases of $\omega',\eta'\in \Omega'$ satisfying $\omega'_{[L,R]}=\eta'_{[L,R]}$. Denote by $U_-,U_+\in \mathscr{U}_{[l,r]}(\omega')$ and $\tilde{U}_-,\tilde{U}_+\in \mathscr{U}_{[l,r]}(\eta')$ the outer unfixed areas corresponding to $\omega'$ and $\eta'$, respectively, as introduced in Definition \ref{def: restricted unfixed areas}. Note that spins in the observation window $[l,r]$ are influenced by perturbations outside of $[L,R]$ if there exists an outer unfixed area which contains $L+1$ or $R-1$, see Theorem \ref{thm: Existence image specification}. Therefore, we discuss three different cases.

   \textbf{Case 1.} ($L+1\in U_-$ and $R-1\in U_+$ and $U_-\neq U_+$) First, note that this also implies $L+1\in \Tilde{U}_-$ and $R-1\in \Tilde{U}_+$, since $\omega'_{[L,R]}=\eta'_{[L,R]}$. Secondly, we do not consider the case of a singleton, i.e., $l=r$, because in this situation both kernels are equal to one. Without loss of generality, we assume that the outer unfixed areas are finite. Otherwise, the specification kernels for configurations with infinite unfixed areas can be constructed via a sequence of kernels for configurations with finite unfixed areas, see proof of Theorem \ref{thm: Existence image specification}. Consequently, the partition function for the kernel corresponding to $\omega'$ is given by \eqref{eq: Partition function kernel rewritten}. Note that $I\setminus[l,r]=[\min U_-,l)\cup (r, \max U_+]$ and decompose the indicator as follows 
   \begin{align*}
       \mathds{1}_{T_{I \setminus [l,r]}(\hat{\omega}_{I\setminus (l,r)}\omega'_{I^c\cup (l,r)})=\omega'_{I \setminus [l,r]}}&=\mathds{1}_{T_{[\min U_-,l)}(\hat{\omega}_{[\min U_-,l]}\omega'_{[\min U_-,l]^c})=0'_{[\min U_-,l)}}\\ &\times\mathds{1}_{T_{(r,\max U_+]}(\hat{\omega}_{[r,\max U_+]}\omega'_{[r,\max U_+]^c})=0'_{(r,\max U_+]}}.
   \end{align*}
   Using the relation \eqref{eq: Relation Q and prob.} together with the fact that $\omega'_{\min U_--1}=0=\omega'_{\max U_++1}$ and $r-l>0$, we obtain 
   \begin{equation}\label{eq: Quasilocality eq 1}
       Z_{[l,r]}(\omega'_{[l,r]^c})=(1-p)^{-1}\sum_{x,y \in \{0,1\}} \sqrt{\alpha(x)}\sqrt{\alpha(y)}Q^{|[\min U_-,l)|+1}(0,x)Q^{|(r,\max U_+]|+1}(y,0).
   \end{equation}
    Analogously, one can express the partition function corresponding to $\eta'$. Define the lengths of the outer unfixed areas $U_-$ and $U_+$ before $l-1$ and after $r+1$, by 
    \begin{equation*}
        \alpha_-:=\big|U_-\cap (-\infty , l-1)\big|~~\text{and}~~\alpha_+:=\big|U_+\cap (r+1,+\infty)\big|.
    \end{equation*}
     Let $\Tilde{\alpha}_-$ and $\Tilde{\alpha}_+$ denote the corresponding quantities associated with $\eta'$. Furthermore, we define 
     \begin{equation*}
        \beta_-:=\big|U_-\cap [l-1,r+1]\big|~~\text{and}~~\beta_+:=\big|U_+\cap [l-1,r+1]\big|
    \end{equation*}
    as the length of $U_-$ and $U_+$ inside of the interval $[l-1,r+1]$. Note that these quantities coincide for $\Tilde{U}_-$ and $\Tilde{U}_+$ as $\omega'_{[l-1,r+1]}=\eta'_{[l-1,r+1]}$ for $L\leq l-1$ and $R\geq r+1$. Consequently, we have $|U_-|=\alpha_-+\beta_-$ and $|U_+|=\alpha_++\beta_+$, as well as $|\Tilde{U}_-|=\Tilde{\alpha}_-+\beta_-$ and $|\Tilde{U}_+|=\Tilde{\alpha}_++\beta_+$. Combining Theorem \ref{thm: Existence image specification} and \eqref{eq: Quasilocality eq 1}, we can write
      \begin{align}\label{eq: Quasilocality eq 2}
       \nonumber&p^{|\Theta\cap[l-1,r+1]|}(1-p)^{|\partial_+ \Theta \cap [l-1,r+1]|}\Big(\prod_{U\in \mathscr{U}_{[l,r]}(\omega')\setminus\{U_-,U_+\}}Z(U)\Big)\\
       &\times\big|f(\alpha_-,\beta_-)f(\alpha_+,\beta_+)-f(\Tilde{\alpha}_-,\beta_-)f(\Tilde{\alpha}_+,\beta_+)\big|
   \end{align}
   for the difference in \eqref{eq: Quasilocality eq 0}. The functions
   \begin{equation*}
       f(i,j):=\frac{(1-p)^{-\frac{1}{2}}Q^{i+j+1}(0,0)}{\sum_{x\in \{0,1\}}\sqrt{\alpha(x)}Q^{i+2}(0,x)},~~i,j\in \N_0
   \end{equation*}
   appearing in \eqref{eq: Quasilocality eq 2}, represent the weights of the outer unfixed areas. We can bound the difference $|f(\alpha_-,\beta_-)f(\alpha_+,\beta_+)-f(\Tilde{\alpha}_-,\beta_-)f(\Tilde{\alpha}_+,\beta_+)|$ in \eqref{eq: Quasilocality eq 2} from above as follows 
   \begin{align}\label{eq: Quasilocality eq 3}
        f(\alpha_+,\beta_+)\big|f(\alpha_-,\beta_-)-f(\Tilde{\alpha}_-,\beta_-)\big|+f(\Tilde{\alpha}_-,\beta_-)\big|f(\alpha_+,\beta_+)-f(\Tilde{\alpha}_+,\beta_+)\big|.
   \end{align}
   Consequently, it is enough to compute an upper bound for the quantities $f(i,j)$ and\\ $|f(i,j)-f(\Tilde{i},j)|$ where $j\in \N_0$ and $i,\Tilde{i}\in \N_0$ satisfy $i,\Tilde{i}\geq n-2$ with $n:=\min\{R-r,l-L\}$. First of all, one obtains 
   \begin{equation*}
       f(i,j)=\lambda_{\text{PF}}^{j-1}\frac{\lambda_{\text{PF}}-a^{i+j+1}\lambda_r}{\lambda_{\text{PF}}+p-a^{i+2}(\lambda_r+p)}
   \end{equation*}
   after inserting the entries of $Q^m$, given in \eqref{eq: N-th power Q}. Using the fact that $a\in (-1,0)$, we obtain
   \begin{equation}\label{eq: Quasilocality eq 4}
       f(i,j)\leq \lambda_{\text{PF}}^{j-1}\frac{\lambda_{\text{PF}}+a^{2}|\lambda_r|}{\lambda_{\text{PF}}+p-a^2(\lambda_r+p)}=\lambda_{\text{PF}}^{j-1}\frac{1}{1+p}.
   \end{equation} 
 Similarly, the difference $|f(i,j)-f(\Tilde{i},j)|$ reads
 \begin{equation*}
     \lambda_{\text{PF}}^{j}\frac{a^2|a^{i}-a^{\Tilde{i}}||(\lambda_r+p)-a^{j}(\lambda_{\text{PF}}+p)|}{\big(\lambda_{\text{PF}}+p-a^{i+2}(\lambda_r+p)\big)\big(\lambda_{\text{PF}}+p-a^{\Tilde{i}+2}(\lambda_r+p)\big)}
 \end{equation*}
 and applying $\max_{\substack{k,\Tilde{k}\in \N_{\geq n-2}}}|a^k-a^{\Tilde{k}}|=|a|^{n-2}(1-a)$
 leads to the upper bound
    \begin{equation}\label{eq: Quasilocality eq 5}
        |a|^{n}\lambda_{\text{PF}}^{j}\frac{(1-a)(\lambda_{\text{PF}}+\lambda_r+2p)}{\lambda_{\text{PF}}+p-a^2(\lambda_r+p)}=|a|^{n}\lambda_{\text{PF}}^{j+1}\frac{1}{1-p}.
    \end{equation}
     Combining the upper bounds \eqref{eq: Quasilocality eq 3}-\eqref{eq: Quasilocality eq 5} results in 
    \begin{equation}\label{eq: Quasilocality eq 6}
        |f(\alpha_-,\beta_-)f(\alpha_+,\beta_+)-f(\Tilde{\alpha}_-,\beta_-)f(\Tilde{\alpha}_+,\beta_+)|\leq  2(\lambda_{\text{PF}})^{\beta_-+\beta_+}|a|^{n}(1-p^2)^{-1}.
    \end{equation}
    Finally, we give an upper bound for the remaining weights in \eqref{eq: Quasilocality eq 2}. To this end, note that
    \begin{equation*}
        \frac{Q^m(0,0)}{(1-p)}\leq \lambda_{\text{PF}}^{m-2}\frac{\lambda_{\text{PF}}^2(\lambda_{\text{PF}}+a^2|\lambda_r|)}{(1-p)\sqrt{(1-p)(3p+1)}}= \lambda_{\text{PF}}^{m-2}
    \end{equation*}
    for all $m\in \N_{\geq 2}$. This implies $Z(U)\leq \lambda_{\text{PF}}^{|U|-1}$ for all $U\in \mathscr{U}$. However, each finite unfixed area $U\in \mathscr{U}$ is surrounded by unoccupied fixed sites which are located at $\partial_+U$. Therefore, we have
    \begin{equation*}
        (1-p)^2Z(U)\leq \max\{(1-p)^{2/3},\lambda_{\text{PF}}\}^{|U|+2}=\lambda_{\text{PF}}^{|U|+2}
    \end{equation*}
    for each $U\in \mathscr{U}$. Consequently, we obtain
    \begin{align}\label{eq: Quasilocality eq 7}
       p^{|\Theta\cap[l-1,r+1]|}(1-p)^{|\partial_+ \Theta \cap [l-1,r+1]|}\Big(\prod_{U\in \mathscr{U}_{[l,r]}(\omega')\setminus\{U_-,U_+\}}Z(U)\Big)\leq \max\{p,\lambda_{\text{PF}}\}^{3+r-l-\beta_--\beta_+},
   \end{align}
   where we used that $\lambda_{\text{PF}}\geq (1-p)^{2/3}\geq (1-p)$ for all $p\in (0,1)$. Combining \eqref{eq: Quasilocality eq 6} and \eqref{eq: Quasilocality eq 7} in \eqref{eq: Quasilocality eq 2}, results in an upper bound of $C_1:=2(1-p^2)^{-1}\max\{p,\lambda_{\text{PF}}\}^{3+r-l}|a|^n$ for the difference in \eqref{eq: Quasilocality eq 0}.

   \textbf{Case 2.} ($L+1\in U_-$ and $R-1\notin U_+$ and $U_-\neq U_+$) For the case $L+1\notin U_-$ and $R-1\in U_+$, the argument proceeds analogously. If $U_+\neq \emptyset$, one can follow the same reasoning as in the first case and note that $\alpha_+=\Tilde{\alpha}_+$, which yields
   \begin{equation*}
        \big|f(\alpha_-,\beta_-)f(\alpha_+,\beta_+)-f(\Tilde{\alpha}_-,\beta_-)f(\Tilde{\alpha}_+,\beta_+)\big|=f(\alpha_+,\beta_+)\big|f(\alpha_-,\beta_-)-f(\Tilde{\alpha}_-,\beta_-)\big|
   \end{equation*}
   for the difference in \eqref{eq: Quasilocality eq 2}. If $U_+=\emptyset$, the partition function must be treated differently. This case also covers singletons, i.e., when $l=r$. Note that $I\setminus [l,r]=[\min U_-,l)\cup \{r+1\}$ and the indicator can be decomposed as follows
   \begin{align*}
       \mathds{1}_{T_{I \setminus [l,r]}(\hat{\omega}_{I\setminus (l,r)}\omega'_{I^c\cup (l,r)})=\omega'_{I \setminus [l,r]}}&=\mathds{1}_{T_{[\min U_-,l)}(\hat{\omega}_{[\min U_-,l]}\omega'_{[\min U_-,l]^c})=0'_{[\min U_-,l)}}\\ &\times\mathds{1}_{T_{\{r+1\}}(\hat{\omega}_{\{r,r+1\}}\omega'_{\{r,r+1\}^c})=\omega'_{\{r+1\}}}.
   \end{align*}
   The fact that $U_+=\emptyset$ implies that the spin of the configuration at $r+1$ must be fixed independently of the spin at $r$. Consequently, the equation $\mathds{1}_{T_{\{r+1\}}(\hat{\omega}_{\{r,r+1\}}\omega'_{\{r,r+1\}^c})=\omega'_{\{r+1\}}}=\mathds{1}_{\hat{\omega}_{\{r+1\}}=\omega'_{\{r+1\}}}$ holds independently of the choice of $\hat{\omega}_{\{r\}}\in \{0,1\}$. Therefore, the partition function takes the form 
    \begin{equation*}
       Z_{[l,r]}(\omega'_{[l,r]^c})=\alpha(\omega'_{\{r+1\}})(1-p)^{-1/2}\sum_{x \in \{0,1\}} \sqrt{\alpha(x)}Q^{|[\min U_-,l)|+1}(0,x).
   \end{equation*}
   This leads to the expression
   \begin{align}\label{eq: Quasilocality eq 8}
       p^{|\Theta\cap[l-1,r]|}(1-p)^{|\partial_+ \Theta \cap [l-1,r]|}\Big(\prod_{U\in \mathscr{U}_{[l,r]}(\omega')\setminus\{U_-\}}Z(U)\Big)\big|f(\alpha_-,\beta_-)-f(\Tilde{\alpha}_-,\beta_-)\big|
   \end{align}
   for the difference in \eqref{eq: Quasilocality eq 0}. Each unfixed area $U\in \mathscr{U}_{[l,r]}(\omega')\setminus\{U_-\}$ is surrounded by at least one unoccupied site in $\partial_+ \Theta \cap [l-1,r]$ and consequently
   \begin{equation*}
       (1-p)Z(U)\leq \max\{\sqrt{1-p},\lambda_{\text{PF}}\}^{|U|+1}= \lambda_{\text{PF}}^{|U|+1}
   \end{equation*}
   with the ideas derived in Case 1. In combination with the result in \eqref{eq: Quasilocality eq 5}, we obtain $C_2:=|a|^n\max\{p,\lambda_{\text{PF}}\}^{3+r-l}\frac{1}{1-p}$ as an upper bound for \eqref{eq: Quasilocality eq 8}.

    \textbf{Case 3.} ($U_-=U_+\neq \emptyset$) Let us start to discuss the case where $L+1\in U_-$ and $R-1 \notin U_+$. The partition function takes the form as in Case 1. The influencing sets only consist of the outer unfixed area, i.e., $I_{[l,r]}(\omega')=U_-$ and $I_{[l,r]}(\eta')=\Tilde{U}_-$. From the fact that $\alpha_+=\Tilde{\alpha}_+$, the difference in \eqref{eq: Quasilocality eq 0} is equal to
    \begin{equation}\label{eq: Quasilocality eq 9}
       g(\alpha_+)\cdot\big|f(\alpha_-,\beta_-+\alpha_+)-f(\Tilde{\alpha}_-,\beta_-+\alpha_+)\big|
    \end{equation}
    where
    \begin{equation*}
        g(i):=\frac{\sqrt{1-p}}{\sum_{y\in \{0,1\}} \sqrt{\alpha(y)} Q^{i+2}(y,0)}
    \end{equation*}
    for all $i\in \N_0$. Note that 
    \begin{equation}\label{eq: Quasilocality eq 10}
        g(i)=\frac{\sqrt{(1-p)(3p+1)}}{\lambda_{\text{PF}}^{i+2}\big((\lambda_{\text{PF}}+p)-a^{i+2}(\lambda_r+p)\big)}\leq\lambda_{\text{PF}}^{-i-2}
    \end{equation}
    where we used that $\lambda_{\text{PF}}-\lambda_r=\sqrt{(1-p)(3p+1)}$. Combining the bounds in \eqref{eq: Quasilocality eq 5} and \eqref{eq: Quasilocality eq 10}, leads to $C_3:=|a|^n\lambda_{\text{PF}}^{r-l+2}(1-p)^{-1}$ as an upper bound for \eqref{eq: Quasilocality eq 9}. Here, we used the fact that $\beta_-=r-l+3$. 
    
    In the situation $L+1\in U_-$ and $R-1 \in U_+$, we possibly have $\alpha_+\neq\Tilde{\alpha}_+$. Again, the partition function takes the form as in Case 1 and the influencing sets are equal to the outer unfixed area. Therefore, we have to compute the difference $|h(\alpha_-,\alpha_+)-h(\Tilde{\alpha}_-,\Tilde{\alpha}_+)|$ where
    \begin{equation*}
        h(i,j):=\frac{Q^{i+j+\beta+3}(0,0)}{\sum_{x\in \{0,1\}} \sqrt{\alpha(x)} Q^{i+2}(
    0,x)\sum_{y\in \{0,1\}} \sqrt{\alpha(y)} Q^{j+2}(y,0)}
    \end{equation*}
    for all $i,j\in \N_0$ and $\beta:=r-l$. Inserting the entries of $Q^m$, displayed in \eqref{eq: N-th power Q}, in the expression $|h(i,j)-h(\Tilde{i},\Tilde{j})|$, leads to 
    \begin{align}\label{eq: Quasilocality eq 11}
        \nonumber&\lambda_{\text{PF}}^{\beta-1}\sqrt{3p+1}|a|^2\cdot\Big|p^2\lambda_{\text{PF}}(a^i-a^{\Tilde{i}}+a^j-a^{\Tilde{j}})+p^2\lambda_ra^{3+\beta}\big(a^{i+\tilde{i}}(a^j-a^{\Tilde{j}})+a^{j+\Tilde{j}}(a^i-a^{\Tilde{i}})\big)\\
        \nonumber&+\big(\lambda_{\text{PF}}(\lambda_{\text{r}}+p)^2a^2+\lambda_{\text{r}}(\lambda_{\text{PF}}+p)^2a^{1+\beta}\big)(a^{\Tilde{i}+\Tilde{j}}-a^{i+j})\Big|\Big/\Big(\big(\lambda_{\text{PF}}+p-a^{i+2}(\lambda_{\text{r}}+p)\big)\\
        &\times\big(\lambda_{\text{PF}}+p-a^{j+2}(\lambda_{\text{r}}+p)\big)\big(\lambda_{\text{PF}}+p-a^{\Tilde{i}+2}(\lambda_{\text{r}}+p)\big)\big(\lambda_{\text{PF}}+p-a^{\Tilde{j}+2}(\lambda_{\text{r}}+p)\big)\Big).
    \end{align}
    Here, we used the relation $(\lambda_{\text{PF}}+p)(\lambda_{\text{r}}+p)=p^2$. The denominator can be bounded from below by $(1-p)^2(3p+1)^2$. Using $a\in (-1,0)$ and $\max_{\substack{k,\Tilde{k}\in \N_{\geq n-2}}}|a^k-a^{\Tilde{k}}|=|a|^{n-2}(1-a)$, leads to
    \begin{equation*}
        |a|^n\lambda_{\text{PF}}^{\beta-1}(1-a)\sqrt{3p+1}\Big(2p^2(\lambda_{\text{PF}}+|\lambda_r|)+\lambda_{\text{PF}}(\lambda_{\text{r}}+p)^2+|\lambda_{\text{r}}|(\lambda_{\text{PF}}+p)^2\Big)
    \end{equation*}
    as an upper bound for the numerator of \eqref{eq: Quasilocality eq 11}. Note that $\lambda_{\text{PF}}+|\lambda_r|=\sqrt{(1-p)(3p+1)}$ and $\lambda_{\text{PF}}(\lambda_{\text{r}}+p)^2+|\lambda_{\text{r}}|(\lambda_{\text{PF}}+p)^2=p\sqrt{(1-p)(3p+1)}$ for all $p\in (0,1)$. Consequently, the numerator of \eqref{eq: Quasilocality eq 11} is upper bounded by $24|a|^n\lambda_{\text{PF}}^{\beta-1}p\sqrt{(1-p)}$. Combining both bounds, results in 
    \begin{equation*}
       |h(\alpha_-,\alpha_+)-h(\Tilde{\alpha}_-,\Tilde{\alpha}_+)|\leq |a|^n\max\{p,\lambda_{\text{PF}}\}^{\beta+1}\frac{24}{(1-p)^{2}}=:C_4.
    \end{equation*}

Comparing the bounds $C_i$ for $i\in \{1,2,3,4\}$, we summarize that \eqref{eq: Quasilocality eq 0} is upper bounded by $C_+(l,r)=|a|^n (b_+)^{r-l+1}A_+$ where 
\begin{equation}\label{eq: Constants b_+ and A_+}
    b_+=b_+(p):=\max\{p,\lambda_{\text{PF}}\}~~\text{and}~~A_+=A_+(p):=\frac{24}{(1-p)^2}.
\end{equation}
   \end{proof}

    \subsection{Proposition \ref{thm: Uniqueness of the Gibbs measure}: Finite energy on non-isolation subspace and uniqueness}\label{Subsec: Uniqueness of Gibbs measure}

The fact that $\mathscr{G}(\gamma'_p)=\{\mu'_p\}$ follows directly from \eqref{Ineq: finite energy cond} since this is the "finite-energy" criterion of Proposition 8.38 in \cite{Ge11}. Note that Proposition 8.38 is given in the context of specifications on product spaces. However, it still holds on the smaller measurable space $(\Omega',\mathscr{F}')$ which we explain in \hyperref[Sec: Appendix B]{Appendix B}.

\begin{proof}[Proof of Proposition \ref{thm: Uniqueness of the Gibbs measure}]
Without loss of generality, let $m\in \N$ be such that $A \in \mathscr{F}'_{[-m,m]}$. For this event, we then choose $\Lambda:=[-m-4,m+4]$, see Figure \ref{fig: Explanation Uniqueness Gibbs measures} for an illustration in the case of $m=5$.

\begin{figure}[ht!]
\centering
\scalebox{0.5}{
\begin{tikzpicture}[every label/.append style={scale=1.3}]
    \node[shape=circle,draw=cyan,dashed,   minimum size=0.5cm] (-B) at (-1,0) {};
    \node[shape=circle,draw=cyan,dashed,   minimum size=0.5cm] (-C) at (-2,0) {};
    \node[shape=circle,draw=cyan,dashed,   minimum size=0.5cm] (-D) at (-3,0) {};
    \node[shape=circle,draw=cyan,dashed,  minimum size=0.5cm] (-E) at (-4,0) {};
    \node[shape=circle,draw=cyan,dashed,   minimum size=0.5cm] (-F) at (-5,0) {};
    \node[shape=circle,draw=black,  minimum size=0.5cm] (-G) at (-6,0) {}; \node[shape=circle,draw=black, minimum size=0.5cm] (-H) at (-7,0) {};
    \node[shape=circle,draw=black,  minimum size=0.5cm] (-I) at (-8,0) {};
    \node[shape=circle,draw=black,  minimum size=0.5cm] (-J) at (-9,0) {};
    \node[shape=circle,draw=black,  minimum size=0.5cm] (-K) at (-10,0) {};
    \node[shape=circle,draw=black,  minimum size=0.5cm] (-L) at (-11,0) {};
    \node[shape=circle,draw=black,  minimum size=0.5cm] (-M) at (-12,0) {};
    \node[shape=circle,draw=black,  minimum size=0.5cm] (-O) at (-13,0) {};
    \node[shape=circle,draw=cyan,dashed, minimum size=0.5cm, label=above:{$0$}] (A) at (0,0) {};
    \node[shape=circle,draw=cyan,dashed,   minimum size=0.5cm] (B) at (1,0) {};
    \node[shape=circle,draw=cyan,dashed, minimum size=0.5cm] (C) at (2,0) {};
    \node[shape=circle,draw=cyan,dashed,   minimum size=0.5cm] (D) at (3,0) {};
    \node[shape=circle,draw=cyan,dashed,   minimum size=0.5cm] (E) at (4,0) {};
    \node[shape=circle,draw=cyan,dashed,   minimum size=0.5cm] (F) at (5,0) {};
    \node[shape=circle,draw=black, minimum size=0.5cm] (G) at (6,0) {};
    \node[shape=circle,draw=black,  minimum size=0.5cm] (H) at (7,0) {};
    \node[shape=circle,draw=black, minimum size=0.5cm] (I) at (8,0) {};
    \node[shape=circle,draw=black,  minimum size=0.5cm] (J) at (9,0) {};
    \node[shape=circle,draw=black,  minimum size=0.5cm] (K) at (10,0) {};
    \node[shape=circle,draw=black,  minimum size=0.5cm] (L) at (11,0) {};
    \node[shape=circle,draw=black, minimum size=0.5cm] (M) at (12,0) {};
    \node[shape=circle,draw=black, minimum size=0.5cm] (O) at (13,0) {};
     \node[minimum size=0.5, label=:{$-5$}] () at (-5.1,0.2) {};
     \node[minimum size=0.5, label=:{$5$}] () at (5,0.2) {};
     \node[minimum size=0.5, label=:{$9$}] () at (9,0.2) {};
     \node[minimum size=0.5, label=:{$-9$}] () at (-9.1,0.2) {};
     \node[minimum size=0.5, label=:{{\color{cyan}$[-5,5]$}}] () at (3,0.4) {};
     \node[minimum size=0.5, label=:{$\Lambda$}] () at (0,-2) {};

      \draw[-, black, very thick] (-9.5,-0.8) to (9.5,-0.8)  node [text width=2.5cm,midway,above=0.5cm, align=center] {};
    \draw[-, black, very thick] (-9.5,-0.8) to (-9.5,0)  node [text width=2.5cm,midway,above=0.5cm, align=center] {};
    \draw[-, black, very thick] (9.5,-0.8) to (9.5,0)  node [text width=2.5cm,midway,above=0.5cm, align=center] {};

    \draw[-] (-O)--(-13.7,0){};
    \draw [-] (-O) -- (-M);
     \draw [-] (-L) -- (-M);
     \draw [-] (-L) -- (-K);
    \draw [-] (-J) -- (-K);
    \draw [-] (-I) -- (-J);
    \draw [-] (-H) -- (-I);
    \draw [-] (-G) -- (-H);
    \draw [-] (-F) -- (-G);
    \draw [-] (-F) -- (-E);
    \draw [-] (-E) -- (-D);
    \draw [-] (-D) -- (-C);
    \draw [-] (-C) -- (-B);
    \draw [-] (-B) -- (A);
    \draw [-] (A) -- (B);
    \draw[-](O)--(13.7,0){};
    \draw [-] (M) -- (O);
    \draw [-] (M) -- (L);
    \draw [-] (K) -- (L);
    \draw [-] (K) -- (J);
    \draw [-] (J) -- (I);
    \draw [-] (H) -- (I);
    \draw [-] (G) -- (H);
    \draw [-] (G) -- (F);
    \draw [-] (F) -- (E);
    \draw [-] (E) -- (D);
    \draw [-] (D) -- (C);
    \draw [-] (C) -- (B);

\end{tikzpicture}
}
\small \caption{The sites of the dependence set $[-5,5]$ on the line $\Z$ are coloured and dashed in blue and the chosen set $\Lambda=[-9,9]$ is marked with a squared bracket.}
\label{fig: Explanation Uniqueness Gibbs measures}
\end{figure}
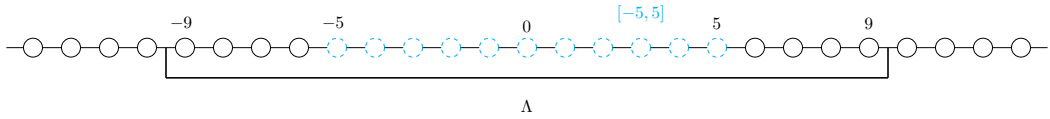

Instead of verifying \eqref{Ineq: finite energy cond}, we will show that 
\begin{align}\label{Ineq: finite energy cond rewritten}
    \frac{\gamma'_{p,\Lambda}
    (A|\zeta')}{\gamma'_{p,\Lambda}
    (
    \omega'_{[-m,m]}|\zeta')}\frac{\gamma'_{p,\Lambda}
    (
    \omega'_{[-m,m]}|\eta')}{\gamma'_{p,\Lambda}
    (A|\eta')}\geq C
\end{align}
for all $\omega'_{[-m,m]}\in \Omega'_{[-m,m]}$. This implies \eqref{Ineq: finite energy cond} after multiplying the denominators on both sides and sum up over $\omega'_{[-m,m]}$. The fractions in \eqref{Ineq: finite energy cond rewritten} have the advantage that we do not have to deal with the partition functions. First of all, \eqref{Ineq: finite energy cond rewritten} has the form 
\begin{align*}
    \frac{\sum_{\Tilde{\omega}'\in A}\gamma'_{p,\Lambda}
    (\Tilde{\omega}'_{[-m,m]}|\zeta')}{\gamma'_{p,\Lambda}
    (
    \omega'_{[-m,m]}|\zeta')}\frac{\gamma'_{p,\Lambda}
    (
    \omega'_{[-m,m]}|\eta')}{\sum_{\Tilde{\omega}'\in A}\gamma'_{p,\Lambda}
    (\Tilde{\omega}'_{[-m,m]}|\eta')}\geq C.
\end{align*}
Consider the fraction 
\begin{align*}
     \frac{\gamma'_{p,\Lambda}
    (\Tilde{\omega}'_{[-m,m]}|\zeta')}{\gamma'_{p,\Lambda}
    (
    \omega'_{[-m,m]}|\zeta')}=\frac{\sum_{\hat{\omega}'_{\Lambda \setminus [-m,m]}}\gamma'_{p,\Lambda}
    (\hat{\omega}'_{\Lambda \setminus [-m,m]}\Tilde{\omega}'_{[-m,m]}|\zeta')}{\sum_{\hat{\omega}'_{\Lambda \setminus [-m,m]}}\gamma'_{p,\Lambda}
    (\hat{\omega}'_{\Lambda \setminus [-m,m]}
    \omega'_{[-m,m]}|\zeta')}
\end{align*}
for an $\Tilde{\omega}'\in A$. Due to the idea in the proof of Theorem \ref{thm: Existence image specification}, the above expression can be written as
\begin{align}\label{eq: Fraction Uniqueness}
     \frac{\sum_{\hat{\omega}'_{\Lambda \setminus [-m,m]}}\sum_{\Tilde{\omega}_{
     I}}\mu_p(\Tilde{\omega}_{I})\mathds{1}_{T_{I}(\Tilde{\omega}_{I}\zeta'_{\partial_+ I})=\hat{\omega}'_{\Lambda \setminus [-m,m]}\Tilde{\omega}'_{[-m,m]}\zeta'_{I\setminus \Lambda}}}{\sum_{\hat{\omega}'_{\Lambda \setminus [-m,m]}}\sum_{\omega_{
     I}}\mu_p(\omega_{I})\mathds{1}_{T_{I}(\omega_{I}\zeta'_{\partial_+ I})=\hat{\omega}'_{\Lambda \setminus [-m,m]}\omega'_{[-m,m]}\zeta'_{I\setminus \Lambda}}}.
\end{align}
 where $I$ is the influencing set of $\Lambda$ with the boundary condition $\zeta'_{\Lambda^c}$, see \eqref{eq: Influence set}. A lower bound of the numerator and an upper bound of the denominator can be obtained by taking a look at the sites between $[-m,m]$ and $\Lambda^c$, compare Figure \ref{fig: Explanation Uniqueness Gibbs measures}. For the first one, we need to lower bound every term with zero in the sum over $\hat{\omega}'_{\Lambda \setminus [-m,m]}$ except those ones with  $\hat{\omega}'_{\{-m-3,-m-2,m+2,m+3\}}=1_{\{-m-3,-m-2,m+2,m+3\}}$. Note that we can split the indicator as follows  
\begin{align}\label{eq: Splitted indicator uniqueness}
    \nonumber\mathds{1}_{T_{I}(\Tilde{\omega}_{I}\zeta'_{\partial_+ I})=\hat{\omega}'_{\Lambda \setminus [-m,m]}\Tilde{\omega}'_{[-m,m]}\zeta'_{I\setminus \Lambda}}&=\mathds{1}_{T_{I\setminus \Lambda}(\Tilde{\omega}_{I\setminus \mathring{\Lambda}}\zeta'_{\partial_+ I})=\zeta'_{I\setminus \Lambda}}\mathds{1}_{T_{\Lambda\setminus[-m,m]}(\Tilde{\omega}_{\Bar{\Lambda}\setminus (-m,m)})=\hat{\omega}'_{\Lambda \setminus [-m,m]}}\\
    &\times\mathds{1}_{T_{[-m,m]}(\Tilde{\omega}_{[-m-1,m+1]})=\Tilde{\omega}'_{[-m,m]}}
\end{align}
and hence the second indicator together with the above assumption implies that also\\ $\Tilde{\omega}_{\{-m-3,-m-2,m+2,m+3\}}=1_{\{-m-3,-m-2,m+2,m+3\}}$. Note that this results in a prefactor of $p^4$ in the final lower bound stated in \eqref{eq: Lower bound uniqueness}. Therefore, the second indicator of \eqref{eq: Splitted indicator uniqueness} can be separated into the following four indicators:
\begin{align*}
    &\mathds{1}_{T_{\{-m-4\}}(\Tilde{\omega}_{\{-m-5,-m-4\}}1_{-m-3})=\hat{\omega}'_{-m-4}}\mathds{1}_{T_{\{-m-1\}}(1_{-m-2}\Tilde{\omega}_{\{-m-1,-m\}})=\hat{\omega}'_{-m-1}}\\
    &\times\mathds{1}_{T_{\{m+1\}}(\Tilde{\omega}_{\{m,m+1\}}1_{m+2})=\hat{\omega}'_{m+1}}\mathds{1}_{T_{\{m+4\}}(1_{m+3}\Tilde{\omega}_{\{m+4,m+5\}})=\hat{\omega}'_{m+4}}.
\end{align*}
Recall that we sum over $\hat{\omega}'_{\{-m-4,-m-1,m+1,m+4\}}$ in the numerator of \eqref{eq: Fraction Uniqueness} and hence 
\begin{align*}
    &\sum_{\hat{\omega}'_{\{-m-4,-m-1,m+1,m+4\}}}\sum_{\Tilde{\omega}_{\{-m-4,-m-1,m+1,m+4\}}}\mu_p(\Tilde{\omega}_{\{-m-4,-m-1,m+1,m+4\}})\\
    &\times\mathds{1}_{T_{\{-m-4\}}(\Tilde{\omega}_{\{-m-5,-m-4\}}1_{-m-3})=\hat{\omega}'_{-m-4}}\mathds{1}_{T_{\{-m-1\}}(1_{-m-2}\Tilde{\omega}_{\{-m-1,-m\}})=\hat{\omega}'_{-m-1}}\\
    &\times\mathds{1}_{T_{\{m+1\}}(\Tilde{\omega}_{\{m,m+1\}}1_{m+2})=\hat{\omega}'_{m+1}}\mathds{1}_{T_{\{m+4\}}(1_{m+3}\Tilde{\omega}_{\{m+4,m+5\}})=\hat{\omega}'_{m+4}}=1.
\end{align*}
To conclude, we have lower bounded the numerator of \eqref{eq: Fraction Uniqueness} by 
\begin{align}\label{eq: Lower bound uniqueness}
    \nonumber&p^4\sum_{\Tilde{\omega}_{I\setminus \mathring{\Lambda}}}\mu_p(\Tilde{\omega}_{I\setminus \mathring{\Lambda}})\mathds{1}_{T_{I\setminus \Lambda}(\Tilde{\omega}_{I\setminus \mathring{\Lambda}}\zeta'_{\partial_+ I})=\zeta'_{I\setminus \Lambda}} \\
    &\times\sum_{\Tilde{\omega}_{[-m-1,m+1]}}\mu_p(\Tilde{\omega}_{[-m-1,m+1]})\mathds{1}_{T_{[-m,m]}(\Tilde{\omega}_{[-m-1,m+1]})=\Tilde{\omega}'_{[-m,m]}}.
\end{align}

We proceed by establishing an upper bound for the denominator in \eqref{eq: Fraction Uniqueness}. For this purpose, note that $\mathds{1}_{T_{\Lambda\setminus[-m,m]}(\omega_{\Bar{\Lambda}\setminus (-m,m)})=\hat{\omega}'_{\Lambda \setminus [-m,m]}} \leq 1$ which we can apply to \eqref{eq: Splitted indicator uniqueness}. Consequently, the denominator does not depend on the sum over $\hat{\omega}'_{\Lambda \setminus [-m,m]}$ and we get a combinatorial factor of 256 in the final upper bound in \eqref{eq: Upper bound uniqueness}. Furthermore, we have 
\begin{align*}
    \sum_{\omega_{\{-m-3,-m-2,m+2,m+3\}}}\mu_p(\omega_{\{-m-3,-m-2,m+2,m+3\}})=1.
\end{align*}
This leads to the following upper bound 
\begin{align}\label{eq: Upper bound uniqueness}
    \nonumber&256\sum_{\omega_{C\setminus \mathring{\Lambda}}}\mu_p(\omega_{C\setminus \mathring{\Lambda}})\mathds{1}_{T_{C\setminus \Lambda}(\omega_{C\setminus \mathring{\Lambda}}\zeta'_{\partial_+ C})=\zeta'_{C\setminus \Lambda}} \\
    &\times\sum_{\omega_{[-m-1,m+1]}}\mu_p(\omega_{[-m-1,m+1]})\mathds{1}_{T_{[-m,m]}(\omega_{[-m-1,m+1]})=\omega'_{[-m,m]}}
\end{align}
of the denominator in \eqref{eq: Fraction Uniqueness}. Combining \eqref{eq: Lower bound uniqueness} and \eqref{eq: Upper bound uniqueness} in \eqref{eq: Fraction Uniqueness} leads to the following lower bound
\begin{align*}
    \frac{p^4\sum_{\Tilde{\omega}_{[-m-1,m+1]}}\mu_p(\Tilde{\omega}_{[-m-1,m+1]})\mathds{1}_{T_{[-m,m]}(\Tilde{\omega}_{[-m-1,m+1]})=\Tilde{\omega}'_{[-m,m]}}}{256\sum_{\omega_{[-m-1,m+1]}}\mu_p(\omega_{[-m-1,m+1]})\mathds{1}_{T_{[-m,m]}(\omega_{[-m-1,m+1]})=\omega'_{[-m,m]}}}.
\end{align*}
Switching numerator and denominator in the above computation leads to the following upper bound of \eqref{eq: Fraction Uniqueness}
\begin{align*}
    \frac{256\sum_{\Tilde{\omega}_{[-m-1,m+1]}}\mu_p(\Tilde{\omega}_{[-m-1,m+1]})\mathds{1}_{T_{[-m,m]}(\Tilde{\omega}_{[-m-1,m+1]})=\Tilde{\omega}'_{[-m,m]}}}{p^4\sum_{\omega_{[-m-1,m+1]}}\mu_p(\omega_{[-m-1,m+1]})\mathds{1}_{T_{[-m,m]}(\omega_{[-m-1,m+1]})=\omega'_{[-m,m]}}}.
\end{align*}
Note that the upper and lower bound do not depend on the given boundary condition outside $\Lambda$. Consequently, the constant in \eqref{Ineq: finite energy cond rewritten} can be chosen as $C:=2^{-16}\cdot p^8$.
\end{proof}

\begin{remark}
    Note that Proposition 8.38 in \cite{Ge11} allowed us to choose for each cylinder event $A$ an appropriate $\Lambda \Subset \Z$ such that the inequality \eqref{Ineq: finite energy cond} is valid. This subtlety was really needed in the proof of Proposition \ref{thm: Uniqueness of the Gibbs measure}.
\end{remark}

\section{Proofs for one-sided view and for GHoC view} \label{Sec: Proofs one-sided view}

\subsection{Theorem \ref{thm: g-function for mu'_p}: Computation of the g-function for $\mu_p'$}\label{Sec: Image measure is g-measure}

\begin{proof}[Proof of Theorem \ref{thm: g-function for mu'_p}]
We first employ the idea of Theorem 4.19 in \cite{FeMa04}, which provides a method to obtain an LIS $\rho_p$ from the specification $\gamma'_p$. Define 
\begin{equation}\label{eq: Idea construction g-function}
    \rho_{p,\{0\}}(\sigma_0=x|\omega'):=\lim_{k\rightarrow \infty} \gamma'_{p,[0,k]}(\sigma_0=x|\omega')
\end{equation}
where $x\in \{0,1\}$ and $\omega'\in \Omega'$ with $\omega'_{\{0\}}=x$. Provided that this limit exists independently of the choice of $\mu_p'$-almost all $\omega' \in \Omega'$, one can conclude that $\mu_p'$ is consistent with the single-site kernels, as the identity 
\begin{equation*}
    (\mu'_p~ \rho_{p,\{0\}})(h)=\int_{\Omega'} \rho_{p,\{0\}}(h|\omega') \mu'_p(d\omega')=\lim_{k\rightarrow \infty}\int_{\Omega'} \gamma'_{p,[0,k]}(h|\omega') \mu'_p(d\omega')=\mu'_p(h)
\end{equation*}
holds for all $\mathscr{F}'_{\leq 0}$-measurable functions $h$. Here, we applied the dominated convergence theorem together with the consistency of $\mu'_p$ with the specification $\gamma'_{p,\Lambda}$ for all $\Lambda \Subset \Z$. Observe that, by the translational invariance of the model, $\rho_{p,\{i\}}$ is independent of $i$, and in particular $=\rho_{p,\{0\}}=\rho_{p,\{i\}}$ for all $i\in \Z$. Consequently, the kernel in \eqref{eq: Idea construction g-function} determines a whole LIS $\rho_p$ which admits $\mu_p'$ as a consistent measure, see Theorem 3.2 in \cite{FeMa04} or Theorem 3.1 in \cite{FeMa05}.

    Assume first that $\omega'\in \Omega'$ and $\omega'_{\{0\}}=x_{\{0\}}$. We compute the fraction 
    \begin{equation}\label{eq: g-function fraction}
        \frac{\gamma'_{p,[0,k]}(\sigma_0=1|\omega')}{\gamma'_{p,[0,k]}(\sigma_0=0|\omega')}=\frac{\sum_{\Tilde{\omega}'_{[1,k]}}\gamma'_{p,[0,k]}(1_{\{0\}}\Tilde{\omega}'_{[1,k]}|\omega')}{\sum_{\hat{\omega}'_{[1,k]}}\gamma'_{p,[0,k]}(0_{\{0\}}\hat{\omega}'_{[1,k]}|\omega')}
    \end{equation}
    for all $k\in \N$, which has the advantage that the partition function of $\gamma'_{p,[0,k]}(\cdot|\omega')$ does not have to be treated. If the limit $k\rightarrow \infty$ of \eqref{eq: g-function fraction} exists, then it is equal to $\frac{1-g_p(n)}{g_p(n)}$. Rearranging this identity yields $g_p(n)=\big(1+\frac{1-g_p(n)}{g_p(n)}\big)^{-1}$. We will discuss the cases $n=1$, $n=2$ and $n=\infty$ separately. 
    
    \textbf{Case 1.} ($n=1$) In this case, the spins of $T^{-1}(\omega')$ are fixed at the vertices $[-2,0]$. Let 
    \begin{equation}\label{eq: definition m g-function}
        m:=m(k):=\min\big\{i\in \N_{>k}:~(\omega'_{\{i\}},\omega'_{\{i+1\}})=(1,1)\big\}
    \end{equation}
    denote the index of the first pair of adjacent occupied sites strictly after $k$. We set $m(k):=k$ if $(\omega'_{\{k+1\}},\omega'_{\{k+2\}})=(1,0)$ and note that the set in \eqref{eq: definition m g-function} is not empty for $\mu_p'$-almost every $\omega'\in \Omega'$. Furthermore, we assume that $k+1<m(k)<\infty$ for all $k\in \N$. Therefore, we can rewrite \eqref{eq: g-function fraction} with Theorem \ref{thm: Existence image specification} as follows 
    \begin{equation*}
       \frac{pZ(U_0)+p^2Z(U_1)+p^3\sum_{l=2}^{k}Z(U_l)}{Z(U_{-1})+(1-p)p^2\sum_{l=2}^{k}Z(U_l)}.
    \end{equation*}
    Here, we sum over all possible outer unfixed areas $U_l:=(l+1,m-1)$ with $-2\leq l\leq k$ that could influence the spins in $[0,k]$, compare Definition \ref{def: restricted unfixed areas} and Figure \ref{fig: Explanation proof g-function 1} for an illustration of these areas. The prefactors of the partition functions correspond to the remaining weights at the boundary of the unfixed regions, where the spins are fixed. Furthermore, note that we sum over $\omega'_{[1,l-2]}\in \Omega'_{[1,l-2]}$ outside of each unfixed area $(l+1,m-1)$ and hence the remaining weights of \eqref{eq: g-function fraction} equals one due to the fact that 
    \begin{equation}\label{eq: probability measure fixed Bc.}
        \sum_{ \omega'_\Lambda}\sum_{\Tilde{\omega}_\Lambda}\mu_p(\Tilde{\omega}_\Lambda)\mathds{1}_{T_\Lambda(\Tilde{\omega}_\Lambda\omega_{\Lambda^c})=\omega'_{\Lambda}}=1
    \end{equation}
    for all connected $\Lambda \Subset \Z$, all $\omega'\in \Omega'$ and all fixed first-layer boundary conditions $\omega\in T^{-1}(\omega')$. 
    Substituting the definition $Z(U)=(1-p)^{-1}Q^{|U|+1}(0,0)$ together with the explicit form of the entries of $Q^m$, given in \eqref{eq: N-th power Q}, leads to
    \begin{equation*}
        \frac{p\lambda_{\text{PF}}^{m-2}\big(\lambda_{\text{PF}}-a^{m-2}\lambda_r\big)+p^2\lambda_{\text{PF}}^{m-3}\big(\lambda_{\text{PF}}-a^{m-3}\lambda_r\big)+p^3\sum^{k}_{l=2}\lambda_{\text{PF}}^{m-l-2}\big(\lambda_{\text{PF}}-a^{m-l-2}\lambda_r\big)}{\lambda_{\text{PF}}^{m-1}\big(\lambda_{\text{PF}}-a^{m-1}\lambda_r\big)+(1-p)p^2  \sum_{l=2}^{k}\lambda_{\text{PF}}^{m-l-2}\big(\lambda_{\text{PF}}-a^{m-l-2}\lambda_r\big)}.
    \end{equation*}

\begin{figure}[ht!]
\centering
\scalebox{0.5}{
\begin{tikzpicture}[every label/.append style={scale=1.3},fill fraction/.style={path picture={
\fill[#1] 
(path picture bounding box.south) rectangle
(path picture bounding box.north west);
}},
fill fraction/.default=gray!50]
\node[shape=circle,draw=orange, fill=orange, minimum size=0.5cm] (-B) at (7,0) {};
    \node[shape=circle,draw=black, minimum size=0.5cm] (-C) at (6,0) {};
    \node[shape=circle,draw=black, minimum size=0.5cm] (-D) at (5,0) {};
    \node[shape=circle,draw=black,  minimum size=0.5cm] (-E) at (4,0) {};
    \node[shape=circle,draw=black, fill=black,  minimum size=0.5cm,label=above:{$0$}] (-F) at (3,0) {};
    \node[shape=circle,draw=black,fill=black,  minimum size=0.5cm] (-G) at (2,0) {}; 
    \node[shape=circle,draw=black, fill=black, minimum size=0.5cm] (-H) at (1,0) {};
    \node[shape=circle,draw=orange, fill=orange, minimum size=0.5cm,label=above:{$l$}] (A) at (8,0) {};
    \node[shape=circle,draw=black, minimum size=0.5cm] (B) at (9,0) {};
    \node[shape=circle,draw=black, fill fraction=black, minimum size=0.5cm] (C) at (10,0) {};
    \node[shape=circle,draw=black,fill fraction=black, minimum size=0.5cm,label=above:{$k$}] (D) at (11,0) {};
    \node[shape=circle,draw=black,fill fraction=black, minimum size=0.5cm] (E) at (12,0) {};
    \node[shape=circle,draw=black,fill fraction=black, minimum size=0.5cm] (F) at (13,0) {};
    \node[shape=circle,draw=black, minimum size=0.5cm] (G) at (14,0) {};
    \node[shape=circle,draw=black, fill=black, minimum size=0.5cm,label=above:{$m$}] (H) at (15,0) {};
    \node[shape=circle,draw=black, fill=black, minimum size=0.5cm] (I) at (16,0) {};

    \node[minimum size=0.5, label=:{\text{I}}] () at (-0.5,-0.5) {};

    \draw[-] (-H)--(0.3,0){};
    \draw [-] (-G) -- (-H);
    \draw [-] (-F) -- (-G);
    \draw [-] (-F) -- (-E);
    \draw [-] (-E) -- (-D);
    \draw [-] (-D) -- (-C);
    \draw [-] (-C) -- (-B);
    \draw [-] (-B) -- (A);
    \draw [-] (A) -- (B);
    \draw [-] (B) -- (C);
    \draw [-] (C) -- (D);
    \draw [-] (D) -- (E);
    \draw [-] (E) -- (F);
    \draw [-] (F) -- (G);
    \draw [-] (G) -- (H);
    \draw [-] (H) -- (I);
    \draw[-](I)--(16.7,0){};

    \node[shape=circle,draw=orange, fill=orange, minimum size=0.5cm] (-B) at (7,-2) {};
    \node[shape=circle,draw=black, minimum size=0.5cm] (-C) at (6,-2) {};
    \node[shape=circle,draw=black, minimum size=0.5cm] (-D) at (5,-2) {};
    \node[shape=circle,draw=black,  minimum size=0.5cm] (-E) at (4,-2) {};
    \node[shape=circle,draw=black,  minimum size=0.5cm,label=above:{$0$}] (-F) at (3,-2) {};
    \node[shape=circle,draw=black,fill=black,  minimum size=0.5cm] (-G) at (2,-2) {}; 
    \node[shape=circle,draw=black, fill=black, minimum size=0.5cm] (-H) at (1,-2) {};
    \node[shape=circle,draw=orange, fill=orange, minimum size=0.5cm,label=above:{$l$}] (A) at (8,-2) {};
    \node[shape=circle,draw=black,minimum size=0.5cm] (B) at (9,-2) {};
    \node[shape=circle,draw=black,fill fraction=black,minimum size=0.5cm] (C) at (10,-2) {};
    \node[shape=circle,draw=black,fill fraction=black, minimum size=0.5cm,label=above:{$k$}] (D) at (11,-2) {};
    \node[shape=circle,draw=black,fill fraction=black, minimum size=0.5cm] (E) at (12,-2) {};
    \node[shape=circle,draw=black,fill fraction=black, minimum size=0.5cm] (F) at (13,-2) {};
    \node[shape=circle,draw=black, minimum size=0.5cm] (G) at (14,-2) {};
    \node[shape=circle,draw=black, fill=black, minimum size=0.5cm,label=above:{$m$}] (H) at (15,-2) {};
    \node[shape=circle,draw=black, fill=black, minimum size=0.5cm] (I) at (16,-2) {};

    \node[minimum size=0.5, label=:{\text{II}}] () at (-0.5,-2.5) {};

    \draw[-] (-H)--(0.3,-2){};
    \draw [-] (-G) -- (-H);
    \draw [-] (-F) -- (-G);
    \draw [-] (-F) -- (-E);
    \draw [-] (-E) -- (-D);
    \draw [-] (-D) -- (-C);
    \draw [-] (-C) -- (-B);
    \draw [-] (-B) -- (A);
    \draw [-] (A) -- (B);
    \draw [-] (B) -- (C);
    \draw [-] (C) -- (D);
    \draw [-] (D) -- (E);
    \draw [-] (E) -- (F);
    \draw [-] (F) -- (G);
    \draw [-] (G) -- (H);
    \draw [-] (H) -- (I);
    \draw[-](I)--(16.7,-2){};

\end{tikzpicture}
}
\small \caption{Illustrated are two configurations considered in the proof for the explicit expression for $g_p(1)$. The top (bottom) line shows 
one of the configurations which are summed over in the numerator (denominator) on the r.h.s. of \eqref{eq: g-function fraction}.
In the first configuration, the spin at the origin is occupied, whereas in the second it is unoccupied. In both cases, the maximum of the observation window $k$ and the first adjacent pair of occupied sites after $k$, starting from $m=m(k)$, are indicated. Furthermore, each configuration contains the outer unfixed area $U_l=(l+1,m-1)$, marked by half-black, half-white dots, which follows a pair of fixed occupied sites at $\{l-1,l\}$, highlighted in orange. The proof examines the limit $k\rightarrow \infty$ of the kernels along an arbitrary sequence $m(k)$.
}
\label{fig: Explanation proof g-function 1}
\end{figure}
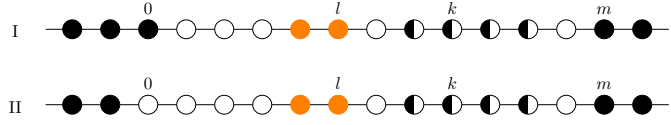
    
   \noindent However, this expression converges to $\frac{p}{1-p}$ as $k\rightarrow\infty $ since the difference
    \begin{align}\label{eq: Proof g-function 1}
    \nonumber&\Bigg|\frac{\lambda_{\text{PF}}^{m-2}\big(\lambda_{\text{PF}}-a^{m-2}\lambda_r\big)+p\lambda_{\text{PF}}^{m-3}\big(\lambda_{\text{PF}}-a^{m-3}\lambda_r\big)+p^2\sum^{k}_{l=2}\lambda_{\text{PF}}^{m-l-2}\big(\lambda_{\text{PF}}-a^{m-l-2}\lambda_r\big)}{(1-p)^{-1}\lambda_{\text{PF}}^{m-1}\big(\lambda_{\text{PF}}-a^{m-1}\lambda_r\big)+p^2  \sum_{l=2}^{k}\lambda_{\text{PF}}^{m-l-2}\big(\lambda_{\text{PF}}-a^{m-l-2}\lambda_r\big)}-1\Bigg|\\
    &=\frac{\big|\lambda_{\text{PF}}^{m-2}\big(\lambda_{\text{PF}}-a^{m-2}\lambda_r\big)+p\lambda_{\text{PF}}^{m-3}\big(\lambda_{\text{PF}}-a^{m-3}\lambda_r\big)-(1-p)^{-1}\lambda_{\text{PF}}^{m-1}\big(\lambda_{\text{PF}}-a^{m-1}\lambda_r\big)\big|}{(1-p)^{-1}\lambda_{\text{PF}}^{m-1}\big(\lambda_{\text{PF}}-a^{m-1}\lambda_r\big)+p^2  \sum_{l=2}^{k}\lambda_{\text{PF}}^{m-l-2}\big(\lambda_{\text{PF}}-a^{m-l-2}\lambda_r\big)}
    \end{align}
    converges to zero uniformly in the choice of the sequence $m=m(k)$.
    Note that we can rewrite the sum in the denominator of \eqref{eq: Proof g-function 1} as follows 
    \begin{equation*}
        \sum^{k}_{l=2}\lambda_{\text{PF}}^{m-l-2}\big(\lambda_{\text{PF}}-a^{m-l-2}\lambda_r\big)=\lambda_{\text{PF}}^{m-k-1}\frac{1-\lambda_{\text{PF}}^{k-1}}{1-\lambda_{\text{PF}}}-\lambda_{r}^{m-k-1}\frac{1-\lambda_{r}^{k-1}}{1-\lambda_r}.
    \end{equation*}
    Consequently, \eqref{eq: Proof g-function 1} can be written as 
    \begin{equation*}
        \frac{\big|\lambda_{\text{PF}}^{k-1}\big(\lambda_{\text{PF}}-a^{m-2}\lambda_r\big)+p\lambda_{\text{PF}}^{k-2}\big(\lambda_{\text{PF}}-a^{m-3}\lambda_r\big)-(1-p)^{-1}\lambda_{\text{PF}}^{k}\big(\lambda_{\text{PF}}-a^{m-1}\lambda_r\big)\big|}{(1-p)^{-1}\lambda_{\text{PF}}^{k}\big(\lambda_{\text{PF}}-a^{m-1}\lambda_r\big)+p^2\big((1-\lambda_{\text{PF}}^{k-1})\cdot(1-\lambda_{\text{PF}})^{-1}-a^{m-k-1}(1-\lambda_{r}^{k-1})\cdot(1-\lambda_r)^{-1}\big)}
    \end{equation*}
    which converges to zero as $k\rightarrow \infty$ independently of the sequence $m>k+1$. In order to see this, note that the numerator is upper bounded by $3(1-p)^{-1}(\lambda_{\text{PF}}+|\lambda_r|)\lambda_{\text{PF}}^{k-2}$, since $\lambda_{\text{PF}}\in (0,1)$ and $a\in (-1,0)$. The denominator is bounded from below by $p^2\big((1-\lambda_{\text{PF}}^{k-1})\cdot(1-\lambda_{\text{PF}})^{-1}-(1-\lambda_{r}^{k-1})\cdot(1-\lambda_r)^{-1}\big)$, where we used the fact that $|\lambda_r|\leq\lambda_{\text{PF}}$. This lower bound is positive for sufficiently large $k$.  

    If $m=k$, then the configuration satisfies $(\omega'_{k+1},\omega'_{k+2})=(1,0)$, so the spin at $\{k\}$ is fixed. If $m=k+1$, we have $(\omega'_{k+1},\omega'_{k+2})=(1,1)$ and the spin at $\{k+1\}$ is fixed. Furthermore, in the numerator the spin at site $0$ is fixed to be \say{occupied} (giving a weight $p$), whereas in the denominator it is fixed to be \say{empty} (giving a weight $1-p$). Summing over all possible states in the intermediate regions $[1,k-1]$ and $[1,k]$, respectively, we obtain that \eqref{eq: g-function fraction} equals $\frac{p}{(1-p)}$ for all $k\in \N$ and $m\in \{k,k+1\}$, using  \eqref{eq: probability measure fixed Bc.}. Consequently, the limit of \eqref{eq: g-function fraction} as  $k\rightarrow \infty$ exists, independently of $\omega'\in \Omega'$, and we obtain $g_p(1)=1-p$.

      \textbf{Case 2.} ($2<n<\infty$) In this situation, the spins at the vertices $[-n-1,-n+1]$ are fixed and the fraction in \eqref{eq: g-function fraction} becomes
    \begin{equation}\label{eq: g-function fraction 2}
       \frac{\sum_{\Tilde{\omega}'_{[2,k]}}\gamma'_{p,[0,k]}(1_{\{0,1\}}\Tilde{\omega}'_{[2,k]}|\omega')}{\sum_{\hat{\omega}'_{[1,k]}}\gamma'_{p,[0,k]}(0_{\{0\}}\hat{\omega}'_{[1,k]}|\omega')}.
    \end{equation}
    Note that the non-isolation constraint forces $\Tilde{\omega}'_{\{1\}}=1$, because $\omega'_{\{-1\}}=0$. Again, we assume $k+1<m<\infty$, where $m$ is defined in \eqref{eq: definition m g-function}. Following the approach of Case 1, we sum in the numerator over all outer unfixed areas of the form $U_j:=(j+1,m-1)$ with $1\leq j\leq k$, located to the right of $[0,k]$. Additionally, we consider the weights on the outer unfixed area $(-n+1,-1)$ on the left side of the observation window. In the denominator, we consider all outer unfixed areas $V_l:=(-n+1,l-2)$ for $2\leq l\leq k$ and separately the case $(-n+1,m-1)$. For each unfixed area $V_l$ on the left side of $[0,k]$, we also have to discuss the outer unfixed areas $U_j$ for $l\leq j\leq k$, located to the right of $[0,k]$, see Figure \ref{fig: Explanation plot g-function 2} for an illustration. Finally, since we sum over all possible configurations  $\omega'_{[2,j-2]}$ and $\omega'_{[l+1,j-2]}$ in the intermediate regions of the numerator and denominator, respectively, their weights sum to one. Therefore, the expression in \eqref{eq: g-function fraction 2} can be rewritten as follows
    \begin{align}\label{eq: Fraction g-function 3}
    \nonumber&p^2(1-p)Z(V_1)\cdot\Big((1-p)Z(U_1)+p(1-p)Z(U_2)+p^2(1-p)\sum_{j=3}^kZ(U_j)\Big)\Big/\Big(Z(V_{m+1})+\\
    &+p^2(1-p)^2\sum_{l=2}^{k} Z(V_l) Z(U_l)+(1-p)^2p^3\sum_{l=3}^{k}Z(V_{l-1})Z(U_l)+(1-p)^2p^4\sum^{k-2}_{l=2}Z(V_l)\sum^{k}_{j=l+2}Z(U_j)\Big).
    \end{align}
    In order to compute the limit as $k\rightarrow\infty$, let us prove the following lemma.

    \begin{figure}[ht!]
\centering
\scalebox{0.5}{
\begin{tikzpicture}[every label/.append style={scale=1.3},fill fraction/.style={path picture={
\fill[#1] 
(path picture bounding box.south) rectangle
(path picture bounding box.north west);
}},
fill fraction/.default=gray!50]
\node[shape=circle,draw=black,  minimum size=0.5cm] (-B) at (7,0) {};
    \node[shape=circle,draw=black, minimum size=0.5cm] (-C) at (6,0) {};
    \node[shape=circle,draw=black, minimum size=0.5cm] (-D) at (5,0) {};
    \node[shape=circle,draw=black, fill=black,  minimum size=0.5cm] (-E) at (4,0) {};
    \node[shape=circle,draw=black, fill=black,  minimum size=0.5cm,label=above:{$0$}] (-F) at (3,0) {};
    \node[shape=circle,draw=black,  minimum size=0.5cm] (-G) at (2,0) {}; 
    \node[shape=circle,draw=black, fill fraction=black, minimum size=0.5cm] (-H) at (1,0) {};
    \node[shape=circle,draw=black, fill fraction=black, minimum size=0.5cm] (-I) at (0,0) {};
    \node[shape=circle,draw=black, minimum size=0.5cm] (-J) at (-1,0) {};
    \node[shape=circle,draw=black, fill=black, minimum size=0.5cm,label=above:{$-n$}] (-K) at (-2,0) {};
    \node[shape=circle,draw=black, fill=black, minimum size=0.5cm] (-L) at (-3,0) {};
    \node[shape=circle,draw=black,  minimum size=0.5cm] (A) at (8,0) {};
    \node[shape=circle,draw=black, minimum size=0.5cm] (B) at (9,0) {};
    \node[shape=circle,draw=black,  minimum size=0.5cm] (C) at (10,0) {};
    \node[shape=circle,draw=cyan,fill=cyan,  minimum size=0.5cm] (D) at (11,0) {};
    \node[shape=circle,draw=cyan, fill=cyan, minimum size=0.5cm,label=above:{$j$}] (E) at (12,0) {};
    \node[shape=circle,draw=black,  minimum size=0.5cm] (F) at (13,0) {};
    \node[shape=circle,draw=black, fill fraction=black, minimum size=0.5cm,label=above:{$k$}] (G) at (14,0) {};
    \node[shape=circle,draw=black, fill fraction=black,minimum size=0.5cm] (H) at (15,0) {};
    \node[shape=circle,draw=black, fill fraction=black, minimum size=0.5cm] (I) at (16,0) {};
    \node[shape=circle,draw=black,minimum size=0.5cm] (J) at (17,0) {};
     \node[shape=circle,draw=black, fill=black, minimum size=0.5cm,label=above:{$m$}] (K) at (18,0) {};
      \node[shape=circle,draw=black, fill=black, minimum size=0.5cm] (L) at (19,0) {};

    \node[minimum size=0.5, label=:{\text{I}}] () at (-4.5,-0.5) {};

    \draw[-] (-L)--(-3.7,0){};
    \draw [-] (-K) -- (-L);
    \draw [-] (-J) -- (-K);
    \draw [-] (-I) -- (-J);
    \draw [-] (-H) -- (-I);
    \draw [-] (-G) -- (-H);
    \draw [-] (-F) -- (-G);
    \draw [-] (-F) -- (-E);
    \draw [-] (-E) -- (-D);
    \draw [-] (-D) -- (-C);
    \draw [-] (-C) -- (-B);
    \draw [-] (-B) -- (A);
    \draw [-] (A) -- (B);
    \draw [-] (B) -- (C);
    \draw [-] (C) -- (D);
    \draw [-] (D) -- (E);
    \draw [-] (E) -- (F);
    \draw [-] (F) -- (G);
    \draw [-] (G) -- (H);
    \draw [-] (H) -- (I);
    \draw [-] (I) -- (J);
    \draw [-] (J) -- (K);
    \draw [-] (K) -- (L);
    \draw[-](L)--(19.7,0){};

    \node[shape=circle,draw=black, minimum size=0.5cm] (-B) at (7,-2) {};
    \node[shape=circle,draw=orange, fill=orange, minimum size=0.5cm,label=above:{$l$}] (-C) at (6,-2) {};
    \node[shape=circle,draw=orange,fill=orange, minimum size=0.5cm] (-D) at (5,-2) {};
    \node[shape=circle,draw=black, minimum size=0.5cm] (-E) at (4,-2) {};
    \node[shape=circle,draw=black, fill fraction=black, minimum size=0.5cm,label=above:{$0$}] (-F) at (3,-2) {};
    \node[shape=circle,draw=black,fill fraction=black,  minimum size=0.5cm] (-G) at (2,-2) {}; 
    \node[shape=circle,draw=black,fill fraction=black, minimum size=0.5cm] (-H) at (1,-2) {};
    \node[shape=circle,draw=black,fill fraction=black, minimum size=0.5cm] (-I) at (0,-2) {};
    \node[shape=circle,draw=black, minimum size=0.5cm] (-J) at (-1,-2) {};
    \node[shape=circle,draw=black, fill=black, minimum size=0.5cm,label=above:{$-n$}] (-K) at (-2,-2) {};
    \node[shape=circle,draw=black, fill=black, minimum size=0.5cm] (-L) at (-3,-2) {};
    \node[shape=circle,draw=black,  minimum size=0.5cm] (A) at (8,-2) {};
    \node[shape=circle,draw=black, minimum size=0.5cm] (B) at (9,-2) {};
    \node[shape=circle,draw=black, minimum size=0.5cm] (C) at (10,-2) {};
    \node[shape=circle,draw=cyan, fill=cyan, minimum size=0.5cm] (D) at (11,-2) {};
    \node[shape=circle,draw=cyan, fill=cyan, minimum size=0.5cm,label=above:{$j$}] (E) at (12,-2) {};
    \node[shape=circle,draw=black, minimum size=0.5cm] (F) at (13,-2) {};
    \node[shape=circle,draw=black,fill fraction=black, minimum size=0.5cm,label=above:{$k$}] (G) at (14,-2) {};
    \node[shape=circle,draw=black,fill fraction=black,minimum size=0.5cm] (H) at (15,-2) {};
    \node[shape=circle,draw=black,fill fraction=black, minimum size=0.5cm] (I) at (16,-2) {};
    \node[shape=circle,draw=black,minimum size=0.5cm] (J) at (17,-2) {};
     \node[shape=circle,draw=black, fill=black, minimum size=0.5cm,label=above:{$m$}] (K) at (18,-2) {};
      \node[shape=circle,draw=black, fill=black, minimum size=0.5cm] (L) at (19,-2) {};

    \node[minimum size=0.5, label=:{\text{II}}] () at (-4.5,-2.5) {};

    \draw[-] (-L)--(-3.7,-2){};
    \draw [-] (-K) -- (-L);
    \draw [-] (-J) -- (-K);
    \draw [-] (-I) -- (-J);
    \draw [-] (-H) -- (-I);
    \draw [-] (-G) -- (-H);
    \draw [-] (-F) -- (-G);
    \draw [-] (-F) -- (-E);
    \draw [-] (-E) -- (-D);
    \draw [-] (-D) -- (-C);
    \draw [-] (-C) -- (-B);
    \draw [-] (-B) -- (A);
    \draw [-] (A) -- (B);
    \draw [-] (B) -- (C);
    \draw [-] (C) -- (D);
    \draw [-] (D) -- (E);
    \draw [-] (E) -- (F);
    \draw [-] (F) -- (G);
    \draw [-] (G) -- (H);
    \draw [-] (H) -- (I);
    \draw [-] (I) -- (J);
    \draw [-] (J) -- (K);
    \draw [-] (K) -- (L);
    \draw[-](L)--(19.7,-2){};

\end{tikzpicture}
}
\small \caption{An illustration of two configurations considered in the computation of $g_p(n)$ with $n=5$. The top line I shows the situation where the pair of sites $\{0,1\}$ are occupied and line II the case where they are unoccupied. In both configurations, the first adjacent pair of occupied sites before $0$, located at $\{-n-1,-n\}$, is indicated. In line II the first pair of occupied sites after the origin $\{l-1,-l\}$ is highlighted in orange. Additionally to the outer unfixed area in the future $U_j=(j+1,m-1)$, the configuration in line II possesses the outer unfixed area coming from the past $V_l=(-n+1,l-2)$.}
\label{fig: Explanation proof g-function 2}
\end{figure}
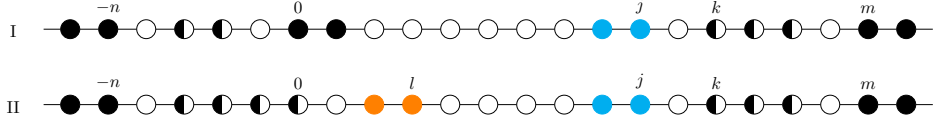

    \begin{lemma}\label{lem: Asymptotics g-functions 1}
        Equation \eqref{eq: Fraction g-function 3} can be rewritten in the form
       \begin{equation}\label{eq: asymptotics fraction g-function}
           \frac{a_k+Q^{n-2}(0,0)\sum_{j=3}^kQ^{m-j-2}(0,0)}{b_k+\sum^{k-2}_{l=2}Q^{n+l-3}(0,0)\sum^{k}_{j=l+2}Q^{m-j-2}(0,0)}
       \end{equation}
       where $a_k$ and $b_k$ are defined in \eqref{eq: Definition ak} and \eqref{eq: Definition bk}, respectively. Both sequences satisfy $\lim_{k\rightarrow \infty}\lambda_{\text{PF}}^{k-m}a_k=\lim_{k\rightarrow \infty}\lambda_{\text{PF}}^{k-m}b_k=0$, where we recall that $m=m(k)$, see \eqref{eq: definition m g-function}.
    \end{lemma}

\begin{proof}
       We define 
       \begin{equation}\label{eq: Definition ak}
           a_k:=Q^{n-2}(0,0)\cdot\Big(p^{-2}Q^{m-3}(0,0)+p^{-1}Q^{m-4}(0,0)\Big)
       \end{equation}
       and 
       \begin{align}\label{eq: Definition bk}
           \nonumber b_k&:=p^{-4}(1-p)^{-1}Q^{m+n-2}(0,0)+p^{-2}\sum^{k}_{l=2}Q^{n+l-3}(0,0)Q^{m-l-2}(0,0)\\
           &+p^{-1}\sum^k_{l=3}Q^{n+l-4}(0,0)Q^{m-l-2}(0,0).
       \end{align}
       Inserting the definition $Z(U)=(1-p)^{-1}Q^{|U|+1}(0,0)$ for each unfixed region $U$ in \eqref{eq: Fraction g-function 3}, together with the definitions of $a_k$ and $b_k$ leads to the desired expression for the fraction in \eqref{eq: asymptotics fraction g-function}. Further, substitute the entries of $Q^m$, see \eqref{eq: N-th power Q}, into $a_k$ yields 
       \begin{align*}
           a_k&=\frac{\lambda_{\text{PF}}^{n+m-5}(\lambda_{\text{PF}}-a^{n-2}\lambda_r)\Big(p^{-2}(\lambda_{\text{PF}}-a^{m-3}\lambda_r)+p^{-1}\lambda_{\text{PF}}^{-1}(\lambda_{\text{PF}}-a^{m-2}\lambda_r)\Big)}{(1-p)(3p+1)}
       \end{align*}
       which satisfies $\lim_{k\rightarrow \infty} \lambda_{\text{PF}}^{k-m} a_k =0$, since $0<\lambda_{\text{PF}}<1$ and $|a|<1$. With the definition of $Q^m$, one can immediately verify that $\lim_{k\rightarrow \infty} \lambda_{\text{PF}}^{k-m} Q^{m+n-2}(0,0)=0$. Furthermore, we can rewrite
       \begin{equation*}
           \sum^{k}_{l=2}Q^{n+l-3}(0,0)Q^{m-l-2}(0,0)=\frac{\lambda_{\text{PF}}^{n+m-5}\sum^k_{l=2}(\lambda_{\text{PF}}-a^{n+l-3}\lambda_r)(\lambda_{\text{PF}}-a^{m-l-2}\lambda_r)}{(1-p)(3p+1)}.
       \end{equation*}
       Note that the numerator is upper bounded by $k\lambda_{\text{PF}}^{n+m-5}(\lambda_{\text{PF}}+|\lambda_r|)^2$
and consequently, we have $\lim_{k\rightarrow \infty}\lambda_{\text{PF}}^{k-m}\sum^{k}_{l=2}Q^{n+l-3}(0,0)Q^{m-l-2}(0,0)=0$. The remaining sum in $b_k$ can be treated in the same way, which completes the proof of $\lim_{k\rightarrow \infty}\lambda_{\text{PF}}^{k-m}b_k=0$.
\end{proof}

Let us investigate the remaining terms of \eqref{eq: asymptotics fraction g-function}. 

\begin{lemma}\label{lem: Asymptotics g-functions 2}
    For all $n>2$ and $k\in \N$, the following relations hold: 
    \begin{equation}\label{eq: Asymptotics g-function 2}
    Q^{n-2}(0,0)\sum_{j=3}^kQ^{m-j-2}(0,0)=\Tilde{a}_k+c_k\cdot Q^{n-2}(0,0)
\end{equation}
and 
\begin{align}\label{eq: Asymptotics g-function 3}
    \sum^{k-2}_{l=2}Q^{n+l-3}(0,0)\sum^{k}_{j=l+2}Q^{m-j-2}(0,0)=\Tilde{b}_k+c_k\cdot\sum^{k-2}_{l=2}Q^{n+l-3}(0,0)
\end{align}
where $\Tilde{a}_k$ and $\Tilde{b}_k$ are defined in \eqref{eq: Definition Tilde(ak) g-function} and \eqref{eq: Definition Tilde(bk) g-function}, respectively. The sequence $c_k$, given in \eqref{eq: Definition ck g-function}, satisfies that $c_k\lambda_{\text{PF}}^{k-m}$ is bounded and positive. Moreover, we have $\lim_{k\rightarrow\infty}\lambda_{\text{PF}}^{k-m}\Tilde{a}_k=\lim_{k\rightarrow\infty}\lambda_{\text{PF}}^{k-m}\Tilde{b}_k=0$.
\end{lemma}
\begin{proof}
We begin by carrying out the finite geometric sums over $j$
\begin{equation*}
    \sum^k_{j=3}Q^{m-j-2}(0,0)=\frac{\lambda_{\text{PF}}^{m-k}}{\sqrt{(1-p)(3p+1)}}\Big(\frac{1-\lambda_{\text{PF}}^{k-2}}{\lambda_{\text{PF}}(1-\lambda_{\text{PF}})}-a^{m-k}\frac{1-\lambda_{r}^{k-2}}{\lambda_{\text{r}}(1-\lambda_{\text{r}})}\Big).
\end{equation*}
Define
\begin{equation}\label{eq: Definition Tilde(ak) g-function}
    \Tilde{a}_k:=Q^{n-2}(0,0)\cdot\frac{\lambda_{\text{PF}}^{m-k}}{\sqrt{(1-p)(3p+1)}}\Big(a^{m-k}\frac{\lambda_{r}^{k-2}}{\lambda_{\text{r}}(1-\lambda_{\text{r}})}-\frac{\lambda_{\text{PF}}^{k-2}}{\lambda_{\text{PF}}(1-\lambda_{\text{PF}})}\Big)
\end{equation}
which satisfies $\lim_{k\rightarrow \infty}\lambda_{\text{PF}}^{k-m}\Tilde{a}_k=0$, since $|a|,|\lambda_r|,\lambda_{\text{PF}}\in (0,1)$. 
Equation \eqref{eq: Asymptotics g-function 2}, then follows by defining
\begin{equation}\label{eq: Definition ck g-function}
    c_k:=\frac{\lambda_{\text{PF}}^{m-k}}{\sqrt{(1-p)(3p+1)}}\Big(\frac{1}{\lambda_{\text{PF}}(1-\lambda_{\text{PF}})}-a^{m-k}\frac{1}{\lambda_{\text{r}}(1-\lambda_{\text{r}})}\Big).
\end{equation}
The sequence $\lambda_{\text{PF}}^{k-m}c_k$ is bounded because $a\in (-1,0)$ and $m-k\geq 0$. Furthermore, it is positive since it is lower bounded by $1/(p^2\sqrt{(1-p)(3p+1)})$. This completes the proof of \eqref{eq: Asymptotics g-function 2}.

To prove the second part of Lemma \ref{lem: Asymptotics g-functions 2}, we insert the entries of $Q^m$ and perform the geometric summations which leads to 
\begin{equation*}
    \sum^{k}_{j=l+2}Q^{m-j-2}(0,0)=c_k+\frac{\lambda_{\text{PF}}^{m-k}}{\sqrt{(1-p)(3p+1)}}\Big(a^{m-k}
    \frac{\lambda_{r}^{k-l-1}}{\lambda_{r}(1-\lambda_{r})}-\frac{\lambda_{\text{PF}}^{k-l-1}}{\lambda_{\text{PF}}(1-\lambda_{\text{PF}})}\Big),
\end{equation*}
where $c_k$ was defined in \eqref{eq: Definition ck g-function}. We obtain \eqref{eq: Asymptotics g-function 3} by defining
\begin{equation}\label{eq: Definition Tilde(bk) g-function}
   \Tilde{b}_k:= \frac{\lambda_{\text{PF}}^{m-k}}{\sqrt{(1-p)(3p+1)}}\sum^{k-2}_{l=2}Q^{n+l-3}(0,0)\Big(a^{m-k}
    \frac{\lambda_{r}^{k-l-1}}{\lambda_{r}(1-\lambda_{r})}-\frac{\lambda_{\text{PF}}^{k-l-1}}{\lambda_{\text{PF}}(1-\lambda_{\text{PF}})}\Big)
\end{equation}
for all $k\in \N$. Substituting the entries of $Q^m$ and the expression reads
\begin{align*}
   \frac{\lambda_{\text{PF}}^{m-k}}{(1-p)(3p+1)}\Bigg(\frac{a^{m-k}\lambda^{k-2}_r\lambda_{\text{PF}}^{n-3}}{1-\lambda_r}\sum^{k-2}_{l=2}a^{-l}(\lambda_{\text{PF}}-a^{n+l-3}\lambda_r)-\frac{\lambda_{\text{PF}}^{n+k-5}}{1-\lambda_{\text{PF}}}\sum^{k-2}_{l=2}(\lambda_{\text{PF}}-a^{n+l-3}\lambda_r)\Bigg).
\end{align*}
From the fact that $(\lambda_{\text{PF}}-a^m\lambda_r)\leq 1$ for all $m\in \N_0$, we have the upper bound
\begin{equation*}
    |\Tilde{b}_k|\leq \frac{\lambda_{\text{PF}}^{m-k}}{(1-p)(3p+1)}\Bigg(\frac{|\lambda_r|^{k-2}(1-|a|^{-k+5})}{a^2(1-\lambda_r)(1-|a|^{-1})}+\frac{(k-3)\lambda_{\text{PF}}^{n+k-5}}{1-\lambda_{\text{PF}}}\Bigg).
\end{equation*}
Consequently, this implies $\lim_{k\rightarrow \infty} \lambda_{\text{PF}}^{k-m}\Tilde{b}_k=0$, since $|a|^{-k}|\lambda_r|^{k}=\lambda_{\text{PF}}^{k}$. This completes the proof of Lemma \ref{lem: Asymptotics g-functions 2}.
\end{proof}

Combining Lemma \ref{lem: Asymptotics g-functions 1} and \ref{lem: Asymptotics g-functions 2} implies that the fraction in \eqref{eq: g-function fraction 2} converges to 
\begin{equation}\label{eq: g-function fraction limit}
    \frac{Q^{n-2}(0,0)}{\sum_{l=2}^{\infty}Q^{n+l-3}(0,0)}
    = \frac{\lambda_{\text{PF}}^{n-1}-\lambda_r^{n-1}}{\frac{\lambda_{\text{PF}}^{n}}{1-\lambda_{\text{PF}}}-\frac{\lambda_r^{n}}{1-\lambda_r}}
\end{equation}
which leads to the expression given in \eqref{eq: g-function for the image measure}. If $m=k$ or $m=k+1$, we do not need to consider the outer unfixed areas $U_j$, see Figure \ref{fig: Explanation plot g-function 2}, since the spins at $k$ and $k+1$, respectively, are fixed. Let us assume $m=k+1$, the case $m=k$ is handled analogously. In the numerator of \eqref{eq: g-function fraction 2}, only the unfixed area $V_1$ and the fixed spins in $[-1,1]$ contribute, since we sum over all possible configurations $\omega'_{[2,k]}$ in the remaining region. This leads to the term $p^2(1-p)Z(V_1)$. In the denominator of \eqref{eq: g-function fraction 2}, we sum over all unfixed areas $V_l$ for $2\leq l\leq k-1$ and treat $V_{k+2}$ separately. For each $V_l$, we additionally obtain the weights $p^2(1-p)$ corresponding to the fixed sites at $[l-2,l]$. Furthermore, we sum over all possible configurations $\omega'_{[l+1,k]}$ in the remaining region. This leads to the term $Z(V_{k+2})+p^2(1-p)\sum_{l=2}^{k-1}Z(V_l)$ in the denominator. Since $\lim_{k\rightarrow \infty}Z(V_{k+2})=0$, it follows that \eqref{eq: g-function fraction 2} converges to \eqref{eq: g-function fraction limit}.

    \textbf{Case 3.} (n=2) Note that $n=2$ implies that the sites in $[-3,-1]$ are fixed and the state $\Tilde{\omega}'_{\{1\}}$ in the numerator of \eqref{eq: g-function fraction} has to be occupied under the non-isolation constraint. Let us only give quick discussion of the case $k+1<m<\infty$. As in the proceeding of Case 2, we sum over all outer unfixed areas of the form $U_j=(j+1,m-1)$ with $1\leq j\leq k$ in the numerator. In the denominator, we consider all unfixed areas $V_l:=(-1,l-2)$ for $2\leq l\leq k$ and in particular the special case $V_{m+1}=(-1,m-1)$. For each unfixed area $V_l$, we also consider the unfixed areas $U_j$ where $l\leq j\leq k$, see Figure \ref{fig: Explanation proof g-function n=2} for an illustration. Therefore, we obtain a slightly different expression for \eqref{eq: g-function fraction} as in the second case, namely
    \begin{align*}
        &\Big(p^2(1-p)Z(U_1)+p^3(1-p)Z(U_2)+p^4(1-p)\sum_{j=3}^kZ(U_j)\Big)\Big/\Big(Z(V_{m+1})+\\
        &+p^2(1-p)^2\sum_{l=2}^{k} Z(V_l)Z(U_l)+(1-p)^2p^3\sum_{l=3}^{k}Z(V_{l-1})Z(U_l)+(1-p)^2p^4\sum^{k-2}_{l=2}Z(V_l)\sum^{k}_{j=l+2}Z(U_j)\Big).
    \end{align*}
    With the same ideas as in the second case, we obtain that this expression converges to\\ $\Big(\sum^\infty_{l=2}Q^{l-1}(0,0)\Big)^{-1}$. Inserting the entries of $Q^m$, see \eqref{eq: N-th power Q}, and performing the geometric sums, the fraction becomes
    \begin{equation*}
        \frac{\sqrt{(1-p)(3p+1)}}{\frac{\lambda_{\text{PF}}^2}{1-\lambda_{\text{PF}}}-\frac{\lambda_{\text{r}}^2}{1-\lambda_{\text{r}}}}=\frac{p^2}{1-p^2}
    \end{equation*}
    which leads to the expression \eqref{eq: g-function fraction} for $n=2$.

    \textbf{Case 4.} ($n=\infty$) 
    This case corresponds to the configuration fully empty past (which is a.s. invisible for finite $p$), and can be treated along the same lines.

    \begin{figure}[ht!]
\centering
\scalebox{0.5}{
\begin{tikzpicture}[every label/.append style={scale=1.3},fill fraction/.style={path picture={
\fill[#1] 
(path picture bounding box.south) rectangle
(path picture bounding box.north west);
}},
fill fraction/.default=gray!50]
\node[shape=circle,draw=black,  minimum size=0.5cm] (-B) at (7,0) {};
    \node[shape=circle,draw=black, minimum size=0.5cm] (-C) at (6,0) {};
    \node[shape=circle,draw=black, minimum size=0.5cm] (-D) at (5,0) {};
    \node[shape=circle,draw=black, fill=black,  minimum size=0.5cm] (-E) at (4,0) {};
    \node[shape=circle,draw=black, fill=black,  minimum size=0.5cm,label=above:{$0$}] (-F) at (3,0) {};
    \node[shape=circle,draw=black,  minimum size=0.5cm] (-G) at (2,0) {}; 
    \node[shape=circle,draw=black, fill=black, minimum size=0.5cm] (-H) at (1,0) {};
    \node[shape=circle,draw=black, fill=black, minimum size=0.5cm] (-I) at (0,0) {};
    
    \node[shape=circle,draw=black,  minimum size=0.5cm] (A) at (8,0) {};
    \node[shape=circle,draw=black, minimum size=0.5cm] (B) at (9,0) {};
    \node[shape=circle,draw=black,  minimum size=0.5cm] (C) at (10,0) {};
    \node[shape=circle,draw=cyan,fill=cyan,  minimum size=0.5cm] (D) at (11,0) {};
    \node[shape=circle,draw=cyan, fill=cyan, minimum size=0.5cm,label=above:{$j$}] (E) at (12,0) {};
    \node[shape=circle,draw=black,  minimum size=0.5cm] (F) at (13,0) {};
    \node[shape=circle,draw=black, fill fraction=black, minimum size=0.5cm,label=above:{$k$}] (G) at (14,0) {};
    \node[shape=circle,draw=black, fill fraction=black,minimum size=0.5cm] (H) at (15,0) {};
    \node[shape=circle,draw=black, fill fraction=black, minimum size=0.5cm] (I) at (16,0) {};
    \node[shape=circle,draw=black,minimum size=0.5cm] (J) at (17,0) {};
     \node[shape=circle,draw=black, fill=black, minimum size=0.5cm,label=above:{$m$}] (K) at (18,0) {};
      \node[shape=circle,draw=black, fill=black, minimum size=0.5cm] (L) at (19,0) {};

    \node[minimum size=0.5, label=:{\text{I}}] () at (-1.5,-0.5) {};

    \draw [-] (-I) -- (-0.7,0);
    \draw [-] (-H) -- (-I);
    \draw [-] (-G) -- (-H);
    \draw [-] (-F) -- (-G);
    \draw [-] (-F) -- (-E);
    \draw [-] (-E) -- (-D);
    \draw [-] (-D) -- (-C);
    \draw [-] (-C) -- (-B);
    \draw [-] (-B) -- (A);
    \draw [-] (A) -- (B);
    \draw [-] (B) -- (C);
    \draw [-] (C) -- (D);
    \draw [-] (D) -- (E);
    \draw [-] (E) -- (F);
    \draw [-] (F) -- (G);
    \draw [-] (G) -- (H);
    \draw [-] (H) -- (I);
    \draw [-] (I) -- (J);
    \draw [-] (J) -- (K);
    \draw [-] (K) -- (L);
    \draw[-](L)--(19.7,0){};

    \node[shape=circle,draw=orange, fill=orange, minimum size=0.5cm,label=above:{$l$}] (-B) at (7,-2) {};
    \node[shape=circle,draw=orange, fill=orange, minimum size=0.5cm] (-C) at (6,-2) {};
    \node[shape=circle,draw=black, minimum size=0.5cm] (-D) at (5,-2) {};
    \node[shape=circle,draw=black,fill fraction=black, minimum size=0.5cm] (-E) at (4,-2) {};
    \node[shape=circle,draw=black, fill fraction=black, minimum size=0.5cm,label=above:{$0$}] (-F) at (3,-2) {};
    \node[shape=circle,draw=black,  minimum size=0.5cm] (-G) at (2,-2) {}; 
    \node[shape=circle,draw=black,fill=black, minimum size=0.5cm] (-H) at (1,-2) {};
    \node[shape=circle,draw=black,fill=black, minimum size=0.5cm] (-I) at (0,-2) {};
    \node[shape=circle,draw=black,  minimum size=0.5cm] (A) at (8,-2) {};
    \node[shape=circle,draw=black, minimum size=0.5cm] (B) at (9,-2) {};
    \node[shape=circle,draw=black, minimum size=0.5cm] (C) at (10,-2) {};
    \node[shape=circle,draw=cyan, fill=cyan, minimum size=0.5cm] (D) at (11,-2) {};
    \node[shape=circle,draw=cyan, fill=cyan, minimum size=0.5cm,label=above:{$j$}] (E) at (12,-2) {};
    \node[shape=circle,draw=black, minimum size=0.5cm] (F) at (13,-2) {};
    \node[shape=circle,draw=black,fill fraction=black, minimum size=0.5cm,label=above:{$k$}] (G) at (14,-2) {};
    \node[shape=circle,draw=black,fill fraction=black,minimum size=0.5cm] (H) at (15,-2) {};
    \node[shape=circle,draw=black,fill fraction=black, minimum size=0.5cm] (I) at (16,-2) {};
    \node[shape=circle,draw=black,minimum size=0.5cm] (J) at (17,-2) {};
     \node[shape=circle,draw=black, fill=black, minimum size=0.5cm,label=above:{$m$}] (K) at (18,-2) {};
      \node[shape=circle,draw=black, fill=black, minimum size=0.5cm] (L) at (19,-2) {};

    \node[minimum size=0.5, label=:{\text{II}}] () at (-1.5,-2.5) {};

    \draw [-] (-I) -- (-0.7,-2);
    \draw [-] (-H) -- (-I);
    \draw [-] (-G) -- (-H);
    \draw [-] (-F) -- (-G);
    \draw [-] (-F) -- (-E);
    \draw [-] (-E) -- (-D);
    \draw [-] (-D) -- (-C);
    \draw [-] (-C) -- (-B);
    \draw [-] (-B) -- (A);
    \draw [-] (A) -- (B);
    \draw [-] (B) -- (C);
    \draw [-] (C) -- (D);
    \draw [-] (D) -- (E);
    \draw [-] (E) -- (F);
    \draw [-] (F) -- (G);
    \draw [-] (G) -- (H);
    \draw [-] (H) -- (I);
    \draw [-] (I) -- (J);
    \draw [-] (J) -- (K);
    \draw [-] (K) -- (L);
    \draw[-](L)--(19.7,-2){};

\end{tikzpicture}
}
\small \caption{An illustration of two configurations discussed in the proof of $g_p(2)$ is shown. Panel I depicts the case where $\{0,1\}$, while Panel II shows the case where these sites are unoccupied. In both panels, the maximum of the observation window $k$, the occupied pair $\{-3,-2\}$ and the first occupied pair after $k$, located at $\{m,m+1\}$, are indicated. In addition, the first occupied pair $\{j-1,j\}$ before $k$ is highlighted in blue, and in Panel II the first occupied pair after the origin $\{l-1,-l\}$, is highlighted in orange.}
\label{fig: Explanation proof g-function n=2}
\end{figure}
\end{proof}

\subsection{Proposition \ref{thm: Unique g-measure}: Uniqueness of the g-measure for the LIS}

In order to show that $\mu_p'$ is the unique $g$-measure for the LIS $\rho_p$, computed in Theorem \ref{thm: g-function for mu'_p}, we first define the \textit{$n^{\text{th}}$ variation} of a function $\bar{g}:\Omega'_{(-\infty,0]}\rightarrow (0,1)$ as 
\begin{equation}
    \text{var}_n(\bar{g}):=\sup_{\substack{\omega',\eta'\in \Omega'\\\omega'_{[-n,0]}=\eta'_{[-n,0]}}} \big| \bar{g}(\omega'_{\{0\}}\omega'_{<0})-\bar{g}(\eta'_{\{0\}}\eta'_{<0})\big|
\end{equation}
for $n \in \N_0$. We build on results in \cite{FeMa05}, namely Theorem 4.1 and Proposition 4.2, which state that the shift-invariant LIS $\rho_p$ with $g$-function $\bar{g}_p:\Omega'_{(-\infty,0]}\rightarrow (0,1)$ has a unique g-measure if the variations of $\bar{g}_p$ are summable and $\bar{g}_p$ is uniformly non-null. More precisely, the conditions are
\begin{equation}\label{eq: Condition uniqueness g-measure}
    \sum_{n=0}^\infty \text{var}_n(\bar{g}_p) <\infty~~\text{and}~~\inf_{\omega'\in \Omega'}\bar{g}_p(\omega'_{(-\infty,0]})>0.
\end{equation}

\begin{proof}[Proof of Proposition \ref{thm: Unique g-measure}]
   We need to analyze the function $\bar{g}_p(\cdot)=\rho_{p,\{0\}}(0_{\{0\}}|\cdot)$ in more detail via the special form given as $g_p(n)$ (see Theorem \ref{thm: g-function for mu'_p}). From the definition of $g_p(n)$, given in \eqref{eq: g-function for the image measure}, and the fact that $|a|<1$, it is easy to see that $g_p(n)$ converges to $\lambda_{\text{PF}}$ as $n\rightarrow \infty$. Consequently, the function $g_p$ is uniformly non-null because $\lambda_{\text{PF}}>0$ and $g_p(n)>0$ for all $n\in \N$.
    Furthermore, one can easily see that $g_p(2n) \leq g_p(2n+2)$ and $g_p(2n+1) \geq g_p(2n+3)$ for all $n \in \N$. First of all, we consider the sequence of even numbers and obtain
     \begin{align*}
        g_p(2n)=\frac{-|a|^{2n}\frac{\lambda_{PF}}{1-\lambda_{r}}+ \frac{\lambda_{PF}}{1-\lambda_{PF}}}{|a|^{2n-1}\frac{1}{1-\lambda_{r}}+\frac{1}{1-\lambda_{PF}}}\leq \frac{-|a|^{2n+2}\frac{\lambda_{PF}}{1-\lambda_{r}}+ \frac{\lambda_{PF}}{1-\lambda_{PF}}}{|a|^{2n+1}\frac{1}{1-\lambda_{r}}+\frac{1}{1-\lambda_{PF}}}=g_p(2n+2)
    \end{align*}
    for all $n \in \N$. Once again, we used the fact that $|a|<1$. Analogously, one can show that $g_p(2n+1)\geq g_p(2n+3)$
    for all $n\in \N$. Combining both results implies $\text{var}_n(g_p)=|g_p(n)-g_p(n+1)|$ for all $n\in \N_{\geq 2}$. From that, we compute the explicit expression for a given $n\in \N_{\geq 2}$ as
    \begin{align*}
        \text{var}_n(g_p)&=|a|^{n-1}\bigg|\frac{\frac{\lambda_{\text{PF}}}{(1-\lambda_r)(1-\lambda_{\text{PF}})}+2|a|\frac{\lambda_{\text{PF}}}{(1-\lambda_r)(1-\lambda_{\text{PF}})}+|a|^{2}\frac{\lambda_{\text{PF}}}{(1-\lambda_r)(1-\lambda_{\text{PF}})}}{-\frac{|a|^{2n-1}}{(1-\lambda_{r})^2}+ (-1)^n\frac{|a|^{n-1}(1-|a|)}{(1-\lambda_r)(1-\lambda_{\text{PF}})}+\frac{1}{(1-\lambda_{\text{PF}})^2}}\bigg|.
    \end{align*}
    Let us take a look at the quotient of adjacent variations at $n$ and $n+1$. We obtain
    \begin{align*}
        \frac{\text{var}_{n+1}(g_p)}{\text{var}_{n}(g_p)}=|a|\frac{-\frac{|a|^{2n-1}}{(1-\lambda_{r})^2}+(-1)^n\frac{|a|^{n-1}(1-|a|)}{(1-\lambda_r)(1-\lambda_{\text{PF}})}+\frac{1}{(1-\lambda_{\text{PF}})^2}}{-\frac{|a|^{2n+1}}{(1-\lambda_{r})^2}+(-1)^{n+1}\frac{|a|^{n}(1-|a|)}{(1-\lambda_r)(1-\lambda_{\text{PF}})}+\frac{1}{(1-\lambda_{\text{PF}})^2}}.
    \end{align*}
    Using $|a|<1$, leads to $\lim_{n \rightarrow \infty} \frac{\text{var}_{n+1}(g)}{\text{var}_{n}(g)}=|a|$ and by comparison with the geometric series this leads to the absolute convergence of $
        \sum_{n=0}^\infty \text{var}_{n}(g)$.
\end{proof}

\subsection{Theorem \ref{thm: Relation TBF and GHoC}: Relation between GHoC and TBF}

We show uniqueness of the stationary distribution. 
To see this almost without computation, let us apply Foster's Theorem, see \cite{Br10} Chapter 7. If we can find a function 
$h:\N_0\cup\{\infty\} \rightarrow [0,\infty)$ such that 
$\E^{GHoC}_p (h(X_{n+1}) | X_{n}=j)\leq h(j)-\epsilon$ outside of a finite set 
of $j$'s, then the Markov chain is positive recurrent. Therefore, as 
our Markov chain is also irreducible, existence and uniqueness of the stationary distribution follows. 
This can be achieved with the choice $h(0)=h(\infty):=0$ and 
$h(j):=A^j$ where $A:=(1+1/\lambda_{PF})/2$. To verify this, write out the condition and observe 
that $\lim_{j\rightarrow\infty} g_p(j)=\lambda_{PF}<1$. 

\subsection{Remark \ref{rk: GHoC implies unique g-measure}: GHoC implies uniqueness of g-measure}

\begin{proof}[Proof of Remark \ref{rk: GHoC implies unique g-measure}]
    Let $p\in (0,1)$ and assume that $\nu_p$ is a $g$-measure for the LIS $\rho_p$, which is uniquely determined by the $g$-function in \eqref{eq: g-function for the image measure}. Consequently, we have 
    \begin{equation*}
        \nu_p(\omega'_{\{i\}}|\omega'_{(-\infty,i-1]})=\rho_{p,\{i\}}(\omega'_{\{i\}}|\omega'_{(-\infty,i-1]})
    \end{equation*}
    for all $i\in \Z$ and $\nu_p$-almost all $\omega'\in \Omega'$. First, note that 
    \begin{equation*}
        \nu_p(\omega'_{[0,n]})=\int\nu_p(d\omega'_{-\N})\prod^n_{i=0}\nu_p(\omega'_{\{i\}}|\omega'_{(-\infty,i-1]}).
    \end{equation*}
    Then, recalling the definition of the distance $n_i$ to the next pair of occupied sites in the past of the site $i$, given in \eqref{def: D2D1 definition}, we can write 
    \begin{equation*}
        \rho_{p,\{i\}}(\omega'_{\{i\}}|\omega'_{(-\infty,i-1]})=g_p(n_i(\omega'))\mathds{1}_{\omega'_i=0}+\big(1-g_p(n_i(\omega'))\big)\mathds{1}_{\omega'_i=1}
    \end{equation*}
    for all $i\in \Z$ and $\omega'\in \Omega'$. The function $n_i(\omega')$ depends only on the values of $n_{i-1}(\omega')$ and $\omega'_{i-1}$. Consequently, the display above reads
    \begin{equation*}
        \nu_p(\omega'_{[0,n]})=\sum^\infty_{k=0}\nu_p(n_0=k)\mathds{1}_{n_0(\omega')=k}\prod^{n}_{i=0}\Pi_p(n_i(\omega'),n_{i+1}(\omega')).
    \end{equation*}
     which implies that $\nu_p$ is determined by $(\nu_p(n_0=k))_{k\in \N_0\cup \{\infty\}}$ and the entries of the transfer matrix $\Pi_p$. By the translation invariance of the measure $\nu_p$, the distribution of $n_0$ under $\nu_p$ satisfies the following equality 
     \begin{equation*}
        \nu_p(n_0=k)=\nu_p(n_1=k)=\sum_{l=0}^\infty\nu_p(n_0=l)\Pi_p(l,k)
    \end{equation*}
     for all $l,k\in \N_0\cup \{\infty\}$. Therefore, we have $\nu_p(n_0=k)=\pi_p(k)$ for all $k\in \N_0\cup \{\infty\}$ where $\pi_p$ is the unique stationary distribution of the GHoC Markov chain, see Theorem \ref{thm: Relation TBF and GHoC}.
\end{proof}

\subsection*{Appendix A: Properties of the transfer operator $Q$}\label{Sec: Appendix A}\addcontentsline{toc}{section}{Appendix A}

The following lemma serves to derive the kernels of the the two-sided specification $\gamma_p'$ for the TBF, see the proof of Theorem \ref{thm: Existence image specification}.

\begin{lemma}\label{lem: Limit Matrix Q}
    Let $Q$ be the transfer operator defined in \eqref{eq: Transfer matrix Q}. Then, 
    \begin{equation}\label{eq: One-sided limit Q^n}
        \lim_{n\rightarrow \infty}\frac{Q^{n+i}(\omega_n,x)}{Q^{n}(\omega_n,y)}=\lim_{n\rightarrow \infty}\frac{Q^{n+i}(x,\eta_n)}{Q^{n}(y,\eta_n)}=\lambda_{\text{PF}}^{i}\frac{v_{\text{PF}}(x)}{v_{\text{PF}}(y)}
    \end{equation}
    and
    \begin{equation}\label{eq: Two-sided limit Q^n}
        \lim_{n,m \rightarrow \infty}\frac{Q^{n+m+i}(\omega_{n},\eta_{m})}{Q^n(\omega_n,x)Q^m(y,\eta_m)}=\lambda^i_{\text{PF}}\frac{C(p)}{v_{\text{PF}}(x)v_{\text{PF}}(y)}
    \end{equation}
    for all $x,y \in\{0,1\}$, $i\in \Z$ and independently of the sequences $(\omega_n)_{n\in \N}$ and $(\eta_n)_{n\in \N}$ in $\Omega^\N$. Here, $\lambda_{\text{PF}}$ denotes the Perron-Frobenius eigenvalue of the transfer operator $Q$, see \eqref{eq: Eigenvalues of Q}, $v_{\text{PF}}$ the corresponding eigenvector, cf. \eqref{eq: Right eigenvectors Q}, and $C(p):=\sqrt{(1-p)(3p+1)}/|\lambda_{\text{r}}|$, where $\lambda_r$ is the real eigenvalue of $Q$, see \eqref{eq: Eigenvalues of Q}.
\end{lemma}

\begin{proof}
    Let us start to compute the matrix $Q^N$ for a fixed $N \in \N$. To this end, note that 
    \begin{equation}\label{eq: Right eigenvectors Q}
    v_{\text{r}}=\begin{pmatrix}
        \frac{\lambda_r}{\sqrt{p(1-p)}}\\
        1
    \end{pmatrix}~~\text{and}~~~~~
    v_{\text{PF}}=\begin{pmatrix}
        \frac{\lambda_{\text{PF}}}{\sqrt{p(1-p)}}\\
        1
    \end{pmatrix}
\end{equation}
are right eigenvectors corresponding to the eigenvalues $\lambda_r$ and $\lambda_{\text{PF}}$, respectively. Define the matrix $B=(B_{xy})_{x,y\in \{0,1\}}=(v_{\text{r}},v_{\text{PF}})$ whose columns consist of these eigenvectors of $Q$. Its inverse can be computed as 
\begin{equation*}
    B^{-1}=(B^{-1}_{xy})_{x,y\in \{0,1\}}=\frac{1}{\sqrt{(1-p)(3p+1)}}\begin{pmatrix}
        -\sqrt{p(1-p)}& \lambda_{\text{PF}}\\
        \sqrt{p(1-p)}& -\lambda_r
    \end{pmatrix}.
\end{equation*}
Hence, we obtain the relation $D=B^{-1}QB$ where $D$ denotes the diagonal form of $Q$ with the eigenvalues on the diagonal. Consequently,
    \begin{align}\label{eq: N-th power Q generalized}
        Q^N=BD^NB^{-1}=\lambda_{\text{PF}}^N\begin{pmatrix}
        a^N B_{00} B^{-1}_{00}+B_{01}B^{-1}_{10}& a^NB_{00}B^{-1}_{01}+B_{01}B_{11}^{-1}\\
        a^N B_{10}B^{-1}_{00}+B_{11}B^{-1}_{10}& a^NB_{10}B_{01}^{-1}+B_{11}B^{-1}_{11}
    \end{pmatrix}
    \end{align}
    where $a=a(p)$ denotes the eigenvalue ratio, see \eqref{eq: Fraction eigenvalues}. Thus, each entry can be written as $BD^NB^{-1}(x,y)=\lambda_{\text{PF}}^N\big(a^N B_{x0}B^{-1}_{0y}+B_{x1}B^{-1}_{1y}\big)$ for all $x,y\in \{0,1\}$.
    
    Let $(\omega_n)_{n \in \N}$ be an arbitrary sequence in $\{0,1\}^\N$ and let $x,y\in \{0,1\}$. Then, 
    \begin{align*}
        &\bigg|\frac{Q^{n+i}(\omega_n,x)}{Q^{n}(\omega_n,y)}-\lambda_{\text{PF}}^i\frac{B^{-1}_{1x}}{B^{-1}_{1y}}\bigg|=\lambda_{\text{PF}}^ia^n\bigg|\frac{a^{i} B_{\omega_n 0}B^{-1}_{0x}B^{-1}_{1y}-B^{-1}_{1x}B_{\omega_n 0}B^{-1}_{0y}}{a^n B_{\omega_n 0}B^{-1}_{0y}B^{-1}_{1y}+B_{\omega_n 1}\big(B^{-1}_{1y}\big)^2}\bigg| \xrightarrow[]{n \rightarrow \infty}0
    \end{align*}
    independently of the sequence $(\omega_n)_{n\in \N}$, since $\lim_{n\rightarrow \infty}a^n=0$ and $B_{\omega_n1}>0$ for all $\omega_n\in \{0,1\}$. Note that $ B^{-1}_{1x}=C(p)^{-1}v_{\text{PF}}(x)$ for all $x\in \{0,1\}$,
 and therefore the first part of \eqref{eq: One-sided limit Q^n} follows. Analogously, the second part of \eqref{eq: One-sided limit Q^n} can be verified. Indeed, for $x,y\in \{0,1\}$ and any sequence $(\eta_n)_{n\in \N}$ in $\{0,1\}^\N$, we have
 \begin{align*}
        &\bigg|\frac{Q^{n+i}(x,\eta_n)}{Q^{n}(y,\eta_n)}-\lambda_{\text{PF}}^i\frac{B_{x 1}}{B_{y 1}}\bigg|=\lambda_{\text{PF}}^ia^n \bigg|\frac{a^{i} B_{x 0}B^{-1}_{0\eta_n}B_{y 1}-B_{x1}B_{y 0}B^{-1}_{0\eta_n}}{a^n B_{y 0}B^{-1}_{0\eta_n}B_{y 1}+B^{-1}_{1\eta_n}\big(B_{y 1}\big)^2}\bigg|\xrightarrow[]{n \rightarrow \infty}0.
    \end{align*}
    Since $B_{x1}=v_{\text{PF}}(x)$ for all $x\in \{0,1\}$, the desired statement follows. 

To establish \eqref{eq: Two-sided limit Q^n}, let $x,y\in \{0,1\}$ and $(\omega_n)_{n\in \N}$, $(\eta_m)_{m \in \N}$ be arbitrary sequences in $\{0,1\}^\N$. Then we obtain
    \begin{align*}
        &\bigg|\frac{Q^{n+m+i}(\omega_n,\eta_m)}{Q^n(\omega_n,x)Q^m(y,\eta_m)}-\lambda_{\text{PF}}^i\frac{1}{B^{-1}_{1x}B_{y1}}\bigg|\\
        &=\lambda_{\text{PF}}^i\big|\big(a^{n+m+i} B_{\omega_n 0}B^{-1}_{0\eta_m }B^{-1}_{1x}B_{y1}-a^{n+m} B_{\omega_n 0} B^{-1}_{0x}B_{y0}B^{-1}_{0\eta_m}-a^n B_{\omega_n0} B^{-1}_{0x}B_{y1}B^{-1}_{1\eta_m}\\
        &-a^mB_{\omega_n 1}B^{-1}_{1x}B_{y0}B^{-1}_{0\eta_m}\big)\big/\big(a^{n+m} B_{\omega_n 0} B^{-1}_{0x}B_{y0}B^{-1}_{0\eta_m}B^{-1}_{1x}B_{y1}+a^n B_{\omega_n0} B^{-1}_{0x}B_{y1}B^{-1}_{1\eta_m}B^{-1}_{1x}B_{y1}\\
        &+a^mB_{\omega_n 1}(B^{-1}_{1x})^2B_{y0}B^{-1}_{0\eta_m}B_{y1}+B_{\omega_n 1}(B^{-1}_{1x}B_{y1})^2B^{-1}_{1\eta_m}\big)\big|.
    \end{align*}
    Note that $B_{\omega_n 1}(B^{-1}_{1x}B_{y1})^2B^{-1}_{1\eta_m}>0$ for all $\omega_n,\eta_m\in \{0,1\}$ and therefore, the denominator is positive for $n,m\in \N$ large enough. Since the difference in the numerator converges to zero as $n,m\rightarrow \infty$, the claim in \eqref{eq: Two-sided limit Q^n} follows after substituting the definition of $B_{1x}^{-1}$ and $B_{y1}$.
\end{proof}

\subsection*{Appendix B: Finite-energy criterion on subspace $\Omega'$}\label{Sec: Appendix B}\addcontentsline{toc}{section}{Appendix B}

 We formulate our specification only on the smaller measurable space $(\Omega',\mathscr{F}')$ containing only the non-isolated configurations. (Further definitions for boundary conditions allowing for configurations which do not respect the non-isolation constraint seem artificial and we would like to avoid them.) One may wonder whether the "finite-energy" criterion which is given in \cite{Ge11} as Proposition 8.38 in the context of specifications in product spaces remains true.

In inspection of the proof given in \cite{Ge11} on pages 164-165 however reveals that we can fully 
restrict ourselves to the constructions of the \textit{trace $\sigma$-algebra} of the product $\sigma$-algebra 
$\mathscr{F}$ to $\Omega'$, which is by definition $\mathscr{F}':=\mathscr{F}|_{\Omega'}=\mathscr{F}\cap \Omega'$. Note the handy measure-theoretic property, see Corollary 1.83 in \cite{Kl14}, that 
\begin{equation*}
    \sigma( \mathcal{C}_{\Z})|_{\Omega'}=\sigma(\mathcal{C}_{\Z}\cap \Omega')
\end{equation*}
where $\mathcal{C}_{\Z}:=\{\sigma_\Lambda^{-1}(A):~\Lambda\Subset \Z~\text{and}~A\in \mathcal{P}(\{-1,1\})^\Lambda\}$ denotes the algebra of cylinders which forms a generator of $\mathscr{F}$.

In words, this means that the smallest sigma algebra containing a traced generator (here: the cylinder sets), 
is equal to the traced sigma algebra of the generator. In our case, forming the trace 
means to restrict to the space of non-isolates. 

A second important point is to realize that the implication from tail-triviality for Gibbs-measures supported on $\Omega'$, 
to extremality (in the set of Gibbs measures which are supported on $\Omega'$) remains true as it does on product spaces. 
This is proved along the same lines as in the corresponding theorem for product spaces, see for instance Theorem 7.7 a) of \cite{Ge11} where the key ingredient is Proposition 7.3 in \cite{Ge11} applied to finite-volume kernels 
(where the last proposition notably does not assume product spaces).

\subsection*{Appendix C: Parity dependence on the boundary condition as $p\uparrow 1$}\label{Sec: Appendix C}\addcontentsline{toc}{section}{Appendix C}

In this part of the appendix, we present the proof of Lemma \ref{lem: Limit g-function n odd} which relies on the application of l’Hôpital’s rule.

\begin{proof}[Proof of Lemma \ref{lem: Limit g-function n odd}]
    We can rewrite \eqref{eq: g-function for the image measure} as 
    \begin{equation*}
        g_p(n)=\frac{\lambda_{\text{PF}}(1-\lambda_{\text{r}})-a^n\lambda_{\text{PF}}(1-\lambda_{\text{PF}})}{(1-\lambda_{\text{r}})-a^{n-1}(1-\lambda_{\text{PF}})}
    \end{equation*}
    where the numerator converges to zero as $p\uparrow 1$. The denominator, on the other hand, converges to zero if $n$ is odd and to $2$ otherwise. Consequently, we focus on the case $n\in \N_{\geq 3}$ with odd $n$, for which l’Hôpital’s can be applied to evaluate the limit. More precisely, we have
    \begin{align}\label{eq: Derivative g-function 1}
        \lim_{p\uparrow 1}g_p(n)
        =\lim_{p\uparrow 1}\frac{\diff{}{p}\big(\lambda_{\text{PF}}(1-\lambda_{\text{r}})\big)-na^{n-1}\lambda_{\text{PF}}(1-\lambda_{\text{PF}})\diff{}{p}a-a^n\diff{}{p}\big(\lambda_{\text{PF}}(1-\lambda_{\text{PF}})\big)}{-\diff{}{p}\lambda_{\text{r}}-(n-1)a^{n-2}(1-\lambda_{\text{PF}})\diff{}{p}a+a^{n-1}\diff{}{p}\lambda_{\text{PF}}}.
    \end{align}
    Let us compute all parts of this fraction with their derivatives. First of all,
    \begin{align}\label{eq: Derivative g-function 2}
        \lambda_{\text{PF}}(1-\lambda_{\text{r}})
        =\frac{1}{2}\big(1+p-2p^2+\sqrt{(1-p)(3p+1)}\big)
    \end{align}
    whose $p$-derivative is given by
        $\frac{(1-4p)\sqrt{(1-p)(3p+1)}+1-3p}{2\sqrt{(1-p)(3p+1)}}$.
    Therefore, one can compute
    \begin{align*}\label{eq: Derivative g-function 3}
       \lim_{p\uparrow 1} \sqrt{(1-p)(3p+1)} \diff{}{p}\big(\lambda_{\text{PF}}(1-\lambda_{\text{r}})\big)=-1.
    \end{align*}
    Secondly,
    \begin{align*}
        \diff{}{p}\big(\lambda_{\text{PF}}(1-\lambda_{\text{PF}})\big)=\frac{(1-p)(3p+1)-(1-2p)\sqrt{(1-p)(3p+1)}+p(1-3p)}{2\sqrt{(1-p)(3p+1)}}.
    \end{align*}
    Consequently, we have
    \begin{equation}\label{eq: Derivative g-function 6}
        \lim_{p\uparrow 1}\sqrt{(1-p)(3p+1)}\diff{}{p}\big(\lambda_{\text{PF}}(1-\lambda_{\text{PF}})\big)=-1.
    \end{equation}
    The derivatives of the eigenvalues take the form
    \begin{align*}\label{eq: Derivative g-function 7}
        \diff{}{p}\lambda_{\text{r}}
        =-\frac{\sqrt{(1-p)(3p+1)}+1-3p}{2\sqrt{(1-p)(3p+1)}}~~\text{and}~~\diff{}{p}\lambda_{\text{PF}}
        =\frac{(1-3p)-\sqrt{(1-p)(3p+1)}}{2\sqrt{(1-p)(3p+1)}}
    \end{align*}
    which leads to
    \begin{equation}\label{eq: Derivative g-function 7}
        \lim_{p\uparrow 1}\sqrt{(1-p)(3p+1)}\diff{}{p}\lambda_{\text{r}}=1~~\text{and}~~ \lim_{p\uparrow 1}\sqrt{(1-p)(3p+1)}\diff{}{p}\lambda_{\text{PF}}=-1.
    \end{equation}
     Finally, we have
     \begin{align}\label{eq: Derivative g-function 8}
       \diff{}{p}(a)
       =-\frac{2}{\big(1+p+\sqrt{(1-p)(3p+1)}\big)\sqrt{(1-p)(3p+1)}}.
    \end{align}
    Combining \eqref{eq: Derivative g-function 2}-\eqref{eq: Derivative g-function 8} and $\lim_{p\uparrow 1}a=-1$  in \eqref{eq: Derivative g-function 1}, gives us 
    \begin{align*}
        &\lim_{p\uparrow 1}g_p(n)=\lim_{p\uparrow 1}\frac{-1+(-1)^n+2n(-1)^{n-1}\lambda_{\text{PF}}(1-\lambda_{\text{PF}})\big(1+p+\sqrt{(1-p)(3p+1)}\big)^{-1}}{-1-(-1)^{n-1}+2(n-1)(-1)^{n-2}(1-\lambda_{\text{PF}})\big(1+p+\sqrt{(1-p)(3p+1)}\big)^{-1}}.
    \end{align*}
    Since $n$ is odd and $\lim\lambda_{\text{PF}}=0$, the limit equals $\frac{2}{n+1}$, proving the lemma.
 \end{proof}

 Let us discuss the resulting limit of the $g$-function as $p\uparrow 1$ in the following remark. 

\begin{remark}\label{rk: Parity dependence on b.c. one-sided view}
    The limits in Lemma \ref{lem: Limit g-function n odd} can be understood in terms of tight packings of isolated occupied sites in the hidden first-layer configuration $\omega\in T^{-1}(\omega')$. 
    
    We first consider the case where $n\in \N$ is even. If $\omega'_0=0$, observe that there is an unfixed area in $[-n+2,-1]$ in which the spins of the hidden configuration $\omega$ tend to accumulate many isolated occupied sites as $p\rightarrow 1$. If $(\omega_0,\omega_1)=(1,0)$, the configuration $\omega$ typically arranges itself as the alternating configuration $\omega^1$, defined by $(\omega^1)_i=1$ for $i\in \Z$ even and $(\omega^1)_i=0$ otherwise, see picture V in Figure \ref{fig: Explanation plot g-function 1}. If instead $(\omega_0,\omega_1)=(0,1)$, one can insert at most $\frac{n}{2}-1$ isolated occupied sites in this region, compare picture I-IV in Figure \ref{fig: Explanation plot g-function 1}. In total, at most $\frac{n}{2}$ occupied sites can be placed in the interval $[-n-2,1]$. If $\omega'_0=1$, the spins in $\{0,1\}$ are fixed and occupied due to the non-isolation constraint. The remaining unfixed spins in $[-n+2,-2]$ tend to align according to the alternating configuration $\omega^1$, resulting in a tighter packing of occupied sites than in the case $\omega'_0=0$, see picture VI in Figure \ref{fig: Explanation plot g-function 1}. More precisely, one can place up to $\frac{n}{2}+1$ occupied site in the interval $[-n+2,1]$. This imbalance explains why $\lim_{p\rightarrow 1}g_p(n)=0$ when $n$ is even.

\begin{figure}[ht!]
\centering
\scalebox{0.4}{
\begin{tikzpicture}[every label/.append style={scale=1.3}]

    \node[shape=circle,draw=black, minimum size=0.5cm] (-B) at (-4,0) {};
    \node[shape=circle,draw=cyan, fill=cyan, minimum size=0.5cm] (-C) at (-5,0) {};
    \node[shape=circle,draw=black, minimum size=0.5cm] (-D) at (-6,0) {};
    \node[shape=circle,draw=cyan, fill=cyan, minimum size=0.5cm] (-E) at (-7,0) {};
    \node[shape=circle,draw=black, minimum size=0.5cm] (-F) at (-8,0) {};
    \node[shape=circle,draw=cyan, fill=cyan,  minimum size=0.5cm] (-G) at (-9,0) {}; \node[shape=circle,draw=black, minimum size=0.5cm] (-H) at (-10,0) {};
    \node[shape=circle,draw=black, fill=black, minimum size=0.5cm,label=above:{$-8$}] (-I) at (-11,0) {};
    \node[shape=circle,draw=black, fill=black, minimum size=0.5cm] (-J) at (-12,0) {};
    \node[shape=circle,draw=black, minimum size=0.5cm, label=above:{$0$}] (A) at (-3,0) {};
    \node[shape=circle,draw=black, fill=black, minimum size=0.5cm] (B) at (-2,0) {};

    \node[minimum size=0.5, label=:{\text{I}}] () at (-13.5,-0.5) {};

    \draw[-] (-J)--(-12.7,0){};
    \draw [-] (-I) -- (-J);
    \draw [-] (-H) -- (-I);
    \draw [-] (-G) -- (-H);
    \draw [-] (-F) -- (-G);
    \draw [-] (-F) -- (-E);
    \draw [-] (-E) -- (-D);
    \draw [-] (-D) -- (-C);
    \draw [-] (-C) -- (-B);
    \draw [-] (-B) -- (A);
    \draw [-] (A) -- (B);
    \draw[-](B)--(-1.3,0){};

    \node[shape=circle,draw=cyan,  fill=cyan, minimum size=0.5cm] (-B) at (-4,-2) {};
    \node[shape=circle,draw=black, minimum size=0.5cm] (-C) at (-5,-2) {};
    \node[shape=circle,draw=black, minimum size=0.5cm] (-D) at (-6,-2) {};
    \node[shape=circle,draw=cyan, fill=cyan, minimum size=0.5cm] (-E) at (-7,-2) {};
    \node[shape=circle,draw=black,   minimum size=0.5cm] (-F) at (-8,-2) {};
    \node[shape=circle,draw=cyan,  fill=cyan,  minimum size=0.5cm] (-G) at (-9,-2) {}; \node[shape=circle,draw=black, minimum size=0.5cm] (-H) at (-10,-2) {};
    \node[shape=circle,draw=black, fill=black, minimum size=0.5cm, label=above:{$-8$}] (-I) at (-11,-2) {};
    \node[shape=circle,draw=black, fill=black, minimum size=0.5cm] (-J) at (-12,-2) {};
    \node[shape=circle,draw=black,  minimum size=0.5cm, label=above:{$0$}] (A) at (-3,-2) {};
    \node[shape=circle,draw=black, fill=black, minimum size=0.5cm] (B) at (-2,-2) {};
    \node[minimum size=0.5, label=:{\text{II}}] () at (-13.5,-2.5) {};

    \draw[-] (-J)--(-12.7,-2){};
    \draw [-] (-I) -- (-J);
    \draw [-] (-H) -- (-I);
    \draw [-] (-G) -- (-H);
    \draw [-] (-F) -- (-G);
    \draw [-] (-F) -- (-E);
    \draw [-] (-E) -- (-D);
    \draw [-] (-D) -- (-C);
    \draw [-] (-C) -- (-B);
    \draw [-] (-B) -- (A);
    \draw [-] (A) -- (B);
    \draw[-](B)--(-1.3,-2){};


\node[shape=circle,draw=cyan,  fill=cyan, minimum size=0.5cm] (-B) at (-4,-4) {};
    \node[shape=circle,draw=black, minimum size=0.5cm] (-C) at (-5,-4) {};
    \node[shape=circle,draw=cyan, fill=cyan, minimum size=0.5cm] (-D) at (-6,-4) {};
    \node[shape=circle,draw=black, minimum size=0.5cm] (-E) at (-7,-4) {};
    \node[shape=circle,draw=black,   minimum size=0.5cm] (-F) at (-8,-4) {};
    \node[shape=circle,draw=cyan,  fill=cyan,  minimum size=0.5cm] (-G) at (-9,-4) {}; \node[shape=circle,draw=black, minimum size=0.5cm] (-H) at (-10,-4) {};
    \node[shape=circle,draw=black, fill=black, minimum size=0.5cm, label=above:{$-8$}] (-I) at (-11,-4) {};
    \node[shape=circle,draw=black, fill=black, minimum size=0.5cm] (-J) at (-12,-4) {};
    \node[shape=circle,draw=black,  minimum size=0.5cm, label=above:{$0$}] (A) at (-3,-4) {};
    \node[shape=circle,draw=black, fill=black, minimum size=0.5cm] (B) at (-2,-4) {};
    \node[minimum size=0.5, label=:{\text{III}}] () at (-13.5,-4.5) {};

    \draw[-] (-J)--(-12.7,-4){};
    \draw [-] (-I) -- (-J);
    \draw [-] (-H) -- (-I);
    \draw [-] (-G) -- (-H);
    \draw [-] (-F) -- (-G);
    \draw [-] (-F) -- (-E);
    \draw [-] (-E) -- (-D);
    \draw [-] (-D) -- (-C);
    \draw [-] (-C) -- (-B);
    \draw [-] (-B) -- (A);
    \draw [-] (A) -- (B);
    \draw[-](B)--(-1.3,-4){};


\node[shape=circle,draw=black, minimum size=0.5cm] (-B) at (9,-4) {};
    \node[shape=circle,draw=cyan, fill=cyan, minimum size=0.5cm] (-C) at (8,-4) {};
    \node[shape=circle,draw=black, minimum size=0.5cm] (-D) at (7,-4) {};
    \node[shape=circle,draw=cyan, fill=cyan, minimum size=0.5cm] (-E) at (6,-4) {};
    \node[shape=circle,draw=black,   minimum size=0.5cm] (-F) at (5,-4) {};
    \node[shape=circle,draw=cyan, fill=cyan, minimum size=0.5cm] (-G) at (4,-4) {}; \node[shape=circle,draw=black,  minimum size=0.5cm] (-H) at (3,-4) {};
    \node[shape=circle,draw=black, fill=black, minimum size=0.5cm, label=above:{$-8$}] (-I) at (2,-4) {};
    \node[shape=circle,draw=black, fill=black, minimum size=0.5cm] (-J) at (1,-4) {};
    \node[shape=circle,draw=black, fill=black, minimum size=0.5cm, label=above:{$0$}] (A) at (10,-4) {};
    \node[shape=circle,draw=black, fill=black, minimum size=0.5cm] (B) at (11,-4) {};
    \node[minimum size=0.5, label=:{\text{VI}}] () at (12.5,-4.5) {};

    \draw[-] (-J)--(0.3,-4){};
    \draw [-] (-I) -- (-J);
    \draw [-] (-H) -- (-I);
    \draw [-] (-G) -- (-H);
    \draw [-] (-F) -- (-G);
    \draw [-] (-F) -- (-E);
    \draw [-] (-E) -- (-D);
    \draw [-] (-D) -- (-C);
    \draw [-] (-C) -- (-B);
    \draw [-] (-B) -- (A);
    \draw [-] (A) -- (B);
    \draw[-](B)--(11.7,-4){};

    \node[shape=circle,draw=black, minimum size=0.5cm] (-B) at (8,0) {};
    \node[shape=circle,draw=cyan, fill=cyan, minimum size=0.5cm] (-C) at (7,0) {};
    \node[shape=circle,draw=black, minimum size=0.5cm] (-D) at (6,0) {};
    \node[shape=circle,draw=cyan, fill=cyan, minimum size=0.5cm] (-E) at (5,0) {};
    \node[shape=circle,draw=black,   minimum size=0.5cm] (-F) at (4,0) {};
    \node[shape=circle,draw=black,  minimum size=0.5cm] (-G) at (3,0) {}; \node[shape=circle,draw=black, fill=black, minimum size=0.5cm,label=above:{$-8$}] (-H) at (2,0) {};
    \node[shape=circle,draw=black, fill=black, minimum size=0.5cm] (-I) at (1,0) {};
    \node[shape=circle,draw=cyan, fill=cyan, minimum size=0.5cm] (A) at (9,0) {};
    \node[shape=circle,draw=black,  minimum size=0.5cm, label=above:{$0$}] (B) at (10,0) {};
    \node[shape=circle,draw=black, fill=black, minimum size=0.5cm] (C) at (11,0) {};

    \node[minimum size=0.5, label=:{\text{IV}}] () at (12.5,-0.5) {};

    \draw[-] (-I)--(0.3,0){};
    \draw [-] (-H) -- (-I);
    \draw [-] (-G) -- (-H);
    \draw [-] (-F) -- (-G);
    \draw [-] (-F) -- (-E);
    \draw [-] (-E) -- (-D);
    \draw [-] (-D) -- (-C);
    \draw [-] (-C) -- (-B);
    \draw [-] (-B) -- (A);
    \draw [-] (A) -- (B);
    \draw [-] (B) -- (C);
    \draw[-](C)--(11.7,0){};

    \node[shape=circle,draw=cyan, fill=cyan, minimum size=0.5cm] (-B) at (8,-2) {};
    \node[shape=circle,draw=black, minimum size=0.5cm] (-C) at (7,-2) {};
    \node[shape=circle,draw=cyan, fill=cyan, minimum size=0.5cm] (-D) at (6,-2) {};
    \node[shape=circle,draw=black, minimum size=0.5cm] (-E) at (5,-2) {};
    \node[shape=circle,draw=cyan, fill=cyan,  minimum size=0.5cm] (-F) at (4,-2) {};
    \node[shape=circle,draw=black,  minimum size=0.5cm] (-G) at (3,-2) {}; \node[shape=circle,draw=black, fill=black, minimum size=0.5cm,label=above:{$-8$}] (-H) at (2,-2) {};
    \node[shape=circle,draw=black, fill=black, minimum size=0.5cm] (-I) at (1,-2) {};
    \node[shape=circle,draw=black, minimum size=0.5cm] (A) at (9,-2) {};
    \node[shape=circle,draw=cyan, fill=cyan, minimum size=0.5cm, label=above:{$0$}] (B) at (10,-2) {};
    \node[shape=circle,draw=black, minimum size=0.5cm] (C) at (11,-2) {};

    \node[minimum size=0.5, label=:{\text{V}}] () at (12.5,-2.5) {};

    \draw[-] (-I)--(0.3,-2){};
    \draw [-] (-H) -- (-I);
    \draw [-] (-G) -- (-H);
    \draw [-] (-F) -- (-G);
    \draw [-] (-F) -- (-E);
    \draw [-] (-E) -- (-D);
    \draw [-] (-D) -- (-C);
    \draw [-] (-C) -- (-B);
    \draw [-] (-B) -- (A);
    \draw [-] (A) -- (B);
    \draw [-] (B) -- (C);
     \draw[-](C)--(11.7,-2){};

\end{tikzpicture}
}
\small \caption{Shown are six configurations on the interval $[-9,1]$. All possible tight packings of isolated occupied sites, highlighted in blue, under occupied boundary conditions are depicted: on the interval $[-7,0]$ in pictures I-IV, on $[-7,1]$ in picture V, and on $[-7,-1]$ in picture VI.}
\label{fig: Explanation plot g-function 1}
\end{figure}
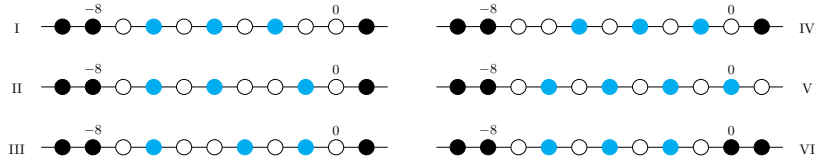

    Finally, we discuss the case where $n\in \N_{\geq 3}$ is odd. In this situation, the probability $g_p(n)$ converges to $\frac{2}{n+1}$ as $p\rightarrow 1$, see Lemma \ref{lem: Limit g-function n odd}. This behavior can be understood as follows: If $\omega_0'=0$, the alternating configuration $\omega^0$, defined by $(\omega^0)_i=(\omega^1)_i-1$ for all $i\in \Z$, can be placed perfectly in the interval $[-n,1]$ within the hidden constrained first-layer, see picture V in Figure \ref{fig: Explanation plot g-function 2}. As $p\rightarrow 1$, this configuration is the most probable arrangement for the hidden-layer. If $\omega_0'=1$, then by the non-isolation constraint, the spin at site $1$ must be occupied and the spin at the site $-1$ must be unoccupied. In this case, one cannot insert the alternating configuration $\omega^0$ or $\omega^1$ perfectly into the interval $[-n,-1]$. However, one can fill this region with $\frac{n-3}{2}$ isolated occupied sites together with exactly one pair of adjacent zeroes, ensuring that all occupied spins in $[-n+1,-1]$ remain isolated. For a given distance $n$, there are $\frac{n-1}{2}$ distinct positions for this pair of zeroes and hence $\frac{n-1}{2}$ admissible arrangements, compare pictures I-IV in Figure \ref{fig: Explanation plot g-function 2}. In total, we get the same number of occupied spins in the area $[-n,1]$ as in the case where the origin is unoccupied. Therefore, for an odd distance $n$, we obtain exactly one maximally packed configuration when $\omega'_0=0$ and $\frac{n-1}{2}$ maximally packed configurations when $\omega'_0=1$. This combinatorial imbalance leads to the limiting value $\frac{2}{n+1}$ for $g_p(n)$ as $p\rightarrow 1$.
\end{remark}

\begin{figure}[ht!]
\centering
\scalebox{0.4}{
\begin{tikzpicture}[every label/.append style={scale=1.3}]

    \node[shape=circle,draw=black, minimum size=0.5cm] (-B) at (-5,0) {};
    \node[shape=circle,draw=black, minimum size=0.5cm] (-C) at (-6,0) {};
    \node[shape=circle,draw=cyan, fill=cyan, minimum size=0.5cm] (-D) at (-7,0) {};
    \node[shape=circle,draw=black, minimum size=0.5cm] (-E) at (-8,0) {};
    \node[shape=circle,draw=cyan, fill=cyan,  minimum size=0.5cm] (-F) at (-9,0) {};
    \node[shape=circle,draw=black,  minimum size=0.5cm] (-G) at (-10,0) {}; \node[shape=circle,draw=cyan, fill=cyan, minimum size=0.5cm] (-H) at (-11,0) {};
    \node[shape=circle,draw=black, minimum size=0.5cm] (-I) at (-12,0) {};
    \node[shape=circle,draw=black, fill=black, minimum size=0.5cm, label=above:{$-9$}] (-J) at (-13,0) {};
    \node[shape=circle,draw=black, fill=black, minimum size=0.5cm] (-K) at (-14,0) {};
    \node[shape=circle,draw=black, fill=black, minimum size=0.5cm, label=above:{$0$}] (A) at (-4,0) {};
    \node[shape=circle,draw=black, fill=black, minimum size=0.5cm] (B) at (-3,0) {};

    \node[minimum size=0.5, label=:{\text{I}}] () at (-15.5,-0.5) {};

    \draw[-] (-K)--(-14.7,0){};
    \draw [-] (-J) -- (-K);
    \draw [-] (-I) -- (-J);
    \draw [-] (-H) -- (-I);
    \draw [-] (-G) -- (-H);
    \draw [-] (-F) -- (-G);
    \draw [-] (-F) -- (-E);
    \draw [-] (-E) -- (-D);
    \draw [-] (-D) -- (-C);
    \draw [-] (-C) -- (-B);
    \draw [-] (-B) -- (A);
    \draw [-] (A) -- (B);
    \draw[-](B)--(-2.3,0){};

    \node[shape=circle,draw=black, minimum size=0.5cm] (-B) at (-5,-2) {};
    \node[shape=circle,draw=cyan, fill=cyan, minimum size=0.5cm] (-C) at (-6,-2) {};
    \node[shape=circle,draw=black, minimum size=0.5cm] (-D) at (-7,-2) {};
    \node[shape=circle,draw=black, minimum size=0.5cm] (-E) at (-8,-2) {};
    \node[shape=circle,draw=cyan, fill=cyan,  minimum size=0.5cm] (-F) at (-9,-2) {};
    \node[shape=circle,draw=black,  minimum size=0.5cm] (-G) at (-10,-2) {}; \node[shape=circle,draw=cyan, fill=cyan, minimum size=0.5cm] (-H) at (-11,-2) {};
    \node[shape=circle,draw=black, minimum size=0.5cm] (-I) at (-12,-2) {};
    \node[shape=circle,draw=black, fill=black, minimum size=0.5cm, label=above:{$-9$}] (-J) at (-13,-2) {};
    \node[shape=circle,draw=black, fill=black, minimum size=0.5cm] (-K) at (-14,-2) {};
    \node[shape=circle,draw=black, fill=black, minimum size=0.5cm, label=above:{$0$}] (A) at (-4,-2) {};
    \node[shape=circle,draw=black, fill=black, minimum size=0.5cm] (B) at (-3,-2) {};
    \node[minimum size=0.5, label=:{\text{II}}] () at (-15.5,-2.5) {};

    \draw[-] (-K)--(-14.7,-2){};
    \draw [-] (-J) -- (-K);
    \draw [-] (-I) -- (-J);
    \draw [-] (-H) -- (-I);
    \draw [-] (-G) -- (-H);
    \draw [-] (-F) -- (-G);
    \draw [-] (-F) -- (-E);
    \draw [-] (-E) -- (-D);
    \draw [-] (-D) -- (-C);
    \draw [-] (-C) -- (-B);
    \draw [-] (-B) -- (A);
    \draw [-] (A) -- (B);
    \draw[-](B)--(-2.3,-2){};


\node[shape=circle,draw=cyan, fill=cyan, minimum size=0.5cm] (-B) at (2,-4) {};
    \node[shape=circle,draw=black, minimum size=0.5cm] (-C) at (1,-4) {};
    \node[shape=circle,draw=cyan, fill=cyan, minimum size=0.5cm] (-D) at (0,-4) {};
    \node[shape=circle,draw=black, minimum size=0.5cm] (-E) at (-1,-4) {};
    \node[shape=circle,draw=cyan, fill=cyan,  minimum size=0.5cm] (-F) at (-2,-4) {};
    \node[shape=circle,draw=black,  minimum size=0.5cm] (-G) at (-3,-4) {}; \node[shape=circle,draw=cyan, fill=cyan, minimum size=0.5cm] (-H) at (-4,-4) {};
    \node[shape=circle,draw=black, minimum size=0.5cm] (-I) at (-5,-4) {};
    \node[shape=circle,draw=black, fill=black, minimum size=0.5cm, label=above:{$-9$}] (-J) at (-6,-4) {};
    \node[shape=circle,draw=black, fill=black, minimum size=0.5cm] (-K) at (-7,-4) {};
    \node[shape=circle,draw=black, minimum size=0.5cm, label=above:{$0$}] (A) at (3,-4) {};
    \node[shape=circle,draw=black, fill=black, minimum size=0.5cm] (B) at (4,-4) {};
    \node[minimum size=0.5, label=:{\text{V}}] () at (-8.5,-4.5) {};

    \draw[-] (-K)--(-7.7,-4){};
    \draw [-] (-J) -- (-K);
    \draw [-] (-I) -- (-J);
    \draw [-] (-H) -- (-I);
    \draw [-] (-G) -- (-H);
    \draw [-] (-F) -- (-G);
    \draw [-] (-F) -- (-E);
    \draw [-] (-E) -- (-D);
    \draw [-] (-D) -- (-C);
    \draw [-] (-C) -- (-B);
    \draw [-] (-B) -- (A);
    \draw [-] (A) -- (B);
    \draw[-](B)--(4.7,-4){};

    \node[shape=circle,draw=black, minimum size=0.5cm] (-B) at (7,0) {};
    \node[shape=circle,draw=cyan, fill=cyan, minimum size=0.5cm] (-C) at (6,0) {};
    \node[shape=circle,draw=black, minimum size=0.5cm] (-D) at (5,0) {};
    \node[shape=circle,draw=black, minimum size=0.5cm] (-E) at (4,0) {};
    \node[shape=circle,draw=cyan, fill=cyan,  minimum size=0.5cm] (-F) at (3,0) {};
    \node[shape=circle,draw=black,  minimum size=0.5cm] (-G) at (2,0) {}; \node[shape=circle,draw=black, fill=black, minimum size=0.5cm,label=above:{$-9$}] (-H) at (1,0) {};
    \node[shape=circle,draw=black, fill=black, minimum size=0.5cm] (-I) at (0,0) {};
    \node[shape=circle,draw=cyan, fill=cyan, minimum size=0.5cm] (A) at (8,0) {};
    \node[shape=circle,draw=black, minimum size=0.5cm] (B) at (9,0) {};
    \node[shape=circle,draw=black, fill=black, minimum size=0.5cm, label=above:{$0$}] (C) at (10,0) {};
    \node[shape=circle,draw=black, fill=black, minimum size=0.5cm] (D) at (11,0) {};

    \node[minimum size=0.5, label=:{\text{III}}] () at (12.5,-0.5) {};

    \draw[-] (-I)--(-0.7,0){};
    \draw [-] (-H) -- (-I);
    \draw [-] (-G) -- (-H);
    \draw [-] (-F) -- (-G);
    \draw [-] (-F) -- (-E);
    \draw [-] (-E) -- (-D);
    \draw [-] (-D) -- (-C);
    \draw [-] (-C) -- (-B);
    \draw [-] (-B) -- (A);
    \draw [-] (A) -- (B);
    \draw [-] (B) -- (C);
    \draw [-] (C) -- (D);
    \draw[-](D)--(11.7,0){};

    \node[shape=circle,draw=black, minimum size=0.5cm] (-B) at (7,-2) {};
    \node[shape=circle,draw=cyan, fill=cyan, minimum size=0.5cm] (-C) at (6,-2) {};
    \node[shape=circle,draw=black, minimum size=0.5cm] (-D) at (5,-2) {};
    \node[shape=circle,draw=cyan, fill=cyan, minimum size=0.5cm] (-E) at (4,-2) {};
    \node[shape=circle,draw=black,  minimum size=0.5cm] (-F) at (3,-2) {};
    \node[shape=circle,draw=black,  minimum size=0.5cm] (-G) at (2,-2) {}; \node[shape=circle,draw=black, fill=black, minimum size=0.5cm,label=above:{$-9$}] (-H) at (1,-2) {};
    \node[shape=circle,draw=black, fill=black, minimum size=0.5cm] (-I) at (0,-2) {};
    \node[shape=circle,draw=cyan, fill=cyan, minimum size=0.5cm] (A) at (8,-2) {};
    \node[shape=circle,draw=black, minimum size=0.5cm] (B) at (9,-2) {};
    \node[shape=circle,draw=black, fill=black, minimum size=0.5cm, label=above:{$0$}] (C) at (10,-2) {};
    \node[shape=circle,draw=black, fill=black, minimum size=0.5cm] (D) at (11,-2) {};

    \node[minimum size=0.5, label=:{\text{IV}}] () at (12.5,-2.5) {};

    \draw[-] (-I)--(-0.7,-2){};
    \draw [-] (-H) -- (-I);
    \draw [-] (-G) -- (-H);
    \draw [-] (-F) -- (-G);
    \draw [-] (-F) -- (-E);
    \draw [-] (-E) -- (-D);
    \draw [-] (-D) -- (-C);
    \draw [-] (-C) -- (-B);
    \draw [-] (-B) -- (A);
    \draw [-] (A) -- (B);
    \draw [-] (B) -- (C);
    \draw [-] (C) -- (D);
    \draw[-](D)--(11.7,-2){};

\end{tikzpicture}
}
\small \caption{Shown are six configurations on the interval $[-10,1]$. All possible tight packings of isolated occupied sites, highlighted in blue, under occupied boundary conditions are depicted: on the interval $[-8,-1]$ in pictures I-IV, and on $[-8,0]$ in picture V.}
\label{fig: Explanation plot g-function 2}
\end{figure}
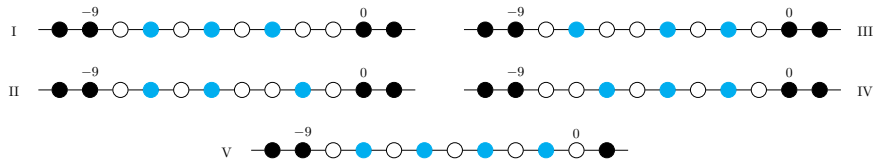

\section*{Declarations}

\noindent\textbf{Data Availability.} 
No datasets were generated or analyzed during this work.
Data sharing is not applicable to this article.
The plots in Figure \ref{fig: Speed of convergence and plot g-function} have been generated by direct computation using Wolfram Mathematica.
\medskip 

\noindent\textbf{Competing interests.} The authors have no competing interests to declare that are relevant to the content of
this article. Furthermore, the authors have no financial interests to disclose.

\printbibliography

@misc{EnFeMaVe24,
      title={On an extension of a theorem by Ruelle to long-range potentials}, 
      author={Aernout C. D. van Enter and Roberto Fernández and Mirmukhsin Makhmudov and Evgeny Verbitskiy},
      year={2024},
      eprint={2404.07326},
      archivePrefix={arXiv},
      primaryClass={math.DS},
      url={https://arxiv.org/abs/2404.07326}, 
}

@article {Ma25,
    AUTHOR = {Makhmudov, Mirmukhsin},
     TITLE = {The {E}igenfunctions of the {T}ransfer {O}perator for the
              {D}yson {M}odel in a {F}ield},
   JOURNAL = {J. Stat. Phys.},
  FJOURNAL = {Journal of Statistical Physics},
    VOLUME = {192},
      YEAR = {2025},
    NUMBER = {7},
     PAGES = {Paper No. 92},
      ISSN = {0022-4715,1572-9613},
   MRCLASS = {37D35 (60B20 82B20 82B26)},
  MRNUMBER = {4927941},
       DOI = {10.1007/s10955-025-03476-z},
       URL = {https://doi.org/10.1007/s10955-025-03476-z},
}

@article {BiEnEnLe18,
    AUTHOR = {Bissacot, Rodrigo and Endo, Eric O. and van Enter, Aernout C.
              D. and Le Ny, Arnaud},
     TITLE = {Entropic repulsion and lack of the {$g$}-measure property for
              {D}yson models},
   JOURNAL = {Comm. Math. Phys.},
  FJOURNAL = {Communications in Mathematical Physics},
    VOLUME = {363},
      YEAR = {2018},
    NUMBER = {3},
     PAGES = {767--788},
      ISSN = {0010-3616,1432-0916},
   MRCLASS = {82B20},
  MRNUMBER = {3858821},
       DOI = {10.1007/s00220-018-3233-6},
       URL = {https://doi.org/10.1007/s00220-018-3233-6},
}

@book {LaPe18,
    AUTHOR = {Last, G\"unter and Penrose, Mathew},
     TITLE = {Lectures on the {P}oisson process},
    SERIES = {Institute of Mathematical Statistics Textbooks},
    VOLUME = {7},
 PUBLISHER = {Cambridge University Press, Cambridge},
      YEAR = {2018},
     PAGES = {xx+293},
      ISBN = {978-1-107-45843-7; 978-1-107-08801-6},
   MRCLASS = {60G55 (60D05 60F05 60G57 60H05)},
  MRNUMBER = {3791470},
MRREVIEWER = {Christoph\ Th\"ale},
}

@book {Ma86,
    AUTHOR = {Mat\'ern, Bertil},
     TITLE = {Spatial variation},
    SERIES = {Lecture Notes in Statistics},
    VOLUME = {36},
   EDITION = {Second},
      NOTE = {With a Swedish summary},
 PUBLISHER = {Springer-Verlag, Berlin},
      YEAR = {1986},
     PAGES = {151},
      ISBN = {3-540-96365-0},
   MRCLASS = {62D05 (62E99)},
  MRNUMBER = {867886},
       DOI = {10.1007/978-1-4615-7892-5},
       URL = {https://doi.org/10.1007/978-1-4615-7892-5},
}

@article {EnSh24,
    AUTHOR = {van Enter, Aernout and Shlosman, Senya},
     TITLE = {The {S}chonmann projection: how {G}ibbsian is it?},
   JOURNAL = {Ann. Inst. Henri Poincar\'e{} Probab. Stat.},
  FJOURNAL = {Annales de l'Institut Henri Poincar\'e{} Probabilit\'es et
              Statistiques},
    VOLUME = {60},
      YEAR = {2024},
    NUMBER = {1},
     PAGES = {2--10},
      ISSN = {0246-0203,1778-7017},
   MRCLASS = {82B20 (60K35 82B41)},
  MRNUMBER = {4718372},
MRREVIEWER = {Malin\ P.\ Forsstr\"om},
       DOI = {10.1214/22-aihp1266},
       URL = {https://doi.org/10.1214/22-aihp1266},
}

@article {BeBeCoLe20,
     TITLE = {Mini-workshop: {O}ne-sided and two-sided stochastic
              descriptions},
      NOTE = {Abstracts from the mini-workshop held February 23--29, 2020,
              Organized by Noam Berger, Stein Andreas Bethuelsen, Diana
              Conache and Arnaud Le Ny},
   JOURNAL = {Oberwolfach Rep.},
  FJOURNAL = {Oberwolfach Reports},
    VOLUME = {17},
      YEAR = {2020},
    NUMBER = {1},
     PAGES = {601--637},
      ISSN = {1660-8933,1660-8941},
   MRCLASS = {60-06 (60K35 60K40 82-06)},
  MRNUMBER = {4214793},
       DOI = {10.4171/owr/2020/11},
       URL = {https://doi.org/10.4171/owr/2020/11},
}

@article {FeGaGr11,
    AUTHOR = {Fern\'andez, Roberto and Gallo, Sandro and Maillard,
              Gr\'egory},
     TITLE = {Regular {$C$}-measures are not always {G}ibbsian},
   JOURNAL = {Electron. Commun. Probab.},
  FJOURNAL = {Electronic Communications in Probability},
    VOLUME = {16},
      YEAR = {2011},
     PAGES = {732--740},
      ISSN = {1083-589X},
   MRCLASS = {60G10 (37A60)},
  MRNUMBER = {2861437},
MRREVIEWER = {Utkir\ Rozikov},
       DOI = {10.1214/ECP.v16-1681},
       URL = {https://doi.org/10.1214/ECP.v16-1681},
}

@article {EnJaKu22,
    AUTHOR = {Engler, Nils and Jahnel, Benedikt and K\"ulske, Christof},
     TITLE = {Gibbsianness of locally thinned random fields},
   JOURNAL = {Markov Process. Related Fields},
  FJOURNAL = {Markov Processes and Related Fields},
    VOLUME = {28},
      YEAR = {2022},
    NUMBER = {2},
     PAGES = {185--214},
      ISSN = {1024-2953},
   MRCLASS = {60K35 (60D05 82B20)},
  MRNUMBER = {4560692},
}

@article {Si72,
    AUTHOR = {Sinai, Ja.\ G.},
     TITLE = {Gibbs measures in ergodic theory},
   JOURNAL = {Uspehi Mat. Nauk},
  FJOURNAL = {Akademija Nauk SSSR i Moskovskoe Matemati\v ceskoe Ob\v s\v
              cestvo. Uspehi Matemati\v ceskih Nauk},
    VOLUME = {27},
      YEAR = {1972},
    NUMBER = {4(166)},
     PAGES = {21--64},
      ISSN = {0042-1316},
   MRCLASS = {28A65 (58F15)},
  MRNUMBER = {399421},
MRREVIEWER = {L.\ A.\ Bunimovich},
}

@article {BrKa93,
    AUTHOR = {Bramson, Maury and Kalikow, Steven},
     TITLE = {Nonuniqueness in {$g$}-functions},
   JOURNAL = {Israel J. Math.},
  FJOURNAL = {Israel Journal of Mathematics},
    VOLUME = {84},
      YEAR = {1993},
    NUMBER = {1-2},
     PAGES = {153--160},
      ISSN = {0021-2172,1565-8511},
   MRCLASS = {28D05 (60G10)},
  MRNUMBER = {1244665},
MRREVIEWER = {J.\ B.\ Robertson},
       DOI = {10.1007/BF02761697},
       URL = {https://doi.org/10.1007/BF02761697},
}

@article {Hu06,
    AUTHOR = {Hulse, Paul},
     TITLE = {An example of non-unique {$g$}-measures},
   JOURNAL = {Ergodic Theory Dynam. Systems},
  FJOURNAL = {Ergodic Theory and Dynamical Systems},
    VOLUME = {26},
      YEAR = {2006},
    NUMBER = {2},
     PAGES = {439--445},
      ISSN = {0143-3857,1469-4417},
   MRCLASS = {28D05},
  MRNUMBER = {2218769},
MRREVIEWER = {Olivier\ Raimond},
       DOI = {10.1017/S0143385705000489},
       URL = {https://doi.org/10.1017/S0143385705000489},
}

@article {AiChChNe88,
    AUTHOR = {Aizenman, M. and Chayes, J. T. and Chayes, L. and Newman, C.
              M.},
     TITLE = {Discontinuity of the magnetization in one-dimensional
              {$1/|x-y|^2$} {I}sing and {P}otts models},
   JOURNAL = {J. Statist. Phys.},
  FJOURNAL = {Journal of Statistical Physics},
    VOLUME = {50},
      YEAR = {1988},
    NUMBER = {1-2},
     PAGES = {1--40},
      ISSN = {0022-4715,1572-9613},
   MRCLASS = {82A68 (82A25)},
  MRNUMBER = {939480},
MRREVIEWER = {Mary\ Lunn},
       DOI = {10.1007/BF01022985},
       URL = {https://doi.org/10.1007/BF01022985},
}

@misc{Hu24,
      title={Pointwise two-point function estimates and a non-pertubative proof of mean-field critical behaviour for long-range percolation}, 
      author={Tom Hutchcroft},
      year={2024},
      eprint={2404.07276},
      archivePrefix={arXiv},
      primaryClass={math.PR},
      url={https://arxiv.org/abs/2404.07276}, 
}

@book{Br10,
Author={Pierre Br{\'e}maud},
Title={Markov Chains},
Series={Texts in Applied Mathematics},
Volume={31},
Publisher={Springer-Verlag New York},
Year={2010},
Pages={444},
ISBN={978-1-4419-3131-3}
}

@book{Ge11,
    author = {Georgii, Hans-Otto},
     title = {{Gibbs measures and phase transitions}},
    series = {de Gruyter Studies in Mathematics},
    volume = {9},
   edition = {Second},
 publisher = {Walter de Gruyter \& Co., Berlin},
      year = {2011},
     pages = {xiv+545},
       DOI = {10.1515/9783110250329},
  MRCLASS = {82B26 (60K35)},
  MRNUMBER = {2807681},
}

@book{Kl14,
    AUTHOR = {Klenke, Achim},
     TITLE = {Probability theory},
    SERIES = {Universitext},
   EDITION = {Second},
 PUBLISHER = {Springer, London},
      YEAR = {2014},
     PAGES = {xii+638},
       DOI = {10.1007/978-1-4471-5361-0},
}

@book {FV17,
    AUTHOR = {Friedli, S. and Velenik, Y.},
     TITLE = {Statistical mechanics of lattice systems},
      NOTE = {A concrete mathematical introduction},
 PUBLISHER = {Cambridge University Press, Cambridge},
      YEAR = {2018},
     PAGES = {xix+622},
   MRCLASS = {82-01 (82B05)},
  MRNUMBER = {3752129},
  DOI={10.1017/9781316882603},
}

@article{EFS93,
AUTHOR = {van Enter, Aernout C. D. and Fern\'{a}ndez, Roberto and Sokal,
              Alan D.},
     TITLE = {Regularity properties and pathologies of position-space
              renormalization-group transformations: scope and limitations
              of {G}ibbsian theory},
   JOURNAL = {J. Statist. Phys.},
  FJOURNAL = {Journal of Statistical Physics},
    VOLUME = {72},
      YEAR = {1993},
    NUMBER = {5-6},
     PAGES = {879--1167},
 MRCLASS = {82B28 (82-02 82B03 82B20)},
MRNUMBER = {1241537},
MRREVIEWER = {Christof K\"{u}lske},
       DOI = {10.1007/BF01048183}
}

@article{EFHR02,
author    = {van Enter, A.C.D. and Fern\'{a}ndez, R. and den Hollander, F. and Redig, F.},
title     = {{Possible Loss and Recovery of {G}ibbsianness during the Stochastic Evolution of {G}ibbs Measures}},
year      = {2002},
journal   = {Communications in Mathematical Physics},
volume    = {226},
issue     = {1},
publisher = {Springer-Verlag},
pages     = {101-130},
MRCLASS = {82C05 (60K35 82C20 82C22)},
MRNUMBER = {1889994},
DOI = {10.1007/s002200200605}
}

@article {JaKu23,
    AUTHOR = {Jahnel, Benedikt and K\"{u}lske, Christof},
     TITLE = {Gibbsianness and non-{G}ibbsianness for {B}ernoulli lattice
              fields under removal of isolated sites},
   JOURNAL = {Bernoulli},
  FJOURNAL = {Bernoulli. Official Journal of the Bernoulli Society for
              Mathematical Statistics and Probability},
    VOLUME = {29},
      YEAR = {2023},
    NUMBER = {4},
     PAGES = {3013--3032},
   MRCLASS = {60D05 (60K35 82B44)},
  MRNUMBER = {4632129},
       DOI = {10.3150/22-bej1572},
}

@article {FeMa06,
    AUTHOR = {Fern\'{a}ndez, Roberto and Maillard, Gr\'{e}gory},
     TITLE = {Construction of a specification from its singleton part},
   JOURNAL = {ALEA Lat. Am. J. Probab. Math. Stat.},
  FJOURNAL = {ALEA. Latin American Journal of Probability and Mathematical
              Statistics},
    VOLUME = {2},
      YEAR = {2006},
     PAGES = {297--315},
   MRCLASS = {60K35 (60G60 82B05)},
  MRNUMBER = {2285734},
MRREVIEWER = {Aernout C. D. van Enter},
URL={https://alea.impa.br/articles/v2/02-13.pdf},
}

@article {HeKuSc23,
    AUTHOR = {Henning, Florian and K\"ulske, Christof and Schubert, Niklas},
     TITLE = {Gibbs properties of the {B}ernoulli field on inhomogeneous
              trees under the removal of isolated sites},
   JOURNAL = {Markov Process. Related Fields},
  FJOURNAL = {Markov Processes and Related Fields},
    VOLUME = {29},
      YEAR = {2023},
    NUMBER = {5},
     PAGES = {641--659},
      ISSN = {1024-2953},
   MRCLASS = {82B26 (60K35 82B20)},
  MRNUMBER = {4717466},
       DOI = {10.61102/1024-2953-mprf.2023.29.5.002},
}

@article {FeMa04,
    AUTHOR = {Fern\'{a}ndez, Roberto and Maillard, Gr\'{e}gory},
     TITLE = {Chains with complete connections and one-dimensional {G}ibbs
              measures},
   JOURNAL = {Electron. J. Probab.},
  FJOURNAL = {Electronic Journal of Probability},
    VOLUME = {9},
      YEAR = {2004},
     PAGES = {no. 6, 145--176},
   MRCLASS = {60G07 (60G10 60G60 60J05 60J10 82B05)},
  MRNUMBER = {2041831},
MRREVIEWER = {M.\ Iosifescu},
       DOI = {10.1214/EJP.v9-149}
}

@article {BeFeVe19,
    AUTHOR = {Berghout, Steven and Fern\'{a}ndez, Roberto and Verbitskiy,
              Evgeny},
     TITLE = {On the relation between {G}ibbs and {$g$}-measures},
   JOURNAL = {Ergodic Theory Dynam. Systems},
  FJOURNAL = {Ergodic Theory and Dynamical Systems},
    VOLUME = {39},
      YEAR = {2019},
    NUMBER = {12},
     PAGES = {3224--3249},
   MRCLASS = {37D35 (60A10 60K35)},
  MRNUMBER = {4027547},
MRREVIEWER = {Utkir\ Rozikov},
       DOI = {10.1017/etds.2018.13},
}

@article {FeMa05,
    AUTHOR = {Fern\'{a}ndez, Roberto and Maillard, Gr\'{e}gory},
     TITLE = {Chains with complete connections: general theory, uniqueness,
              loss of memory and mixing properties},
   JOURNAL = {J. Stat. Phys.},
  FJOURNAL = {Journal of Statistical Physics},
    VOLUME = {118},
      YEAR = {2005},
    NUMBER = {3-4},
     PAGES = {555--588},
      ISSN = {0022-4715,1572-9613},
   MRCLASS = {82B05 (60K35)},
  MRNUMBER = {2123648},
MRREVIEWER = {M.\ Iosifescu},
       DOI = {10.1007/s10955-004-8821-5},
       URL = {https://doi.org/10.1007/s10955-004-8821-5},
}

@article {ChRe09,
    AUTHOR = {Chazottes, Jean-Ren\'{e} and Redig, Frank},
     TITLE = {Concentration inequalities for {M}arkov processes via
              coupling},
   JOURNAL = {Electron. J. Probab.},
  FJOURNAL = {Electronic Journal of Probability},
    VOLUME = {14},
      YEAR = {2009},
     PAGES = {no. 40, 1162--1180},
   MRCLASS = {60E15 (60J05 60J25)},
  MRNUMBER = {2511280},
MRREVIEWER = {Antonio\ Di Crescenzo},
       DOI = {10.1214/ejp.v14-657},
}

@article {Gr79,
    AUTHOR = {Griffiths, Robert B. and Pearce, Paul A.},
     TITLE = {Mathematical properties of position-space
              renormalization-group transformations},
   JOURNAL = {J. Statist. Phys.},
  FJOURNAL = {Journal of Statistical Physics},
    VOLUME = {20},
      YEAR = {1979},
    NUMBER = {5},
     PAGES = {499--545},
      ISSN = {0022-4715,1572-9613},
   MRCLASS = {82A05 (81E15 82A67)},
  MRNUMBER = {533527},
MRREVIEWER = {Michel\ Romerio},
       DOI = {10.1007/BF01012897},
       URL = {https://doi.org/10.1007/BF01012897},
}

@article {Sc89,
    AUTHOR = {Schonmann, Roberto H.},
     TITLE = {Projections of {G}ibbs measures may be non-{G}ibbsian},
   JOURNAL = {Comm. Math. Phys.},
  FJOURNAL = {Communications in Mathematical Physics},
    VOLUME = {124},
      YEAR = {1989},
    NUMBER = {1},
     PAGES = {1--7},
      ISSN = {0010-3616,1432-0916},
   MRCLASS = {82A05 (60F10 60K35 82A35 82A68)},
  MRNUMBER = {1012855},
MRREVIEWER = {Christian\ Maes},
       URL = {http://projecteuclid.org/euclid.cmp/1104179072},
}

@article {MaReShMo00,
    AUTHOR = {Maes, Christian and Redig, Frank and Shlosman, Senya and Van
              Moffaert, Annelies},
     TITLE = {Percolation, path large deviations and weakly {G}ibbs states},
   JOURNAL = {Comm. Math. Phys.},
  FJOURNAL = {Communications in Mathematical Physics},
    VOLUME = {209},
      YEAR = {2000},
    NUMBER = {2},
     PAGES = {517--545},
      ISSN = {0010-3616,1432-0916},
   MRCLASS = {82B43 (60K35 82B28)},
  MRNUMBER = {1737993},
       DOI = {10.1007/s002200050029},
       URL = {https://doi.org/10.1007/s002200050029},
}

@article{KN07,
title     = {{Spin-Flip Dynamics of the {C}urie-{W}eiss Model: Loss of {G}ibbsianness with Possibly Broken Symmetry}},
author    = {K\"{u}lske, Christof and Le Ny, Arnaud},
year      = {2007},
issn      = {0010-3616},
journal   = {Communications in Mathematical Physics},
volume    = {271},
number    = {2},
doi       = {10.1007/s00220-007-0201-y},
publisher = {Springer-Verlag},
pages     = {431-454},
}

@article {KuLeRe04,
    AUTHOR = {K\"ulske, Christof and Le Ny, Arnaud and Redig, Frank},
     TITLE = {Relative entropy and variational properties of generalized
              {G}ibbsian measures},
   JOURNAL = {Ann. Probab.},
  FJOURNAL = {The Annals of Probability},
    VOLUME = {32},
      YEAR = {2004},
    NUMBER = {2},
     PAGES = {1691--1726},
      ISSN = {0091-1798,2168-894X},
   MRCLASS = {60G60 (82B20 82B44)},
  MRNUMBER = {2060315},
MRREVIEWER = {Christian\ Maes},
       DOI = {10.1214/009117904000000342},
       URL = {https://doi.org/10.1214/009117904000000342},
}

@misc{JaKuZa23,
      title={Locality properties for discrete and continuum Widom--Rowlinson models in random environments}, 
      author={Benedikt Jahnel and Christof Külske and Alexander Zass},
      year={2023},
      eprint={2311.07146},
      archivePrefix={arXiv},
      primaryClass={math.PR},
      url={https://arxiv.org/abs/2311.07146}, 
}

@article {AfBiEnHa25,
    AUTHOR = {Affonso, Lucas and Bissacot, Rodrigo and Endo, Eric O. and
              Handa, Satoshi},
     TITLE = {Long-range {I}sing models: contours, phase transitions and
              decaying fields},
   JOURNAL = {J. Eur. Math. Soc. (JEMS)},
  FJOURNAL = {Journal of the European Mathematical Society (JEMS)},
    VOLUME = {27},
      YEAR = {2025},
    NUMBER = {4},
     PAGES = {1679--1714},
      ISSN = {1435-9855,1435-9863},
   MRCLASS = {82B05 (82B20 82B26)},
  MRNUMBER = {4875614},
       DOI = {10.4171/jems/1529},
       URL = {https://doi.org/10.4171/jems/1529},
}

@article {OkOh25,
    AUTHOR = {Okuyama, Manaka and Ohzeki, Masayuki},
     TITLE = {Existence of long-range order in random-field {I}sing model on
              {D}yson hierarchical lattice},
   JOURNAL = {J. Stat. Phys.},
  FJOURNAL = {Journal of Statistical Physics},
    VOLUME = {192},
      YEAR = {2025},
    NUMBER = {2},
     PAGES = {Paper No. 16, 11},
      ISSN = {0022-4715,1572-9613},
   MRCLASS = {82B44 (60G60 60K35)},
  MRNUMBER = {4853578},
       DOI = {10.1007/s10955-025-03399-9},
       URL = {https://doi.org/10.1007/s10955-025-03399-9},
}

@article {Dy69,
    AUTHOR = {Dyson, Freeman J.},
     TITLE = {Existence of a phase-transition in a one-dimensional {I}sing
              ferromagnet},
   JOURNAL = {Comm. Math. Phys.},
  FJOURNAL = {Communications in Mathematical Physics},
    VOLUME = {12},
      YEAR = {1969},
    NUMBER = {2},
     PAGES = {91--107},
      ISSN = {0010-3616,1432-0916},
   MRCLASS = {82.47},
  MRNUMBER = {436850},
MRREVIEWER = {Nicolae\ Angelescu},
       URL = {http://projecteuclid.org/euclid.cmp/1103841344},
}

@article {BiEnvaEn18,
    AUTHOR = {Bissacot, Rodrigo and Endo, Eric O. and van Enter, Aernout C.
              D. and Kimura, Bruno and Ruszel, Wioletta M.},
     TITLE = {Contour methods for long-range {I}sing models: weakening
              nearest-neighbor interactions and adding decaying fields},
   JOURNAL = {Ann. Henri Poincar\'e},
  FJOURNAL = {Annales Henri Poincar\'e. A Journal of Theoretical and
              Mathematical Physics},
    VOLUME = {19},
      YEAR = {2018},
    NUMBER = {8},
     PAGES = {2557--2574},
      ISSN = {1424-0637,1424-0661},
   MRCLASS = {82B20 (81V19)},
  MRNUMBER = {3830223},
       DOI = {10.1007/s00023-018-0693-3},
       URL = {https://doi.org/10.1007/s00023-018-0693-3},
}

@article {FrSp82,
    AUTHOR = {Fr\"ohlich, J\"urg and Spencer, Thomas},
     TITLE = {The phase transition in the one-dimensional {I}sing model with
              {$1/r\sp{2}$}\ interaction energy},
   JOURNAL = {Comm. Math. Phys.},
  FJOURNAL = {Communications in Mathematical Physics},
    VOLUME = {84},
      YEAR = {1982},
    NUMBER = {1},
     PAGES = {87--101},
      ISSN = {0010-3616,1432-0916},
   MRCLASS = {82A25 (82A05)},
  MRNUMBER = {660541},
MRREVIEWER = {Keiichi\ R.\ Ito},
       URL = {http://projecteuclid.org/euclid.cmp/1103921047},
}

@article {CaFePaMe05,
    AUTHOR = {Cassandro, Marzio and Ferrari, Pablo Augusto and Merola,
              Immacolata and Presutti, Errico},
     TITLE = {Geometry of contours and {P}eierls estimates in {$d=1$}
              {I}sing models with long range interactions},
   JOURNAL = {J. Math. Phys.},
  FJOURNAL = {Journal of Mathematical Physics},
    VOLUME = {46},
      YEAR = {2005},
    NUMBER = {5},
     PAGES = {053305, 22},
      ISSN = {0022-2488,1089-7658},
   MRCLASS = {82B20 (82B26)},
  MRNUMBER = {2143008},
MRREVIEWER = {Domingos\ H. U. Marchetti},
       DOI = {10.1063/1.1897644},
       URL = {https://doi.org/10.1063/1.1897644},
}

@article {Pa88,
    AUTHOR = {Park, Yong Moon},
     TITLE = {Extension of {P}irogov-{S}inai theory of phase transitions to
              infinite range interactions. {I}. {C}luster expansion},
   JOURNAL = {Comm. Math. Phys.},
  FJOURNAL = {Communications in Mathematical Physics},
    VOLUME = {114},
      YEAR = {1988},
    NUMBER = {2},
     PAGES = {187--218},
      ISSN = {0010-3616,1432-0916},
   MRCLASS = {82A05 (82A25 82A68)},
  MRNUMBER = {928225},
MRREVIEWER = {Milos\ Zahradnik},
       URL = {http://projecteuclid.org/euclid.cmp/1104160581},
}

@article {Pa288,
    AUTHOR = {Park, Yong Moon},
     TITLE = {Extension of {P}irogov-{S}inai theory of phase transitions to
              infinite range interactions. {II}. {P}hase diagram},
   JOURNAL = {Comm. Math. Phys.},
  FJOURNAL = {Communications in Mathematical Physics},
    VOLUME = {114},
      YEAR = {1988},
    NUMBER = {2},
     PAGES = {219--241},
      ISSN = {0010-3616,1432-0916},
   MRCLASS = {82A05 (82A25 82A68)},
  MRNUMBER = {928226},
MRREVIEWER = {Milos\ Zahradnik},
       URL = {http://projecteuclid.org/euclid.cmp/1104160582},
}

@article {GiGrRu66,
    AUTHOR = {Ginibre, J. and Grossmann, A. and Ruelle, D.},
     TITLE = {Condensation of lattice gases},
   JOURNAL = {Comm. Math. Phys.},
  FJOURNAL = {Communications in Mathematical Physics},
    VOLUME = {3},
      YEAR = {1966},
    NUMBER = {3},
     PAGES = {187--193},
      ISSN = {0010-3616,1432-0916},
   MRCLASS = {99-04},
  MRNUMBER = {1552496},
       DOI = {10.1007/BF01645411},
       URL = {https://doi.org/10.1007/BF01645411},
}

@article {BBDMSW21,
    AUTHOR = {Betken, Annika and Buchsteiner, Jannis and Dehling, Herold and
              M\"unker, Ines and Schnurr, Alexander and Woerner, Jeannette
              H. C.},
     TITLE = {Ordinal patterns in long-range dependent time series},
   JOURNAL = {Scand. J. Stat.},
  FJOURNAL = {Scandinavian Journal of Statistics. Theory and Applications},
    VOLUME = {48},
      YEAR = {2021},
    NUMBER = {3},
     PAGES = {969--1000},
      ISSN = {0303-6898,1467-9469},
   MRCLASS = {62M10 (60G10 60G70 62E20)},
  MRNUMBER = {4303565},
       DOI = {10.1111/sjos.12478},
       URL = {https://doi.org/10.1111/sjos.12478},
}

\end{document}